%% file: main.tex
\documentclass[a4paper,11pt]{amsart}

\usepackage{Baker}

\title{A generalisation of the toric resolution of curves}
\author{Simone Muselli}
\address{University of Bristol, Bristol, UK.}

\begin{document}

\maketitle
\begin{abstract}
    Let $k$ be a perfect field and let $C_0:f=0$ be a smooth curve in the torus $\G_{m,k}^2$. Let $\T_\Delta$ be the toric variety associated to the Newton polygon of $f$. Extending the toric resolution of $C_0$ on $\T_\Delta$, we construct an explicit model over $k$ of the smooth completion of $C_0$. Such a model exists for any smooth projective curve and can be described via a combinatorial algorithm using an iterative construction of Newton polygons. 
\end{abstract}

\section{Introduction}
Let $U$ be any smooth affine curve defined over a perfect field $k$. Up to isomorphism there exists a unique smooth projective curve $C/k$ birational to $U$, called the \textit{smooth completion} of $U$. 
In this paper we study the problem of finding explicit \textit{models of $C$ over $k$}, i.e\ curves $\tilde C$ isomorphic to $C$ over $k$. More precisely, we present an algorithm to construct a model over $k$ of smooth projective curves which are birational to a smooth curve $C_0\subset\G_{m,k}^2$.
In fact, every smooth projective curve is the smooth completion of a curve $C_0$ as above (Corollary \ref{cor:smoothopenmodel}). Note that a curve is not required to be connected in this work (see conventions and notations in \S\ref{subsec:OutlineandNotation}).


\subsection{Overview}\label{subsec:Overview}
When it exists, a \textit{Baker's model} of a smooth projective curve $C/k$ is an explicit model of $C$ over $k$. It is constructed via a toric resolution of a smooth curve $C_0\subset \G_{m,k}^2$, birational to $C$. 
A Baker's model helps in studying the geometry of $C$. For example, it gives combinatorial interpretations of the genus, the gonality, the Clifford index and the Clifford degree \cite{CC}. 
Let us give a brief description of this model.

For any $C$ and $C_0$ as above,
let $f=\sum_{(i,j)\in\Z^2}c_{ij}\,x^iy^j\in k[x^{\pm 1},y^{\pm 1}]$ be a Laurent polynomial defining 
$C_0:f=0$ in the torus $\G_{m,k}^2$. Let $\Delta$ be the Newton polygon of $f$. A classical construction associates a $2$-dimensional toric variety $\T_\Delta$ to the integral polytope $\Delta$. The Zariski closure $C_1$ of $C_0$ in $\T_\Delta$ is called the \textit{completion of $C_0$ with respect to its Newton polygon}. It is an easy-to-describe projective curve, whose $C_0$ is a dense open. 
The construction of $C_1$ from $C_0$ is said \textit{toric resolution} on $\T_\Delta$. 
If $C_1$ is regular, it is isomorphic to 
$C$ and
is said a Baker's model of $C$.
A smooth projective curve does not always admit a Baker's model (see Appendix \ref{appendix:ExistenceBakerModel}). 
Its existence is closely related to another interesting property: the nondegeneracy. 

For any face $\lambda$ of $\Delta$ (of any dimension) let $f_\lambda=\sum_{(i,j)\in\Z^2\cap\lambda}c_{ij}\,x^iy^j$. The Laurent polynomial $f$ is nondegenerate if for every face $\lambda$ of $\Delta$ the system of equations $f_\lambda=x\frac{\partial f_\lambda}{\partial x}=y\frac{\partial f_\lambda}{\partial y}=0$ has no solutions in $(\bar k^\times)^2$. 
The nondegeneracy of $f$ has a geometric interpretation in terms of $C_1$. From the explicit description of $C_1$, there is a canonical way to endow the subset $C_1\setminus C_0$ with a structure of closed subscheme. We say $C_1$ is \textit{outer regular} if $C_1\setminus C_0$ is smooth. One can prove that $f$ is nondegenerate if and only if $C_1$ is outer regular. This is a sufficient condition for the regularity of $C_1$.

A smooth projective curve $C$ is said nondegenerate if it admits an outer regular Baker's model.
Nondegenerate curves have several applications. They have turned out to be useful in singular theory \cite{Kou} and in the theory of sparse resultants \cite{GKZ}, as well as for studying specific classes of curves \cite{Mik},\cite{BP},\cite{KWZ}. Over finite fields, nondegenerate curves have also been used in $p$-adic cohomology theory \cite{AS}, in the computation of zeta-functions \cite{CDV} and in the study of the torsion subgroup of their own Jacobians \cite{CST}. Unfortunately, nondegenerate curves are rare, especially for high genera \cite{CV1}. In fact, recall that even a Baker's model may not exist.

Let $C/k$ be any smooth projective curve. In this paper we construct an explicit model $C_n$ of $C$ over $k$, called \textit{generalised Baker's model} (Definition \ref{defn:GeneralisedBakerModel}), extending the classical toric resolution without losing the connection with Newton polygons. \textit{Every} smooth projective curve $C$ has a generalised Baker's model and it can be constructed from \textit{any} smooth curve $C_0\subset \G_{m,k}^2$ birational to $C$. Similarly to the classical case, the subset $C_n\setminus C_0$ will naturally be equipped with a structure of closed subscheme. We say that $C_n$ is \textit{outer regular} if the subscheme $C_n\setminus C_0$ is smooth. Although not all smooth projective curves are nondegenerate, they always have an outer regular generalised Baker's model (Corollary \ref{cor:OuterRegBakerModExistenceGeneral}). Let us describe our approach briefly. 

For any smooth curve $C_0\subset\G_{m,k}^2$, we construct a sequence of proper birational morphisms of curves
\begin{equation}\label{eqn:Thesequence}
\dots\xrightarrow{s_{n+1}} C_{n+1}\xrightarrow{s_n} C_n\xrightarrow{s_{n-1}}\dots\xrightarrow{s_1} C_1,  
\end{equation}
where $C_1$ is the completion of $C_0$ with respect to its Newton polygon. The curves $C_n$ are birational to $C_0$ and explicitly constructed over an algebraic closure $\bar k/k$ via an iterative construction of Newton polygons. We also describe the action of the absolute Galois group $\Gal(\bar k/k)$ on $C_n\times_k\bar k$. Note that since $C_1$ is projective, the curves $C_n$ will be projective as well. If $C_n$ is regular, for some $n$, then it is a model over $k$ of $C$. 
Such $C_n$ is what we call a generalised Baker's model of $C$. Thus the following theorem is a key result of our paper.

\begin{thm}[Theorems \ref{thm:secondmodel}, \ref{thm:Galoismodel}]\label{thm:FirstTheorem}
For a sufficiently large $n$, the curve $C_n$ is outer regular.
\end{thm}
From the explicit construction of an outer regular generalised Baker's model one can also describe the set $C(\bar k)\setminus C_0(\bar k)$. The result that is obtained extends the known one for nondegenerate curves. We will state them in \S\ref{subsec:IntroductionTheorem}, in the case of geometrically connected curves. In the next subsection we discuss one of the main motivations of this work: the study of regular models of curves over discrete valuation rings.

\subsection{Models of curves over discrete valuation rings}
Let $K$ be a complete discretely valued field with ring of integers $O_K$ and residue field $k$. 
Let $C/K$ be a projective curve. A \textit{model of $C$ over $O_K$} is a proper flat scheme $\m C\rightarrow\Spec O_K$ of dimension $2$ such that its generic fibre $\m C_\eta=\m C\times_{O_K} K$ is a model of $C$ over $K$.
The study of regular models over $O_K$ of geometrically connected smooth projective curves $C$ is of great interest in Arithmetic Geometry. The understanding of such models is essential for describing the arithmetic of $C$ and leads to the computation of important objects, such as Tamagawa numbers and integral differentials.

Let $C_0\subset \G_{m,K}^2$ be an affine curve given by $f(x,y)=0$ and let $C_1$ be the completion of $C_0$ with respect to its Newton polygon $\Delta$.
Via a toric resolution approach, \cite{Dok} constructs a model of $C_1$ over $O_K$, denoted $\m C_\Delta$. This is an innovative result, able to construct regular models of curves over discrete valuation rings in cases that were previously hard to tackle (such as the case of curves with wildly potential semistable reduction). However, this approach has two major limits. First, it can construct a model of a smooth projective curve $C$ only if $C$ admits a Baker's model. Second, although we are mainly interested in regular models, $\m C_{\Delta}$ may be singular. Let us discuss more in detail this second aspect.

The scheme $\m C_\Delta$ is given as the Zariski closure of $C_0$ in a toric scheme $X_\Sigma$. The ambient space $X_\Sigma$ is constructed from $\Delta$, taking into account also the valuations of the coefficients of $f$. The connection of $\m C_\Delta$ with toric resolution of curves goes beyond its generic fibre. 
Let $\m C_{\Delta,s}^\mathrm{red}$ be the reduced closed subscheme with the same underlying topological space of the special fibre $\m C_{\Delta,s}$ of $\m C_{\Delta}$. Then $\m C_{\Delta,s}^\mathrm{red}$ can be decomposed in principal components $\bar X_F$ and chains of $\P^1$s. The components $\bar X_F$ are the completions of curves $X_F\subset\G_{m,k}^2$ with respect to their Newton polygons. One can see that if all $\bar X_F$ are outer regular, then $\m C_{\Delta}$ is regular. Thus the fact that not every projective curve has an outer regular Baker's model is the main obstruction for the regularity of $\m C_{\Delta}$.

Therefore the existence of outer regular generalised Baker's models, subject of this paper, has the potential to extend Dokchitser's result to construct regular models of all smooth projective curves. Although such an extension is highly non-trivial, in \cite{Mus} we can already see an implicit application of generalised Baker's model towards that goal. Let us spend a few lines explaining why.
In \cite{Mus} the author constructs a regular model $\m C$ over $O_K$ for a wide class of hyperelliptic curves $C/K$ as follows. Let $C: y^2=h(x)$ be a hyperelliptic curve in this class. One considers smooth curves $C_0^w\subset\G_{m,K}^2$, for $w\in W\subseteq K$, given by $y^2=h(x+w)$ and so birational to $C$. For each $w\in W$, let $\m C_{\Delta^w}$ be the model of $C$ constructed from $C_0^w$ by \cite{Dok}. The regular model $\m C$ is then obtained by glueing regular open subschemes $\mathring{\m C}_{\Delta^w}$ of $\m C_{\Delta^w}$, containing all points of codimension $1$. In particular, for any principal component $\bar X_F$ of $\m C_{\Delta^w,s}^\mathrm{red}$ there exists a closed subscheme $\Gamma_\t$ of $\m C_s=\m C\times_{O_K} k$, birational to $\bar X_F$. The regularity of $\m C$ follows from the fact that $\Gamma_\t$ is an outer regular generalised Baker's model of the smooth completion of $X_F$ (this can be checked by comparing the description of $\Gamma_\t$ in \cite[\S5]{Mus} and the construction in \S\ref{sec:Superelliptic} of an outer regular generalised Baker's model for curves given by superelliptic equations).

\subsection{Outer regular generalised Baker's model}\label{subsec:IntroductionTheorem}
Let $k$ be a perfect field with algebraic closure $\bar k$. Let $f\in k[x^{\pm 1},y^{\pm 1}]$ such that $C_0:f=0$ is a geometrically connected smooth curve over $\G_{m,k}^2$, and let $\Delta$ be the Newton polygon of $f$. If $f$ is nondegenerate, then the completion $C_1$ of $C_0$ with respect to $\Delta$ is outer regular. In particular, $C_1$ is a Baker's model of the smooth completion $C$ of $C_0$. From $C_1$ we can describe the points in $C\setminus C_0$ in an elementary way as follows. 

\begin{defn}\label{defn:normalvector}
For any edge $\ell$ of an integral $2$-dimensional polytope $\m{P}$, consider the unique affine function $\ell^\ast:\Z^2\rightarrow\Z$ given by $\ell^\ast|_\ell=0$, $\ell^\ast|_\m{P}\geq 0$. Write $\ell^\ast(i,j)=ai+bj+c$, for some $a,b,c\in\Z$. Then the primitive vector $(a,b)\in\Z^2$ will be called the \textit{normal vector} of $\ell$.

We also extend this definition to segments $\m{P}$, considered as integral $2$-dimensional polytopes of zero volume. In this case $\m{P}$ has two edges, equal to $\m{P}$ itself, with opposite normal vectors.
\end{defn}

\begin{nt}\label{nt:deltabeta}
For any primitive vector $\beta=(\beta_1,\beta_2)\in\Z^2$ fix $\delta_\beta=(\delta_1,\delta_2)\in\Z^2$ such that $\delta_1\beta_2-\delta_2\beta_1=1$. 
Note that $\delta_\beta$ can be freely chosen, and depends (only) on $\beta$.
\end{nt}

For any edge $\ell$ of $\Delta$:
\begin{enumerate}[label=(\arabic*)]
    \item Consider its normal vector $\beta=(\beta_1,\beta_2)\in\Z^2$ and $\delta_\beta=(\delta_1,\delta_2)\in\Z^2$.
    \item Via the change of variables $x=X^{\delta_1}Y^{\beta_1}$, $y=X^{\delta_2}Y^{\beta_2}$, let $f_{\ell}\in k[X,Y]$ such that $X\nmid f_\ell$, $Y\nmid f_{\ell}$, and
\[f(x,y)=X^{n_X} Y^{n_Y}\cdot f_{\ell}(X,Y),\]
for some $n_X,n_Y\in\Z$.
\end{enumerate}    
Define the curve
$C_\ell:f_\ell(X,Y)=0$ in $\G_{m,k}\times\A_k^1=\Spec k[X^{\pm 1},Y]$. Note that $C_\ell\cap\G_{m,k}^2=C_0$. The completion of $C_0$ with respect to $\Delta$ is 
\[C_1=\bigcup_{\ell\subset\partial\Delta} C_\ell,\]
where the curves $C_\ell$ are glued along their common open subscheme $C_0$.

Let $P_1=\bigsqcup_{\ell\subset\partial\Delta}\{f_\ell\}$, where $\ell$ runs through all edges of $\Delta$. For any $f_\ell  \in P_1$ define $f|_\ell\in k[X]$ by $f|_\ell(X)=f_\ell(X,0)$. 
It is easy to see that $f$ is nondegenerate if and only if $f|_\ell$ has no multiple roots in $\bar k^\times$ for any edge $\ell$ of $\Delta$.
Then from the description of $C_1$ we have the following result.

\begin{thm}[{\cite[Theorem 2.2(3)]{Dok}}]\label{thm:classicalBakerModel}
Suppose $f$ nondegenerate. There is a natural bijection that preserves $\Gal(\bar k/k)$-action,
\[C(\bar k)\setminus C_0(\bar k)\xleftrightarrow{1:1}\bigsqcup_{f_\ell\in P_1}\{\text{(simple) roots of $f|_\ell$ in $\bar k^\times$}\}.\]
\end{thm}

If $f$ is not nondegenerate, or, equivalently, if $C_1$ is not outer regular, we can construct from $C_1$ an outer regular generalised Baker's model $C_n$ of $C$, that always exists. Then the explicit description of $C_n$ can be used to obtain a more general version of Theorem \ref{thm:classicalBakerModel} capable to describe the points in $C\setminus C_0$ unconditionally.

First we are going to define finite indexed sets $P_n$ of polynomials in $\bar k[X,Y]$, for all $n\in\Z_+$. 
A polynomial in $P_n$ will be denoted by $f_\ell$ for an edge $\ell$ of some $2$-dimensional polytope. However, if $n\geq 2$ then $f_\ell\in P_n$ will be indexed not only by $\ell$ but also by a polynomial of $P_{n-1}$ and a non-zero element of $\bar k$. For any $f_\ell\in P_n$, define $f|_\ell\in \bar k[X]$ by $f|_\ell(X)=f_\ell(X,0)$.
Let $P_1$ be as above. For $n\in\Z_+$, we recursively construct the set $P_{n+1}$ from $P_n$ via the following algorithm.

\begin{alg}\label{alg:TheAlgorithm}
For any $f_\ell\in P_n$ and any multiple root $a\in \bar k^\times$ of $f|_\ell$ do:
\begin{enumerate}[label=(\arabic*)]
    \item Rename the variables of $f_\ell$ from $X,Y$ to $x,y$.
    \item Let $f_{\ell,a}\in\bar k[x,y]$ given by $f_{\ell,a}(x,y)=f_\ell(x+a,y)$.
    \item Draw the Newton polygon $\Delta_{\ell,a}$ of $ f_{\ell,a}$.
    \item For any edge $\ell'$ of $\Delta_{\ell,a}$ with normal vector $\beta=(\beta_1,\beta_2)\in\Z_+^2$, consider $\delta_\beta=(\delta_1,\delta_2)\in\Z^2$, previously fixed.
    \item Through the change of variables $x=X^{\delta_1}Y^{\beta_1}$, $y=X^{\delta_2}Y^{\beta_2}$, let $f_{\ell'}=(f_{\ell,a})_{\ell'}\in \bar k[X,Y]$ such that $X\nmid f_{\ell'}$, $Y\nmid f_{\ell'}$, and
    \[f_{\ell,a}(x,y)=X^{n_X} Y^{n_Y}\cdot f_{\ell'}(X,Y),\]
    for some $n_X,n_Y\in\Z$.
    \item Define $P_{\ell,a}=\bigsqcup_{\ell'\subset\partial\Delta_{\ell,a}}\{f_{\ell'}\}$, where $\ell'$ runs through all edges of $\Delta_{\ell,a}$ with normal vector in $\Z_+^2$.
\end{enumerate}
Then \[P_{n+1}:=\bigsqcup_{f_\ell,a}P_{\ell,a},\]
where $f_\ell$ runs through all polynomials in $P_n$ and $a$ runs through all multiple roots of $f|_\ell$ in $\bar k^\times$.
\end{alg}
For every $n\in\Z_+$, one can inductively define an action of $\Gal(\bar k/k)$ on $P_n$ with the following property: for any $\sigma\in\Gal(\bar k/k)$ and $f_{\ell}\in P_n$ the polynomials $\sigma\cdot f_{\ell}$ and $f_{\ell}^\sigma$ are equal. Note that this property is not enough to describe the action since $P_n$ is an indexed set. 

Let $\sigma\in\Gal(\bar k/k)$. If $f_{\ell}\in P_1$, then define $\sigma\cdot f_{\ell}=f_{\ell}$. Let $f_{\ell'}\in P_{n+1}$ for $n\in\Z_+$. From Algorithm \ref{alg:TheAlgorithm} it follows that $f_{\ell'}=(f_{\ell,a})_{\ell'}$ for some $f_\ell\in P_n$ and some multiple root $a\in\bar k^\times$ of $f|_\ell$. By inductive hypothesis $\sigma\cdot f_\ell$ is an element $f_{\sigma(\ell)}$ of $P_n$, and $\sigma(a)$ is a multiple root of $f|_{\sigma(\ell)}$. Moreover, $f_{\sigma(\ell),\sigma(a)}=f_{\ell,a}^\sigma$. Hence the Newton polygon $\Delta_{\sigma(\ell),\sigma(a)}$ coincides with $\Delta_{\ell,a}$. In particular, it has an edge $\sigma(\ell')$ with normal vector equal to the one of $\ell'$. Then define \[\sigma\cdot f_{\ell'}:=f_{\sigma(\ell')}=(f_{\sigma(\ell),\sigma(a)})_{\sigma(\ell')}\in P_{n+1}.\]

Iterate Algorithm \ref{alg:TheAlgorithm} until $P_{n+1}=\varnothing$, i.e.\ for all $f_\ell\in P_n$, the polynomials $f|_\ell$ have no multiple roots in $\bar k^\times$. The procedure terminates. Define
\[P=P_1\sqcup\dots\sqcup P_n.\]
Note that the Galois action on $P_i$ for all $1\leq i\leq n$ induces an action on $P$. For any $\sigma\in \Gal(\bar k/k)$ and $f_\ell\in P$, let $f_{\sigma(\ell)}\in P$ be the element $\sigma\cdot f_\ell$.
We can now generalise Theorem \ref{thm:classicalBakerModel}.
\begin{thm}\label{thm:introduction}
There is a natural bijection 
\[C(\bar k)\setminus C_0(\bar k)\xleftrightarrow{1:1}\bigsqcup_{f_{\ell}\in P}\{\text{simple roots of $f|_\ell$ in $\bar k^\times$}\},\]
that preserves $\Gal(\bar k/k)$-action, where $\sigma\in\Gal(\bar k/k)$ takes a simple root $r\in \bar k^\times$ of $f|_\ell$ to the simple root $\sigma(r)\in \bar k^\times$ of $f|_{\sigma(\ell)}$.
\end{thm}

Theorem \ref{thm:introduction} is proved at the end of \S\ref{sec:GaloisBaker}.

\begin{exa}
Let $f=(x^2+1)^2+y-y^3\in\F_3[x^{\pm 1},y^{\pm 1}]$ and let $C_0:f=0$ in $\G_{m,\F_3}^2$. Note that $C_0$ is regular. 
By \cite[Proposition 3.2]{CV2}, the smooth completion $C$ of $C_0$ is not nondegenerate. Hence 
Theorem \ref{thm:classicalBakerModel} cannot be used. 
We want to describe the points in $C\setminus C_0$ via Theorem \ref{thm:introduction}. First compute the set $P$ via Algorithm \ref{alg:TheAlgorithm}. One has $P=P_1\sqcup P_2$, where 
\begin{itemize}
    \item $P_1$ consists of $3$ polynomials $f_{\ell_1}, f_{\ell_2}, f_{\ell_3}$, where $f|_{\ell_1}=(X^2+1)^2$, $f|_{\ell_2}=X^3+X^2-1$, $f|_{\ell_3}=-X+1$, up to some power of $X$; 
    \item $P_2$ consists of $2$ polynomials $f_{\ell_4}, f_{\ell_5}$, satisfying $f_{\ell_5}=f_{\sigma(\ell_4)}$, where $\sigma$ is the Frobenius automorphism; furthermore, $f|_{\ell_4}=f|_{\ell_5}=-X+1$, up to some power of $X$.
\end{itemize}
Thus Theorem \ref{thm:introduction} shows that $C\setminus C_0$ consists of one point coming from $\ell_4,\ell_5$ with residue field $\F_9$, one point coming from $\ell_2$ with residue field $\F_{27}$ and one $\F_3$-rational point coming from $\ell_3$.
\end{exa}

\subsection{Outline of the paper and notation}\label{subsec:OutlineandNotation}
For the most part of the paper we will assume $k=\bar k$. In \S\ref{sec:ToricVarietiesCharts} we define toric varieties $\T_v$ attached to primitive integer-valued vectors $v$. The charts of the curves $C_n$ in the sequence (\ref{eqn:Thesequence}) will be the Zariski closures of dense opens of $C_0$ inside $\T_v$. In \S\ref{sec:Bakerconstruction} we show how to construct the sequence (\ref{eqn:Thesequence}) recursively and explain its connection with Newton polygons. We also prove the properties of the curves and the morphisms in (\ref{eqn:Thesequence}) previously listed in \S\ref{sec:Bakerdescription}.
Section \ref{sec:GenBakersthm} gives the definition of generalised Baker's model and outer regularity over algebraically closed base fields. We prove crucial results of the paper and present some interesting consequences. 
In \S\ref{sec:multipleSing} we see the construction developed in previous sections from a more general point of view. This will be useful to tackle the case of non-algebraically closed base fields, treated in \S\ref{sec:GaloisBaker}.
Finally, \S\ref{sec:Superelliptic} and \S\ref{sec:Example} consist of applications of our construction. In \S\ref{sec:Superelliptic} we discuss the case of superelliptic equations. In \S\ref{sec:Example} we show an explicit and non-trivial example of a generalised Baker's model.

\subsubsection*{Conventions and notations}
\begin{itemize}[leftmargin=3.5ex]
\item Throughout, $k$ will be a perfect field, algebraically closed in \S\ref{sec:ToricVarietiesCharts}-\ref{sec:multipleSing}. 

\item An \textit{algebraic variety $X$ over $k$}, denoted $X/k$, is a scheme of finite type over $\Spec k$. Let $\m{K}_X$ be the sheaf of stalks of meromorphic functions on $X$ (\cite[Definition 7.1.13]{Liu}). We denote by $k(X)$ the set of global sections of $\m{K}_X$, i.e.\ $k(X)=H^0(X,\m{K}_X)$. It will be called \textit{the ring of rational functions} or \textit{function ring} of $X$. It extends the notions of field of rational functions or function field of integral varieties. 

\item Let $X/k$ be an algebraic variety. Since $k$ is perfect, $X$ is smooth if and only if it is regular. In this context we will then use the words \textit{smooth}, \textit{regular}, \textit{non-singular} interchangeably. 
We will denote by $\Reg(X)$ the open subset of regular points of $X$ and by $\Sing(X)$ the closed subset of singular points of $X$.

\item A morphism $X\rightarrow Y$ between two algebraic varieties $X,Y$ defined over $k$ will always be a \textit{morphism of $k$-schemes}, unless otherwise specified.

\item A \textit{birational map} $f:X\-->Y$ between algebraic varieties $X,Y$ over $k$ is a $k$-rational map (\cite[I.7.1.2]{EGA}) that comes from an isomorphism from a dense open $U\subseteq X$ onto a dense open $V\subseteq Y$. If such a map exists, we say that $X$ is birational to $Y$. A \textit{birational morphism} is a morphism which is (a representative of) a birational map (\cite[Definition 7.5.3]{Liu}). 

\item A \textit{curve} is an equidimensional algebraic variety of dimension $1$. 
We will denote by $\G_m$ the affine algebraic group $\G_{m,k}=\Spec k[x_1^{\pm 1},y^{\pm 1}]$ whenever $k$ is algebraically closed. 

\item Given a ring $A$ and an ideal $I$ of $A$ we identify the ideals of $A/I$ with the ideals of $A$ containing $I$. Furthermore, sometimes we refer to an element $a\in A$ as an element of $A/I$ omitting the class symbol.

\item Finally, the set of natural numbers will contain $0$, i.e.\ $\N=\Z_{\geq 0}$. 
\end{itemize}

\subsection*{Acknowledgements} The author would like to thank his supervisor Tim Dokchitser for the very useful conversations, corrections and general advice.

\section{Ambient toric varieties and charts}\label{sec:ToricVarietiesCharts}

Let $k$ be an algebraically closed field, $n\in\Z_+$, $A=k[x_1^{\pm 1},\dots,x_n^{\pm 1}, y^{\pm 1}]$ and $\mathbb{G}_m^{n+1}=\Spec A$. Let $v=(v_1,\dots,v_n,v_{n+1})\in \Z^{n+1}$ be a primitive vector. Define the affine function $\phi_v:\Z^{n+1}\rightarrow \Z$ given by \[\phi_v(i_1,\dots,i_n,j)=v_1i_1+\dots+v_ni_n+v_{n+1}j.\] 

For any $i=(i_1,\dots,i_n,j)\in\Z^{n+1}$, denote by $\mathbf{x}^i$ the monomial $x_1^{i_1}\cdots x_n^{i_n}y^{j}$ of $k[x_1^{\pm 1},\dots,x_n^{\pm 1}, y^{\pm 1}]$. For any monomial $\mathbf{x}^i$ define $\ord_{v}(\mathbf{x}^i)=\phi_v(i)$. For $f\in A$, with $f\neq0$, expand
\[f=\sum_{i}c_i\mathbf{x}^i,\quad c_i\in k^\times,\]
and set $\ord_v(f)=\min_{i}\ord_v(\mathbf{x}^i)$.
We have just defined a map $\ord_v:A^\times\rightarrow\Z$, which naturally extends to a valuation $\ord_v: \mathrm{Frac}(A)^\times\rightarrow \Z$.

\begin{defn}
Given a primitive vector $w\in\Z^{n+1}$, we say that a matrix $M\in\SL_{n+1}(\Z)$ is \textit{attached to $w$} if its last row is $w$.
\end{defn}

Fix a matrix $M=(a_{ij})$ attached to $v$. It gives the change of variables
\begin{align*}
    (x_1,\dots,x_n,y)&=(X_1^{a_{11}}\cdots X_n^{a_{n1}}Y^{v_1},\dots,X_1^{a_{1(n+1)}}\cdots X_m^{a_{n(n+1)}}Y^{v_{n+1}})\\&=(X_1,\dots,X_n,Y)\bullet M,\\
    (X_1,\dots,X_n,Y)&=(x_1,\dots,x_n,y)\bullet M^{-1}.
\end{align*}
For any $f\in A^\times$, denoting by $\m F\in k[X_1^{\pm 1},\dots,X_n^{\pm 1},Y^{\pm 1}]^\times$ the Laurent polynomial given by
\[\m F(X_1,\dots,X_n,Y)=f((X_1,\dots,X_n,Y)\bullet M),\]
note that $\ord_v(f)=\ord_Y(\m F)$. We get an embedding 
\[A\stackrel{M}{\simeq}k[X_1^{\pm 1},\dots,X_n^{\pm 1},Y^{\pm 1}]\hookleftarrow k[X_1^{\pm 1},\dots,X_n^{\pm 1},Y]=:R,\] from which we define the affine toric variety $\T_v=\Spec R\hookleftarrow \G_{m}^{n+1}$. Since $v$ is the last row of $M$, the toric variety $\T_v$ only depends on $v$ up to isomorphisms that reduce to the identity on $\G_m^{n+1}$. 
Furthermore, up to isomorphism, the closed subvariety $\bar \T_v=\Spec R/(Y)\simeq\G_m^n$ of $\T_v$ only depends on $v$ as well.

Now let $I$ be an ideal of $A$ defining a curve $C_{0,I}=\Spec A/I$ in $\G_m^{n+1}$. 
We denote by $C_{v,I}$ the Zariski closure of $C_{0,I}$ in $\T_v$. Then $C_{v,I}$ is determined by $v$ and $I$, up to isomorphisms that preserve $C_{0,I}$. Recall that $C_{v,I}=\Spec R/\m{I}$, where $\m{I}$ is the inverse image of $I$ under the embedding $R\hookrightarrow A$ above. Suppose $\m{J}\subset R$ is an ideal such that $A/I \simeq R[Y^{-1}]/\m{J}R[Y^{-1}]$ via $M$. Then $\m{J}$ defines $C_{v,I}$ if and only if it equals its saturation with respect to $Y$, i.e.\ $\m{J}=Y^\infty:\m{J}$, or, equivalently, if the image of $Y$ in $R/\m{J}$ is a regular element.

Finally, let $f=f_1\in k[x_1^{\pm 1},y^{\pm1}]$ defining a smooth curve $C_0:f=0$ in $\G_m^2=\Spec k[x_1^{\pm 1},y^{\pm1}]$. For all $i=2,\dots,n$, let $g_i\in k[x_1^{\pm 1},y^{\pm 1}]$ and denote $f_i=x_i-g_i$. Then
\[\frac{k[x_1^{\pm 1},y^{\pm1}]_{g_2\cdots g_n}}{(f)}\simeq\frac{k[x_1^{\pm 1},x_2^{\pm1},\dots,x_n^{\pm 1},y^{\pm1}]}{(f_1,f_2,\dots,f_n)}.\]
Let $\tuplent$ be the tuple $(g_2,\dots,g_n)$ and $I$ the ideal $(f_1,\dots,f_n)$. Define $C_{0,\tuplent}=\Spec\frac{k[x_1^{\pm 1},y^{\pm1}]_{g_2\cdots g_n}}{(f)}$, an affine open of $C_0$. Then $\tuplent$ gives an open immersion $C_{0,\tuplent}\hookrightarrow \G_m^{n+1}$ with image $C_{0,I}$. Let $v\in\Z^{n+1}$ be a primitive vector. Denote by $C_{v,\tuplent}$ the curve $C_{v,I}$ (closure of $C_{0,I}$ inside $\T_v$). We will often identify $C_{0,\tuplent}$ with the dense open image of the immersion $C_{0,\tuplent}\simeq C_{0,I}\hookrightarrow C_{v,\tuplent}$.

Let $C_0$ as above.
For any $m\in\Z_+$ define \[\Omega_m=\{(v,\tuplent)\mid v\in\Z^{m+1}\text{ is a primitive vector and }\tuplent\in k[x_1^{\pm 1},y^{\pm 1}]^{m-1}\}.\]
If $\alpha=(v,\tuplent)\in\Omega_m$ for some $m\in\Z_+$, denote by $C_{0,\alpha}$, $C_\alpha$, respectively the curves $C_{0,\tuplent}$, $C_{v,\tuplent}$ introduced in the previous section. Furthermore, we set $C_\alpha=C_{0,\alpha}=C_0$ when $\alpha=0$. Define \[\Omega=\big\{\alpha\in\textstyle{\bigsqcup_{m\in\Z_+}}\Omega_m\mid C_{0,\alpha}\text{ is dense in }C_0\big\}.\]
If $\alpha=(v,\tuplent)\in\Omega$, denote by $\bar C_\alpha$ the scheme-theoretic intersection of $C_\alpha$ and $\bar \T_v$ in $\T_v$. Note that, up to isomorphism, $\bar C_\alpha$ only depends on $\alpha$.

From the open immersions with dense images $C_{0,\alpha}\hookrightarrow C_{\alpha}$, $C_{0,\alpha}\hookrightarrow C_0$, we have natural birational maps $s_{\alpha\alpha'}:C_\alpha\--> C_{\alpha'}$, for all $\alpha,\alpha'\in \Omega\sqcup\{0\}$. Denote by $U_{\alpha\alpha'}$ the largest (dense) open of $C_{\alpha}$ such that $s_{\alpha\alpha'}$ comes from an open immersion $U_{\alpha\alpha'}\hookrightarrow C_{\alpha'}$. Note that $C_{0,\alpha}\cap C_{0,\alpha'}$ embeds in $U_{\alpha\alpha'}$ via the canonical open immersion $C_{0,\alpha}\hookrightarrow C_\alpha$.

\begin{defn}
Let $m\in\Z_+$ and $c\in\Z$.
Let $v=(v_1,\dots,v_m,v_{m+1})\in\Z^{m+1}$ and $\beta=(\beta_1,\beta_2)\in\Z^2$ be primitive vectors. Define the primitive vector
\[\beta\circ_c v=(v_1\beta_2,v_2\beta_2,\cdots,v_m\beta_2,\beta_1+c\beta_2,v_{m+1}\beta_2)\in\Z^{m+1}.\]

If $g\in k[x_1^{\pm 1},\dots,x_{m}^{\pm 1},y^{\pm 1}]$, define $\beta\circ_g v=\beta\circ_{\ord_v(g)}v$.
\end{defn}

\begin{defn}\label{defn:circ}
Let $m\in\Z_+$ and $\alpha\in\Omega_m$. Write $\alpha=(v,\tuplent)$ where $\tuplent=(g_2,\dots,g_m)$. Fix $g\in k[x_1,\dots,x_m,y]$ and let $g_{m+1}\in k[x_1^{\pm 1},y^{\pm 1}]$ be the unique Laurent polynomial such that $g_{m+1}\equiv g\mod(f_2,\dots,f_m)$, where $f_i=x_i-g_i$. For any primitive vector $\beta\in\N\times\Z_+$, define \[\beta\circ_{g}\alpha=(\beta\circ_{g} v,(g_2,\dots,g_m,g_{m+1}))\in\Omega_{m+1}.\]
\end{defn}

Note that for any $\alpha,\alpha'\in\Omega_m$, polynomials $g,g'\in  k[x_1,\dots,x_m,y]$, and primitive vectors $\beta,\beta'\in\N\times\Z_+$, if $\beta\circ_{g}\alpha=\beta'\circ_{g'}\alpha'$, then $\alpha=\alpha'$.

\begin{defn}\label{defn:Omegaposet}
Let $m\in\Z_+$. Given $\alpha\in\Omega_m$ and $\gamma\in\Omega_{m+1}$, we will write $\alpha<\gamma$ if there exists a polynomial $g\in k[x_1,\dots,x_m,y]$ and a primitive vector $\beta\in\N\times\Z_+$ such that $\gamma=\beta\circ_g\alpha$. Endow $\Omega$ with a structure of partially ordered set by extending $<$ by transitivity.
\end{defn}

\section{Baker's resolution}\label{sec:Bakerdescription}

Let $k$ be an algebraically closed field and let $f\in k[x_1^{\pm 1},y^{\pm 1}]$ be a Laurent polynomial defining a smooth curve $C_0:f=0$ over $\G_m^2$. We will construct a sequence 
\[\dots \xrightarrow{s_{n+1}} C_{n+1}\xrightarrow{s_n} C_n \xrightarrow{s_{n-1}}\dots\xrightarrow{s_1} C_1\]
of proper birational morphisms of projective curves over $k$, birational to $C_0$. Such a sequence will be called a \textit{Baker's resolution of $C_0$} (Definition \ref{defn:BakerResolution}). 
Each curve $C_n$ will be explicitly described and inductively constructed via Newton polygons. In particular, the curve $C_1$ is the completion of $C_0$ with respect to the Newton polygon $\Delta$ of $f$. In \S\ref{sec:GenBakersthm} we will show how to use Baker's resolution to desingularise $C_1$, by finding a regular curve $C_n$, model over $k$ of the smooth completion of $C_0$.

For any $n\in\Z_+$, we aim to construct the projective curve $C_n$ as follows:
\begin{con}\label{construction}
We will define a finite subset $\Sigma_n\subset\Omega$. Then 
\[C_n:=\textstyle\bigcup_{\alpha\in\Sigma_n} C_\alpha\cup C_0,\] 
where the glueing morphisms are given by the birational maps $s_{\alpha\alpha'}$, for $\alpha,\alpha'\in\Sigma_n\sqcup\{0\}$. More precisely, the chart $C_\alpha$ is glued with $C_{\alpha'}$ along $U_{\alpha\alpha'}$ via the isomorphism $U_{\alpha\alpha'}\xrightarrow{\sim} U_{\alpha'\alpha}$ induced by $s_{\alpha\alpha'}$.
In fact, for our choice of $\Sigma_n$ the opens $U_{\alpha\alpha'}$ will be as small as possible, i.e.\ $C_\alpha\cap C_{\alpha'}=C_{0,\alpha}\cap C_{0,\alpha'}$, for any $\alpha,\alpha'\in\Sigma_n\sqcup\{0\}$, $\alpha\neq\alpha'$

Furthermore, for any $\alpha=(v,\tuplent)\in\Sigma_n$, we construct:
\begin{enumerate}[label=(\alph*)]
    \item An ideal $\got a_\alpha=(\m{F}_2,\dots\m{F}_m)\subset k[X_1^{\pm 1},\dots,X_m^{\pm 1},Y]$, 
    and a matrix $M_\alpha\in \SL_{m+1}(\Z)$ attached to $v$ defining an isomorphism
    \[\frac{k[x_1^{\pm 1},\dots, x_m^{\pm 1},y^{\pm 1}]}{(f_2,\dots,f_m)}\stackrel{M_\alpha}{\simeq} \frac{k[X_1^{\pm 1},\dots,X_m^{\pm 1},Y^{\pm 1}]}{(\m{F}_2,\dots,\m{F}_m)},\]
    where $f_i=x_i-g_i$ and $\tuplent=(g_2,\dots,g_m)$. 
    \label{eqn:Bakerprop1}
    \item \label{eqn:Bakerprop2} A positive integer $j_\alpha\leq m$ such that there is an embedding
    \[R_\alpha:=\frac{ k[X_1^{\pm 1},\dots, X_m^{\pm 1},Y]}{(\m{F}_2,\dots,\m{F}_m)}\hookrightarrow k(X_{j_\alpha},Y),\]
    taking $X_{j_{\alpha}}\mapsto X_{j_{\alpha}}$ and $Y\mapsto Y$.
    Moreover, $Y$ is not invertible in $R_\alpha$. 
    \item \label{eqn:Bakerprop3} A polynomial $\m{F}_\alpha\in k[X_{j_\alpha},Y]$, not divisible by $Y$, such that 
    \[C_\alpha=\Spec\frac{k[X_1^{\pm 1},\dots, X_m^{\pm 1},Y]}{(\m{F}_\alpha,\m F_2,\dots,\m{F}_m)}.\]
\end{enumerate}
\end{con}

The ideal $\got a_\alpha$ equals its saturation with respect to $Y$ by \ref{eqn:Bakerprop2}. Therefore \ref{eqn:Bakerprop1} implies that $\got a_\alpha$ is uniquely determined by $M_{\alpha}$.

The homomorphism in \ref{eqn:Bakerprop2} induces an injective ring homomorphism 
\[\frac{R_\alpha}{(Y)}=\frac{ k[X_1^{\pm 1},\dots, X_m^{\pm 1},Y]}{(\m{F}_2,\dots,\m{F}_m,Y)}\hookrightarrow k(X_{j_\alpha}),\]
taking $X_{j_\alpha}\mapsto X_{j_{\alpha}}$. Let $D_\alpha$ be its image. Then $D_\alpha$ is a localisation of $k[X_{j_\alpha}]$, isomorphic to $R_\alpha/(Y)$. 
More precisely, if $t_1,\dots,t_{m}$ are the images of $X_1,\dots,X_{m}$ in $k(X_{j_\alpha})$, then $D_\alpha=k[X_{j_\alpha},t_1^{\pm 1},\dots,t_{m-1}^{\pm 1}]$.

Then, from \ref{eqn:Bakerprop3}, there exists a non-zero polynomial $f|_{\alpha}\in k[X_{j_\alpha}]$, given by $f|_{\alpha}(X_{j_\alpha})=\m F_\alpha(X_{j_\alpha},0)$, such that 
\[\frac {k[X_1^{\pm 1},\dots, X_m^{\pm 1},Y]}{(\m{F}_\alpha,\m{F}_2,\dots,\m{F}_m,Y)}\simeq\frac{D_\alpha}{(f|_{\alpha})}.\]
The closed subscheme $\bar C_\alpha$ of $C_\alpha$ will be identified with
$\Spec D_\alpha/(f|_\alpha)$. 
As a set, it is finite and equals $C_\alpha\setminus C_{0,\alpha}$.

Finally, note that the injective homomorphism in \ref{eqn:Bakerprop2} and the description of $C_\alpha$ in \ref{eqn:Bakerprop3} give an open immersion $C_\alpha\hookrightarrow \Spec k[X_{j_\alpha},Y]/(\m F_\alpha)$.

\section{Construction of the sequence}\label{sec:Bakerconstruction}

Let $k$ be an algebraically closed field and let $f\in k[x_1^{\pm 1},y^{\pm 1}]$ be a Laurent polynomial defining a smooth curve $C_0:f=0$ over $\G_m^2$.

\subsection{Completion with respect to Newton polygon}\label{subsec:BakerFirstStep}
In this subsection we give a description of the curve $C_1$, completion of $C_0$ with respect to its Newton polygon, with the properties of \ref{construction}. We will show that $C_\alpha\simeq C_0$ for all but finitely many $\alpha\in\Omega_1\subset\Omega$. Defining $\Sigma_1\subseteq\Omega_1$ as the subset of those exceptional elements, the curve $C_1$ will be the glueing of $C_\alpha$, $\alpha\in\Sigma_1$, along the common open $C_0$.

Let $v=(a,b)\in\Z^2$ be any primitive vector and $\alpha=(v,())\in\Omega_1$. Fix a matrix $M_\alpha=\lb\begin{smallmatrix}c&d\cr a&b\end{smallmatrix}\rb\in\SL_2(\Z)$ attached to $v$ and define $\phi_v:\Z^2\rightarrow\Z$ by $\phi_v(i,j)=ai+bj-\ord_v(f)$.  
Via the change of variables given by $M_\alpha$ we get
\[f((X_1,Y)\bullet M_\alpha)=X_1^\ast Y^{\ord_v(f)}\m{F}_\alpha(X_1,Y),\qquad\text{where }\m F_\alpha\in k[X_1,Y].\]
Then $\ord_Y(\m{F}_\alpha)=0$ and so
$C_\alpha=\Spec k[X_1^{\pm 1},Y]/(\m{F}_\alpha)$.

Note that $C_{0,\alpha}=C_0$. Let $f|_{\alpha}\in k[X_1]$ given by $f|_{\alpha}(X_1)=\m F_\alpha(X_1,0)$. Recall that the scheme $\bar C_\alpha=\Spec k[X_1^{\pm 1}]/(f|_\alpha)$ equals $C_\alpha\setminus C_{0,\alpha}$ as a set. Therefore $C_\alpha\simeq C_0$ if and only if $f|_\alpha$ is invertible in $k[X_1^{\pm 1}]$. Expand $f=\sum_{i,j} c_{ij}x_1^iy^j$. Let $\Delta$ be the Newton polygon of $f$. It follows that
\[f|_\alpha=X_1^\ast\cdot\sum_{(i,j)\in\phi_v^{-1}(0)}c_{ij}X_1^{ci+dj}.\]
Hence $f|_\alpha$ is not invertible in $k[X_1^{\pm 1}]$ if and only if $\phi_v^{-1}(0)\cap\Delta=\text{edge}$.

Then we can explicitly construct $\Sigma_1$ as follows. For every edge $\ell$ of $\Delta$ consider its normal vector $v_\ell\in\Z^2$ (see Definition \ref{defn:normalvector}). Define $\Sigma_1=\{(v_\ell,())\in\Omega_1\mid \ell\text{ edge of }\Delta\}$. 
The next result follows from the computations above.

\begin{prop}\label{prop:reductionfalphabasecase}
Let $v$ be the normal vector of an edge $\ell$ of $\Delta$ and let $\alpha=(v,())\in\Sigma_1$. Let $(i_0,j_0),\dots,(i_l,j_l)$ be the points of $\ell\cap\Z^2$, ordered along $\ell$ counterclockwise with respect to $\Delta$. Then
\[f|_\alpha=X_1^d\cdot\sum_{r=0}^l c_{i_rj_r}X_1^r,\qquad\text{for some $d\in\N$}.\]
\end{prop}
Glueing $C_\alpha$ for any $\alpha\in\Sigma_1$ gives the curve $C_1$. Note that $C_1$ is the Zariski closure of $C_0$ in $\bigcup_{(v,())\in\Sigma_1}\T_v$ (where the toric varieties $\T_v$ are glued along their common open $\G_m^2$). 

\begin{rem}\label{rem:C1completeness}
Consider the toric surface $\T_\Delta$ of $\Delta$. It is a complete algebraic variety. Then $\bigcup_{(v,())\in\Sigma_1}\T_v$ is a (non-proper) subscheme of $\T_\Delta$. Nevertheless the curve $C_1$ is also the Zariski closure of $C_0$ in $\T_\Delta$ (see \cite[Remark 2.6]{Dok}). Thus it is projective.
\end{rem}

\begin{rem}\label{rem:disjointchartsclassical}
Note that for any $\alpha\in\Sigma_1$, the points on $C_\alpha\setminus C_0$ are not visible on any other chart of $C_1$. Indeed for any $\alpha,\alpha'\in\Sigma_1$, where $\alpha\neq\alpha'$, consider the birational map 
\[s_{\alpha\alpha'}:C_{\alpha}=\Spec k[X_1^{\pm 1},Y]/(\m{F}_\alpha)\-->\Spec k[X_1^{\pm 1},Y]/(\m{F}_{\alpha'})=C_{\alpha'}\] 
given by the matrix $M_{\alpha\alpha'}=M_{\alpha}M_{\alpha'}^{-1}$. Since the lower left entry of $M_{\alpha\alpha'}$ is non-zero, the largest open $U_{\alpha\alpha'}$ of $C_{\alpha}$ for which $s_{\alpha\alpha'}$ comes from an open immersion $U_{\alpha\alpha'}\hookrightarrow C_{\alpha'}$ is $U_{\alpha\alpha'}=D(Y)\subset C_\alpha$, i.e. the image of $C_0$ in $C_\alpha$.
Thus $C_{\alpha}\cap C_{\alpha'}= C_0$ for any $\alpha,\alpha'\in\Sigma_1\sqcup\{0\}$, $\alpha\neq\alpha'$.
\end{rem}

\subsection{Inductive construction of the curves}\label{subsec:BakerInductiveStep}
Until the end of the section let $n\in\Z_+$ and suppose we constructed a finite subset $\Sigma_n\subset\Omega$ and a projective curve $C_n$ as in \ref{construction}. In particular, $C_n=\bigcup_{\alpha\in\Sigma_n}C_\alpha\cup C_0$.

\begin{rem}\label{rem:nonsingularbarjacobian}
Let $\alpha\in\Sigma_n$. Recall $C_{0,\alpha}$ is smooth as so is $C_0$. Therefore $\Sing(C_\alpha)\subseteq C_\alpha\setminus C_{0,\alpha}$.
Then, as an easy consequence of the Jacobian criterion, any singular point of $C_\alpha$ is the image of a singular point of $\bar C_\alpha$ under the closed immersion $\bar C_{\alpha}\hookrightarrow C_\alpha$. This fact can be observed by comparing the Jacobian matrices of $C_\alpha$, defined in \ref{construction}\ref{eqn:Bakerprop3}, and $\bar C_\alpha=C_\alpha\cap\{Y=0\}$, at points of $C_\alpha\setminus C_{0,\alpha}=\bar C_\alpha$. In particular, if $C_n$ is singular then $\bar C_\alpha$ is singular for some $\alpha\in\Sigma_n$.
\end{rem}

Let $\alpha\in\Sigma_n$ and fix $\singpointnt_n\subseteq \mathrm{Sing}(\bar C_\alpha)$. Via the immersion $\bar C_\alpha\hookrightarrow C_n$ given by the closed immersion $\bar C_\alpha\hookrightarrow C_\alpha$ and the inclusion $C_\alpha\subseteq C_n$, the points in $\singpointnt_n$ will be identified with their images in $C_n$. In this subsection we will construct a finite subset $\Sigma_{n+1}\subset\Omega$ defining a curve $C_{n+1}$ as indicated in \ref{construction}. Then, in \S\ref{subsec:sninductiveconstruction} we will define a proper birational morphism $s_n:C_{n+1}\rightarrow C_n$ with exceptional locus $s_n^{-1}(\singpointnt_n\cap\Sing(C_n))$.

Let $m\in\Z_+$ such that $\alpha\in\Omega_m$. Write $\alpha=(v,\tuplent)$, where $v\in\Z^{m+1}$ and $\tuplent\in k[x_1^{\pm 1},y^{\pm 1}]^{m-1}$. Let $M_\alpha\in\SL_{m+1}(\Z)$ be the matrix attached to $v$ fixed by \ref{construction}\ref{eqn:Bakerprop1}, defining a change of variables
\[(x_1,\dots,x_m,y)=(X_1,\dots,X_m,\tilde Y)\bullet M_\alpha.\]
Note that we have changed the notation for the variable $Y$ to $\tilde Y$ for avoiding confusion later.
Let $\got a_\alpha=(\tilde{\m{F}}_2,\dots,\tilde{\m{F}}_m)\subset k[X_1^{\pm 1},\dots,X_m^{\pm 1},\tilde Y]$ be the ideal in \ref{construction}\ref{eqn:Bakerprop1} and $\m{F}_\alpha\in k[X_{j_\alpha},\tilde Y]$ be the polynomial in \ref{construction}\ref{eqn:Bakerprop3} so that
\[C_\alpha=\Spec\frac{k[X_1^{\pm1},\dots,X_m^{\pm 1},\tilde Y]}{(\m{F}_\alpha,\tilde{\m{F}}_2\dots,\tilde{\m{F}}_m)}.\]
Denote $A_m=k[X_1^{\pm1},\dots,X_m^{\pm 1}]$.

Fix a point $p\in \singpointnt_n$. Recall $\bar C_\alpha=\Spec D_\alpha/(f|_\alpha)$, where $D_\alpha$ is a (non-trivial) localisation of $k[X_{j_{\alpha}}]$ and $f|_{\alpha}\in k[X_{j_\alpha}]$ is non-zero.
There exists some irreducible $\bar{\m{G}}_p\in D_\alpha$ such that $(\bar{\m{G}}_p)$ is the maximal ideal of $\O_{\bar C_\alpha,p}$. Then $f|_\alpha\in(\bar{\m{G}}_p)^2$. We choose $\bar{\m{G}}_p\in k[X_{j_\alpha}]$ monic of degree $1$.
Consider $p$ as a point of $C_n$. Then $p\in C_{\alpha}\setminus C_{0,\alpha}$. In particular, $p\notin C_0$, since $C_0\cap C_{\alpha}=C_{0,\alpha}$. For any $\tilde{\m{G}}_p \in k[X_{j_\alpha},\tilde Y]$ such that $\tilde{\m{G}}_p\equiv \bar{\m G}_p\mod\tilde Y$,
the ideal $(\tilde{\m{G}},\tilde Y)+\got a_\alpha$ is the maximal ideal of $\O_{C_\alpha,p}$.
We fix a choice of $\tilde{\m{G}}_p$ such that $\tilde{\m{G}}_p-\bar{\m G}_p\in \tilde Y k[\tilde Y]$ and  $\tilde{\m{G}}_p\nmid\m{F}_\alpha$.

\begin{rem}\label{rem:tildeGpchoice}
Note that such a choice of $\tilde{\m{G}}_p$ is always possible. Indeed, if $\deg_{\tilde Y}(\m{F}_\alpha)$ is the degree of $\m{F}_\alpha$ with respect to $\tilde Y$, it suffices to define \[\tilde{\m{G}}_p=\bar{\m{G}}_{p}+\tilde Y^{\deg_{\tilde Y}(\m{F}_\alpha)+1}\]
On the other hand, $\tilde{\m G}_p=\bar{\m G}_p$ is often admissible and better for computations. For instance, if $C_0$ is connected, then we can always choose $\tilde{\m G}_p=\bar{\m G}_p$.
\end{rem}

\begin{lem}\label{lem:D(G)pdense}
Consider the principal open set $U_p=D(\tilde{\m G}_p)$ of $C_\alpha$. Then $U_p$ is dense in $C_\alpha$.
\proof
As a consequence of \ref{construction}, we saw that there is a natural open immersion 
\[C_\alpha\hookrightarrow \Spec k[X_{j_\alpha},\tilde Y]/(\m F_\alpha).\]
Since $\tilde{\m G}_p\in k[X_{j_\alpha},\tilde Y]$, the image of $U_p$ is the open subset $V_p=D(\tilde{\m G}_p)$ of $\Spec k[X_{j_\alpha},\tilde Y]/(\m F_\alpha)$. Note that if $V_p$ is dense, then $U_p$ is dense in $C_\alpha$. In fact, $V_p$ is dense in $\Spec k[X_{j_\alpha},\tilde Y]/(\m F_\alpha)$ since $\tilde{\m G}_p\nmid \m F_\alpha$. 
\endproof
\end{lem}

Write $\tuplent=(g_2,\dots,g_m)$. From \ref{construction}\ref{eqn:Bakerprop1} recall the isomorphism
\[\frac{k[x_1^{\pm 1},x_2^{\pm1},\dots,x_m^{\pm 1},y^{\pm1}]}{(f_2,\dots,f_m)}\stackrel{M_\alpha}{\simeq}\frac{A_m[\tilde Y^{\pm1}]}{\got a_\alpha},\]\normalsize
where $f_i=x_i-g_i$ for all $i=2,\dots,m$.
Let $g_p\in k[x_1,\dots,x_m,y]$ such that
\[x_1^\ast\cdots x_m^\ast y^\ast\cdot g_p(x_1,\dots,x_m,y)=\tilde{\m{G}}_p\big((x_1,\dots,x_m,y)\bullet M_{\alpha}^{-1}\big).\]
We fix a canonical choice of $g_p$ by requiring $\ord_y{g_p}=0$, and $\ord_{x_i}(g_p)=0$ for all $i=1,\dots,m$. 

\begin{defn}\label{defn:polynomialrelated}
We say that $g_p\in k[x_1,\dots,x_m,y]$ is \textit{related to $\tilde{\m{G}}_p$ by $M_\alpha$} if it is defined as above. Note that it is uniquely determined by $\tilde{\m{G}}_p$ and $M_\alpha$.
\end{defn}

Define $\alpha_p=(0,1)\circ_{g_p}\alpha\in\Omega_{m+1}$ (Definition \ref{defn:circ}). 
Fix a choice of a matrix $M_{\alpha_p}\in\SL_{m+2}(\Z)$ attached to $(0,1)\circ_{g_p} v$ such that the change of variables \[(x_1,\dots,x_m,x_{m+1},y)=(X_1,\dots,X_m,\tilde X_{m+1},\tilde Y)\bullet M_{\alpha_p}\]
restricts to the change of variables given by the matrix $M_\alpha$ on the subring $k[x_1^{\pm 1},\dots, x_m^{\pm 1}, y^{\pm 1}]$ of $k[x_1^{\pm 1},\dots, x_{m+1}^{\pm 1}, y^{\pm 1}]$ and gives the equality
\begin{equation}\label{eqn:gpGp}
x_{m+1}-g_p=X_1^{n_1}\cdots X_m^{n_m} \tilde Y^{\ord_v(g_p)}(\tilde X_{m+1}-\tilde{\m{G}}_p),
\end{equation}
for some $n_1,\dots,n_m\in\Z$. In particular,
\begin{equation}\label{eqn:isoMalphap}
    \sfrac{k[x_1^{\pm 1},\dots,x_{m+1}^{\pm 1},y^{\pm 1}]}{(f_2,\dots,f_m, x_{m+1}-g_{m+1})}\stackrel{M_{\alpha_p}}{\simeq} \sfrac{A_m[\tilde X_{m+1}^{\pm 1}, \tilde Y^{\pm 1}]}{\got a_\alpha+(\tilde X_{m+1}-\tilde{\m{G}}_p)}\hookrightarrow k(X_{j_{\alpha}},\tilde Y).
\end{equation}
where $g_{m+1}\in k[x_1^{\pm 1},y^{\pm 1}]$ is the unique polynomial so that $g_{m+1}\equiv g_{p}\mod(f_2,\dots,f_m)$.

\begin{rem}\label{rem:Matrixalphap}
Such $M_{\alpha_p}$ is constructed as follows. Via $M_\alpha$ write
\[g_p=X_1^{n_1}\cdots X_m^{n_m}Y^{\ord_v(g_p)}\cdot\tilde{\m{G}}_p\]
for some $n_1,\dots,n_m\in\Z$. Then
\begin{itemize}
    \item The $(m+1)$-th row of $M_{\alpha_p}$ is the vector $(0,\dots,0,1,0)$;
    \item The $(m+1)$-th column of $M_{\alpha_p}$ is the vector $(n_1,\dots,n_m,1,\ord_{v}(g_p))$;
    \item The submatrix of $M_{\alpha_p}$ obtained by removing the $(m+1)$-th row and the $(m+1)$-th column equals $M_{\alpha}$.
\end{itemize}
This construction is unique. Indeed, the $(m+1)$-th column is fixed by the equality (\ref{eqn:gpGp}), while all other columns are fixed by the fact that $M_{\alpha_p}$ defines the same change of variables of $M_\alpha$ on $k[x_1^{\pm 1},\dots, x_m^{\pm 1}, y^{\pm 1}]$.
\end{rem}

\begin{lem}\label{lem:Calphap}
With the notation above
 \[C_{\alpha_p}=\Spec \frac{k[X_1^{\pm 1},\dots,X_m^{\pm 1},\tilde X_{m+1}^{\pm 1},\tilde Y]}{(\m{F}_\alpha,\tilde{\m{F}}_2,\dots,\tilde {\m{F}}_m, \tilde X_{m+1}-\tilde{\m{G}}_p)}.\] 
Furthermore, $C_{0,\alpha_p}$ is dense in $C_0$, i.e.\ $\alpha_p\in\Omega$, and the birational map $s_{\alpha_p\alpha}$ comes from an open immersion $s_{\alpha_p\alpha}:C_{\alpha_p}\hookrightarrow C_\alpha$ with image $D(\tilde{\m G}_p)\subset C_\alpha$. Finally, $s_{\alpha_p\alpha}$ induces $\bar C_{\alpha_p}\simeq\bar C_\alpha\setminus \{p\}$.
\proof
First note that $C_{0,\alpha_p}\subset C_{0,\alpha}$. Considering $C_{0,\alpha}$ as an open subscheme of $C_{\alpha}$, then $C_{0,\alpha_p}$ equals $D(\tilde{\m G}_p)\cap C_{0,\alpha}\subset C_\alpha$. Then $C_{0,\alpha_p}$ is dense in $C_{0,\alpha}$ by Lemma \ref{lem:D(G)pdense}. It follows that $C_{0,\alpha_p}$ is dense in $C_0$ since so is $C_{0,\alpha}$. In other words, $\alpha_p\in\Omega$.
The ring homomorphism
\[A_{\alpha}:=\tfrac{A_m[\tilde Y]}{(\m{F}_\alpha)+\got a_\alpha}\rightarrow\tfrac{A_m[\tilde X_{m+1}^{\pm 1},\tilde Y]}{(\m{F}_\alpha, \tilde X_{m+1}-\tilde{\m{G}}_p)+\got a_\alpha}=:A_{\alpha_p}\]
is injective by Lemma \ref{lem:D(G)pdense} and induces the birational map $s_{\alpha_p\alpha}$ if $C_{\alpha_p}=\Spec A_{\alpha_p}$ from (\ref{eqn:isoMalphap}). The injectivity implies that $\tilde Y$ is a regular element of $A_{\alpha_p}$ since $\tilde Y$ is a regular element of $A_\alpha$ by definition of $C_\alpha$. This concludes the proof by definition of $C_{\alpha_p}$.
\endproof
\end{lem}

Now consider the lexicographic monomial order $X_{j_\alpha}>\tilde X_{m+1}>\tilde Y$ on $k[X_{j_\alpha},\tilde X_{m+1},\tilde Y]$ and compute the \textit{normal form} $\m{F}_{\alpha,p}$ of $\m{F}_\alpha$ by $\tilde X_{m+1}-\tilde{\m{G}}_p$ with respect to $>$.
In other words, the polynomial $\m{F}_{\alpha,p}$ is the remainder of the complete multivariate division of $\m{F}_\alpha$ by $\tilde X_{m+1}-\tilde{\m{G}}_p$. Note that $\m{F}_{\alpha,p}\in k[\tilde X_{m+1},\tilde Y]$, as $\tilde{\m G}_p-\bar{\m G}_p\in \tilde Y k[\tilde Y]$ and $\bar{\m{G}}_p\in k[X_{j_\alpha}]$ of degree $1$.

Let $\beta\in\Z_+^2$ be any primitive vector. Fix a matrix $M_\beta\in\SL_2(\Z)$ attached to $\beta$. Then $M_\beta$ gives an isomorphism between $k[\tilde X_{m+1}^{\pm 1},\tilde Y^{\pm 1}]$ and $k[X_{m+1}^{\pm 1},Y^{\pm 1}]$ through the change of variables
$(\tilde X_{m+1},\tilde Y)=(X_{m+1},Y)\bullet M_\beta$. 
This transformation lifts to
\begin{equation}\label{eqn:ImMbetaisomorphism}
A_m[\tilde X_{m+1}^{\pm 1},\tilde Y^{\pm 1}]\stackrel{I_m\oplus M_\beta}{\simeq}A_m[X_{m+1}^{\pm1},Y^{\pm 1}],
\end{equation}
where $I_m\in\SL_m(\Z)$ is the identity matrix of size $m$. Since $\beta\in\Z_+^2$, the isomorphism (\ref{eqn:ImMbetaisomorphism}) restricts to a homomorphism 
\begin{equation*}
A_m[\tilde X_{m+1},\tilde Y]\xrightarrow{I_m\oplus M_\beta}A_m[X_{m+1}^{\pm1},Y]
\end{equation*}

Let $\beta=(\beta_1,\beta_2)$ and let $(\delta_1,\delta_2)$ be the first row of $M_\beta$, so $\delta_1\beta_2-\delta_2\beta_1=1$.
Set $A_{m+1}=A_m[X_{m+1}^{\pm 1}]$. Denote by $\m{F}_2,\dots,\m{F}_m, \m G_p\in A_{m+1}[Y]$ the images of $\tilde{\m{F}}_2,\dots,\tilde{\m{F}}_m, \tilde{\m G}_p$ under $I_m\oplus M_\beta$, respectively. Let $\m{F}_{m+1}=X_{m+1}^{\delta_1} Y^{\beta_1}-\m{G}_p$, image of $\tilde X_{m+1}-\tilde{\m G}_p$. Then we get the homomorphism
\begin{equation}\label{eqn:homIxMbeta}
\frac{A_m[\tilde Y]}{\got a_\alpha}\simeq\frac{A_m[\tilde X_{m+1}, \tilde Y]}{\got a_\alpha+(\tilde X_{m+1}-\tilde{\m{G}}_p)}\xrightarrow{I_m\oplus M_\beta}\frac{A_{m+1}[Y]}{(\m F_2,\dots,\m F_{m+1})}.
\end{equation}
Note that since $\beta_2>0$ then \[\m G_p\equiv \bar{\m G}_p\mod Y,\qquad\text{and}\qquad\m{F}_i\equiv \bar{\m F}_i\mod Y\quad\text{for any $i=2,\dots,m$},\] where $\bar{\m F}_i$ is the unique polynomial in $A_m$ such that $\tilde{\m{F}}_i\equiv \bar {\m F}_i\mod \tilde Y$.

Let $\gamma=\beta\circ_{g_p}\alpha\in\Omega_{m+1}$. By definition, $C_{0,\gamma}=C_{0,\alpha_p}$. Therefore $\gamma\in\Omega$ by Lemma \ref{lem:Calphap}. Let $\got a_\gamma$ be the ideal of $A_{m+1}[Y]$ generated by $\m{F}_2,\dots,\m{F}_{m+1}$ and set $M_\gamma=(I_m\oplus M_\beta)\cdot M_{\alpha_p}\in\SL_{m+2}(\Z)$. Note that the matrix $M_\gamma$ is attached to $\beta\circ_{g_p}v$.
Let $\m{F}_{\gamma}\in k[X_{m+1},Y]$, with $\ord_Y(\m{F}_{\gamma})=0$, satisfying
\[
    \m{F}_{\alpha,p}((X_{m+1},Y)\bullet M_\beta)=X_{m+1}^{n_X} Y^{n_Y}\cdot\m{F}_{\gamma}(X_{m+1},Y),
\]
for some $n_X,n_Y\in\Z$.
Note that 
\begin{equation}\label{eqn:Mgammaisomorphism}
    \tfrac{k[x_1^{\pm 1},\dots, x_{m+1}^{\pm 1},y^{\pm 1}]}{(f_2,\dots,f_{m+1})}\stackrel{M_{\alpha_p}}{\simeq}\tfrac{k[X_1^{\pm 1},\dots, X_m^{\pm 1},\tilde X_{m+1}^{\pm 1},\tilde Y^{\pm 1}]}{(\tilde{\m{F}}_2,\dots,\tilde{\m{F}}_m, \tilde X_{m+1}-\tilde{\m{G}}_p)} \stackrel{I_m\oplus M_\beta}{\simeq} \tfrac{k[X_1^{\pm 1},\dots,X_{m+1}^{\pm 1},Y^{\pm 1}]}{(\m{F}_2,\dots,\m{F}_{m+1})},
\end{equation}
where $f_{m+1}=x_{m+1}-g_{m+1}$. In particular, the ideal $\got a_\gamma$ and the matrix $M_\gamma$ satisfy \ref{construction}\ref{eqn:Bakerprop1} for $\gamma$. With $j_\gamma=m+1$ we are now going to show \ref{construction}\ref{eqn:Bakerprop2} for $\gamma$.

Recall from \ref{construction}\ref{eqn:Bakerprop2} there is an injective homomorphism $R_\alpha\hookrightarrow k(X_{j_\alpha},\tilde Y)$ taking $X_{j_\alpha}\mapsto X_{j_\alpha}$ and $\tilde Y\mapsto \tilde Y$. 
Since $\tilde{\m G}_p-\bar{\m G}_p\in \tilde Y k[\tilde Y]$ and $\bar{\m G}_p\in k[X_{j_\alpha}]$ monic of degree $1$, we have
\[\sfrac{A_m[\tilde X_{m+1},\tilde Y]}{\got a_\alpha+(\tilde X_{m+1}-\tilde{\m{G}}_p)}\hookrightarrow \sfrac{k(X_{j_{\alpha}},\tilde Y)[\tilde X_{m+1}]}{(\tilde X_{m+1}-\tilde{\m{G}}_p)}\simeq k(\tilde X_{m+1},\tilde Y).\]
Then we can construct the following commutative diagram
\begin{equation}\label{eqn:commutativediagramembeddings}
    \begin{tikzcd}
 \nicefrac{A_{m+1}[Y^{\pm 1}]}{\got a_\gamma}=\!\!\!\!\!\!\!\!\!\!\!\!\!\!\!\!\!\!\!\!\!\!&\frac{k[X_1^{\pm 1},\dots,X_{m+1}^{\pm 1},Y^{\pm 1}]}{(\m{F}_2,\dots,\m{F}_m,\m{F}_{m+1})} \arrow[r, "\iota_\gamma"] & {k(X_{m+1},Y)}\\
  &\frac{k[X_1^{\pm 1},\dots, X_m^{\pm 1},\tilde X_{m+1}^{\pm 1},\tilde Y^{\pm 1}]}{(\tilde{\m{F}}_2,\dots,\tilde{\m{F}}_m, \tilde X_{m+1}-\tilde{\m{G}}_p)} \arrow[u, hook, two heads, "I_m\oplus M_\beta"] \arrow[r, hook] & {k(\tilde X_{m+1},\tilde Y)} \arrow[u, "M_\beta"']
\end{tikzcd}
\end{equation}
given by the matrix $M_\beta$. Therefore the homomorphism $\iota_\gamma$ is injective and takes $X_{m+1}\mapsto X_{m+1}$ and $Y\mapsto Y$.

\begin{lem}\label{lem:Dgamma}
With the notation above, there is an isomorphism
\[\frac {k[X_1^{\pm1},\dots, X_{m+1}^{\pm1},Y]}{(\m{F}_2,\dots,\m{F}_{m+1},Y)}\simeq k[X_{m+1}^{\pm 1}],\]
taking $X_{m+1}\mapsto X_{m+1}$. The images of $X_1,\dots,X_{m}$ in $k[X_{m+1}^{\pm 1}]$ lies in $k$.
\proof
Recall that for every $i=2,\dots,m$ there exists a (unique) Laurent polynomial $\bar{\m F}_i\in A_m$ such that $\tilde{\m{F}}_i\equiv\bar{\m F}_i\mod \tilde Y$. Since $\m F_{m+1}\equiv \bar{\m G}_p\mod Y$ and $\m{F}_i\equiv \bar{\m F}_i\mod Y$ for any $i=2,\dots,m$, we have
\[\tfrac {k[X_1^{\pm1},\dots, X_{m+1}^{\pm1},Y]}{(\m{F}_2,\dots,\m{F}_{m+1},Y)}\simeq\tfrac {k[X_1^{\pm1},\dots, X_{m+1}^{\pm1},\tilde Y]}{(\tilde{\m{F}}_2,\dots,\tilde{\m{F}}_{m},\bar{\m{G}}_p,\tilde Y)}\simeq \tfrac{D_\alpha}{(\bar{\m{G}}_p)}[X_{m+1}^{\pm 1}]\simeq k[X_{m+1}^{\pm 1}],\] 
and the isomorphisms take $X_{m+1}\mapsto X_{m+1}$, as required.
\endproof
\end{lem}

\begin{prop}\label{prop:embeddingRgamma}
With the notation above, there is an injective homomorphism
\[R_\gamma:=\sfrac{A_{m+1}[Y]}{\got a_\gamma}\hookrightarrow k(X_{m+1},Y).\]
taking $X_{m+1}\mapsto X_{m+1}$ and $Y\mapsto Y$. Furthermore, $YR_\gamma$ is prime ideal.
\proof
Lemma \ref{lem:Dgamma} shows $YR_\gamma$ is a prime ideal as $R_\gamma/(Y)\simeq k[X_{m+1}^{\pm 1}]$ is an integral domain. From (\ref{eqn:commutativediagramembeddings}) we have
\[\sfrac{A_{m+1}[Y]}{\got a_\gamma} \rightarrow \sfrac{A_{m+1}[Y^{\pm 1}]}{\got a_\gamma}\hookrightarrow k(X_{m+1},Y)\]
taking $X_{m+1}\mapsto X_{m+1}$ and $Y\mapsto Y$. Therefore it suffices to show that the ideal $\got a_\gamma$ of $A_{m+1}[Y]$ equals its saturation $\got a_\gamma:Y^\infty$ with respect to $Y$. Suppose not. Consider the primary decomposition of $\got a_\gamma$,
\[\got a_\gamma=\got q_1\cap\dots\cap\got q_s,\qquad \got p_i=\sqrt{\got q_i}.\]
Recall that the primary decomposition of $\got a_\gamma:Y^\infty$ consists of all the $\got q_i$'s which do not contain any power of $Y$. Hence there exists some $i=1,\dots,s$ such that $\got p_i\supseteq (Y)+\got a_\gamma$. Moreover, we can choose $i$ such that $\got p_i$ is a minimal prime ideal over $\got a_\gamma$, i.e.\ $\got p_i\in\mathrm{Min}(\got a_\gamma)$. Then, by Krull's height theorem, the height of $\got p_i$ is at most $m$ (the number of generators of $\got a_\gamma$), and so $\mathrm{ht}((Y)+\got a_\gamma)\leq m$. But
\[\dim\frac{A_{m+1}[Y]}{(Y)+\got a_\gamma}=1,\]
by Lemma \ref{lem:Dgamma}. This gives a contradiction, since 
\[\mathrm{ht}((Y)+\got a_\gamma)+\dim\tfrac{A_{m+1}[Y]}{(Y)+\got a_\gamma}=\dim A_{m+1}[Y]=m+2,\]
from the regularity of $A_{m+1}[Y]$.
\endproof
\end{prop}

\begin{prop}\label{prop:Cgamma}
Let $\beta\in\Z_+^2$ and $\gamma=\beta\circ_{g_p}\alpha$ as above. Then
\[C_\gamma=\Spec\frac{k[X_1^{\pm 1},\dots,X_{m+1}^{\pm 1},Y]}{(\m{F}_\gamma,\m{F}_2,\dots,\m{F}_{m+1})}.\]
\proof
The isomorphism in (\ref{eqn:Mgammaisomorphism}) implies that $C_{0,\gamma}\simeq\Spec\tfrac{A_{m+1}[Y^{\pm 1}]}{(\m F_\gamma)+\got a_\gamma}$ via $M_\gamma$. Then from the definition of $C_\gamma$, it suffices to show that $Y$ is a regular element of $R_\gamma/(\m{F}_\gamma)$, where $R_\gamma=A_{m+1}[Y]/\got a_\gamma$. From Proposition \ref{prop:embeddingRgamma} there is an injective homomorphism
$R_\gamma\hookrightarrow k(X_{m+1},Y)$,
taking $X_{m+1}\mapsto X_{m+1}$ and $Y\mapsto Y$. Moreover, $YR_\gamma$ is a prime ideal.
Therefore if $Y$ is a zero-divisor of $R_\gamma/(\m{F}_\gamma)$ then $\m{F}_\gamma\in Y R_\gamma$, as $R_\gamma$ is an integral domain. But this is not possible as $Y$ is not invertible in $R_\gamma$ and we chose $\m{F}_\gamma$ such that $\ord_Y(\m{F}_\gamma)=0$. Thus $Y$ is a regular element of $R_\gamma/(\m{F}_\gamma)$.
\endproof
\end{prop}

\begin{nt}\label{nt:Notationforgamma}
Let $\gamma=\beta\circ_{g_p}\alpha$, with $\beta\in\Z_+^2$ primitive. We have defined:
\begin{itemize}
    \item $\got a_\gamma=(\m F_2,\dots,\m F_{m+1})$ and $M_\gamma=(M_\beta\oplus I_m)\cdot M_{\alpha_p}$ for some matrix $M_\beta$ attached to $\beta$;
    \item $j_\gamma=m+1$, $R_\gamma=k[X_1^{\pm 1},\dots,X_{m+1}^{\pm 1},Y]/\got a_\gamma$;
    \item $\m F_\gamma\in k[X_{m+1},Y]$, with $Y\nmid \m F_\gamma$, satisfying $\m{F}_{\alpha,p}\stackrel{M_\beta}{=}X_{m+1}^\ast Y^\ast\cdot\m{F}_{\gamma}$, and $f|_{\gamma}\in k[X_{m+1}]$ given by $f|_{\gamma}(X_{m+1})=\m F_{\gamma}(X_{m+1},0)$.
\end{itemize}
With the notation above, $\gamma$ satisfies the properties \ref{eqn:Bakerprop1}, \ref{eqn:Bakerprop2}, \ref{eqn:Bakerprop3} of \ref{construction} by (\ref{eqn:Mgammaisomorphism}) and Propositions \ref{prop:embeddingRgamma}, \ref{prop:Cgamma}.
\end{nt}

Define $\bar{\m G}_{\singpointnt_n}\in k[X_{j_\alpha}]$, $\tilde{\m G}_{\singpointnt_n}\in k[X_{j_\alpha},\tilde Y]$ and $g_{\singpointnt_n}\in k[x_1, \dots, x_m,y]$ by 
\[\bar{\m G}_{\singpointnt_n}=\prod_{p\in \singpointnt_n}\bar{\m G}_p,\qquad \tilde{\m G}_{\singpointnt_n}=\prod_{p\in \singpointnt_n}\tilde{\m G}_p,\qquad g_{\singpointnt_n}=\prod_{p\in \singpointnt_n}g_p.\]
Then $\tilde{\m G}_{\singpointnt_n}\equiv \bar{\m G}_{\singpointnt_n}\mod \tilde Y$ and $g_{\singpointnt_n}$ is related to $\tilde{\m G}_{\singpointnt_n}$ by $M_\alpha$, i.e.\ $x_1^\ast\cdots x_m^\ast y^\ast\cdot g_{\singpointnt_n}=\tilde{\m{G}}_{\singpointnt_n}$ via $M_{\alpha}$. Define $\tilde\alpha=(0,1)\circ_{g_{\singpointnt_n}}\alpha$. Analogously to what we did for $\alpha_p$ in Remark \ref{rem:Matrixalphap}, we can uniquely construct a matrix $M_{\tilde\alpha}\in\SL_{m+2}(\Z)$ attached to $(0,1)\circ_{g_{\singpointnt_n}}v$ in such a way that the change of variables given by $M_{\tilde \alpha}$
restricts to the change of variables given by $M_\alpha$ on $(x_1,\dots, x_m, y)$ and
\[x_{m+1}-g_{\singpointnt_n}=X_1^\ast\cdots X_m^\ast \tilde Y^{\ord_v(g_{\singpointnt_n})}(\tilde X_{m+1}-\tilde{\m{G}}_{\singpointnt_n})\quad\text{via }M_{\tilde\alpha}.\]
In particular, denoting by $g_{m+1}\in k[x_1^{\pm 1},y^{\pm 1}]$ the unique polynomial such that $g_{m+1}\equiv g_{\singpointnt_n}\mod(f_2,\dots,f_m)$ one has
\begin{equation}\label{eqn:Mtildealphaisomorphism}
\sfrac{k[x_1^{\pm 1},\dots,x_{m+1}^{\pm 1},y^{\pm 1}]}{(f_2,\dots,f_m, x_{m+1}-g_{m+1})}\stackrel{M_{\tilde\alpha}}{\simeq} \sfrac{A_m[X_{m+1}^{\pm 1}, \tilde Y^{\pm 1}]}{\got a_\alpha+(X_{m+1}-\tilde{\m{G}}_{\singpointnt_n})}\hookrightarrow k(X_{j_{\alpha}},\tilde Y).
\end{equation}

\begin{rem}\label{rem:Matrixtildealpha}
The construction of $M_{\tilde\alpha}$ is given by Remark \ref{rem:Matrixalphap} by replacing $M_{\alpha_p}$ with $M_{\tilde\alpha}$, $g_p$ with $g_{\singpointnt_n}$, and $\tilde{\m G}_p$ with $\tilde{\m G}_{\singpointnt_n}$.
\end{rem}

\begin{lem}\label{lem:Ctildealpha}
With the notation above
 \[C_{\tilde\alpha}=\Spec \frac{k[X_1^{\pm 1},\dots,X_m^{\pm 1},X_{m+1}^{\pm 1},\tilde Y]}{(\m{F}_\alpha,\tilde{\m{F}}_2,\dots,\tilde {\m{F}}_m,X_{m+1}-\tilde{\m{G}}_{\singpointnt_n})}.\] 
Moreover, $C_{0,\tilde\alpha}$ is dense in $C_0$, i.e.\ $\tilde \alpha\in\Omega$, and for any $p\in\singpointnt_n$ the birational maps $s_{\tilde\alpha\alpha}, s_{\tilde\alpha\alpha_p}, s_{\alpha_p\alpha}$ induce a commutative diagram of open immersions
\[
\small\begin{tikzcd}[row sep=10pt, column sep=10pt]
C_{\tilde\alpha}\arrow[dr,hook, "s_{\tilde\alpha\alpha_p}"']\arrow[rr,hook, "s_{\tilde\alpha\alpha}"]&& C_\alpha\\
&C_{\alpha_p}\arrow[ur,hook, "s_{\alpha_p\alpha}"']&
\end{tikzcd}\normalsize
\]
where $s_{\tilde\alpha\alpha}$ has image $D(\tilde{\m G}_{\singpointnt_n})\subset C_\alpha$. 
Finally, $s_{\tilde\alpha\alpha}$ induces $\bar C_{\tilde\alpha}\simeq\bar C_\alpha\setminus \singpointnt_n$.
\proof
First note that $C_{0,\tilde\alpha}=\bigcap_{p\in\singpointnt_n}C_{0,\alpha_p}$. Then $C_{0,\tilde\alpha}$ is a dense open of $C_0$ by Lemma \ref{lem:Calphap}. The ring homomorphism
\[A_{\alpha_p}:=\tfrac{A_m[\tilde X_{m+1}^{\pm 1},\tilde Y]}{(\m{F}_\alpha, \tilde X_{m+1}-\tilde{\m{G}}_{p})+\got a_\alpha}\simeq \tfrac{R_\alpha[\tilde{\m G}_p^{-1}]}{(\m F_\alpha)}\rightarrow \tfrac{R_\alpha[\tilde{\m G}_{\singpointnt_n}^{-1}]}{(\m F_\alpha)}\simeq\tfrac{A_m[X_{m+1}^{\pm 1},\tilde Y]}{(\m{F}_\alpha, X_{m+1}-\tilde{\m{G}}_{\singpointnt_n})+\got a_\alpha}=:A_{\tilde\alpha}.\]
is injective by Lemma \ref{lem:D(G)pdense} and induces the birational map $s_{\tilde\alpha\alpha_p}$ if $\Spec A_{\tilde\alpha}= C_{\tilde\alpha}$ from (\ref{eqn:isoMalphap}) and (\ref{eqn:Mtildealphaisomorphism}). Since $\tilde Y$ is a regular element of $A_{\alpha_p}$ by Lemma \ref{lem:Calphap}, then $\tilde Y$ is a regular element of $A_{\tilde\alpha}$. This proves $C_{\tilde\alpha}=\Spec A_{\tilde\alpha}$ by definition of $C_{\tilde\alpha}$ and gives the required commutative diagram again by Lemma \ref{lem:Calphap}.
\endproof
\end{lem}

\begin{nt}\label{nt:Notationfortildealpha}
Define
\begin{itemize}
    \item $\got a_{\tilde\alpha}=\got a_\alpha+(\tilde X_{m+1}-\tilde{\m G}_{\singpointnt_n})\subset k[X_1^{\pm 1},\dots,X_m^{\pm 1},\tilde X_{m+1}^{\pm 1},\tilde Y]$ and $M_{\tilde\alpha}$ as described in Remark \ref{rem:Matrixtildealpha};
    \item $j_{\tilde\alpha}=j_\alpha$, $R_{\tilde\alpha}=R_\alpha[\tilde X_{m+1}^{\pm 1}]/(\tilde X_{m+1}-\tilde{\m G}_{\singpointnt_n})$ and $D_{\tilde\alpha}=D_{\alpha}[\bar{\m G}_{\singpointnt_n}^{-1}]$;
    \item $\m F_{\tilde\alpha}=\m F_\alpha$ and $f|_{\tilde \alpha}=f|_{\alpha}$.
\end{itemize}
With the notation above, $\tilde\alpha$ satisfies the properties \ref{eqn:Bakerprop1}, \ref{eqn:Bakerprop2}, \ref{eqn:Bakerprop3} of \ref{construction}.
\end{nt}

\begin{defn}\label{defn:Sigman}
For any $p\in \singpointnt_n$ let
\[
\Sigma_p=\{\gamma=\beta\circ_{g_p}\alpha\mid\beta\in\Z_+^2\text{ primitive, and } C_{0,\gamma}\subsetneq C_\gamma\}\subset\Omega,
\]
and $\Sigma_{\singpointnt_n}=\textstyle{\bigcup_{p\in \singpointnt_n}}\Sigma_{p}$. Define 
\[\hat\Sigma_{n}=\Sigma_{n}\setminus\{\alpha\},\qquad\tilde\Sigma_n=\hat\Sigma_n\cup\{\tilde\alpha\},\qquad\Sigma_{n+1}=\Sigma_{\singpointnt_n}\cup\tilde\Sigma_{n}.\]
\end{defn}


Recall that for any $\gamma,\gamma'\in\Omega\sqcup\{0\}$ we have a canonical way to glue the curves $C_\gamma$, $C_{\gamma'}$ through the birational maps $s_{\gamma\gamma'}$. Then \[C_{n+1}=\bigcup_{\gamma\in\Sigma_{n+1}}C_\gamma\cup C_0.\] 
We also define the following curves.

\begin{defn}\label{defn:Cn}
For any $p\in \singpointnt_n$ define
\[C_{p}:=\bigcup_{\gamma\in\Sigma_{p}}C_\gamma,
\qquad \hat C_n:=\bigcup_{\alpha'\in\hat\Sigma_{n}} C_{\alpha'}.\]
Then $C_{n+1}=\bigcup_{p\in\singpointnt_n}C_p\cup C_{\tilde\alpha}\cup\hat C_n\cup C_0$. 
\end{defn}

\subsection{The role of Newton polygons}\label{subsec:Newtonpolygons}
Let $p\in\singpointnt_n$. In this subsection we show that Newton polygons can be used to obtain an explicit description of the set $\Sigma_p$. We want to find all primitive vectors $\beta\in\Z_+^2$ such that $C_{0,\gamma}\subsetneq C_\gamma$, where $\gamma=\beta\circ_{g_p}\alpha$.

Let $\beta=(\beta_1,\beta_2)\in\Z_+^2$ be a primitive vector and let $\gamma=\beta\circ_{g_p}\alpha$. Recall that $f|_\gamma(X_{m+1})=\m F_\gamma(X_{m+1},0)$. Hence
$f|_{\gamma}\neq 0$ since $Y\nmid \m F_\gamma$.
Note that $D_\gamma=k[X_{m+1}^{\pm 1}]$ by Lemma \ref{lem:Dgamma}. Therefore $C_{\gamma}=C_{0,\gamma}$ if and only if $f|_\gamma$ is invertible in $k[X_{m+1}^{\pm1}]$. 
Since through the change of variables given by $M_{\beta}$
\[\m{F}_{\alpha,p}=X_{m+1}^\ast Y^{\ord_\beta(\m{F}_{\alpha,p})}\cdot\m{F}_\gamma,\]
from the Newton polygon $\Delta_{\alpha,p}$ of $\m{F}_{\alpha,p}$ one can see whether $f|_\gamma$ is invertible in $k[X_{m+1}^{\pm 1}]$ or not. 
Let $\phi:\Z^2\rightarrow\Z$ be the affine function defined by \[\phi(i,j)=\beta_1i+\beta_2j-\ord_\beta(\m{F}_{\alpha,p}).\]  
Then $f|_\gamma$ is not invertible in $k[X_{m+1}^{\pm 1}]$ if and only if $\phi^{-1}(0)\cap\Delta_{\alpha,p}=\text{edge}$.
Thus $\Sigma_p$ consists of all elements $\beta\circ_{g_p}\alpha$ such that $\beta\in\Z_+^2$ is the normal vector of some edge of $\Delta_{\alpha,p}$. All these elements are distinct as immediate consequence of Definition \ref{defn:circ}. Furthermore, note that this description shows that $\Sigma_p$ is finite and non-empty as $\tilde X_{m+1}\mid \m F_{\alpha,p}(\tilde X_{m+1},0)$ but $\tilde X_{m+1}\nmid \m F_{\alpha,p}$.

\begin{prop}\label{prop:reductionfgamma}
Let $\beta\in\Z_+^2$ be the normal vector of an edge $\ell$ of the Newton polygon $\Delta_{\alpha,p}$ of $\m{F}_{\alpha,p}$. Let $\gamma=\beta\circ_{g_p}\alpha$. Expand $\m{F}_{\alpha,p}=\sum_{i,j} c_{ij}\tilde X_{m+1}^i\tilde Y^j$, where $c_{ij}\in k$. Let $(i_0,j_0),\dots,(i_l,j_l)$ be the points of $\ell\cap\Z^2$, ordered along $\ell$ counterclockwise with respect to $\Delta_{\alpha,p}$. Then
\[f|_\gamma=X_{m+1}^d\sum_{r=0}^l c_{i_rj_r}X_{m+1}^r\]
for some $d\in\N$.
\proof
Let $(\delta_1,\delta_2)$ be the first row of $M_\beta$. It is easy to see that
\[f|_\gamma=\sum_{(i,j)\in \ell}c_{ij}X_{m+1}^{\delta_1i+\delta_2j+d'}\qquad\text{for some }d'\in\Z,\]
with $\delta_1i+\delta_2j+d'\geq 0$. Note that $(i_r,j_r)=(i_0,j_0)+r(\beta_2,-\beta_1)$. Therefore, for $d=d'+(\delta_1 i_0+\delta_2 j_0)$, the proposition follows since $\delta_1\beta_2-\delta_2\beta_1=1$.
\endproof
\end{prop}

\subsection{Inductive construction of the morphisms}\label{subsec:sninductiveconstruction}
In this subsection we want to construct a birational morphism $s_n:C_{n+1}\rightarrow C_n$. In \S\ref{subsec:properties} we will prove that $s_n$ is proper with the exceptional locus $s_n^{-1}(\singpointnt_n\cap\Sing(C_n))$.

\begin{rem}\label{rem:Toricschemedisjointcharts}
Let $p\in \singpointnt_n$. Similarly to the classical case (Remark \ref{rem:disjointchartsclassical}), for any $\gamma,\gamma'\in\Sigma_p$, $\gamma\neq\gamma'$, one has $ C_\gamma\cap C_{\gamma'}=C_{0,\gamma}$. More precisely, the birational map $s_{\gamma\gamma'}:C_\gamma\--> C_{\gamma'}$ has domain of definition $C_{0,\gamma}$ giving an isomorphism $C_{0,\gamma}\rightarrow C_{0,\gamma'}$.
\end{rem}

\begin{rem}\label{rem:sgammaalphadefn}
Let $p\in\singpointnt_n$. For any primitive vector $\beta\in\Z_+^2$, if $\gamma=\beta\circ_{g_p}\alpha$ then from (\ref{eqn:homIxMbeta}) we obtain the homomorphism of rings
\begin{equation}\label{eqn:morphism_sn}
\frac{A_m[\tilde Y]}{(\m{F}_{\alpha})+\got a_\alpha}\simeq\frac{A_m[\tilde X_{m+1}, \tilde Y]}{(\m{F}_{\alpha},\tilde X_{m+1}-\tilde{\m{G}}_p)+\got a_\alpha}\xrightarrow{I_m\oplus M_\beta}\frac{A_{m+1}[Y]}{(\m{F}_{\gamma})+\got a_\gamma}
\end{equation}
that induces a birational morphism $C_\gamma\rightarrow C_\alpha$. In fact, from the definition of $M_\gamma$ we see that it agrees with $s_{\gamma\alpha}:C_\gamma\--> C_\alpha$ as rational map.
\end{rem}

\begin{lem}\label{lem:spstructure}
Let $p\in\singpointnt_n$ and $\gamma=\beta\circ_{g_p}\alpha$ for some primitive $\beta\in\Z_+^2$. Then $s_{\gamma\alpha}: C_\gamma\rightarrow C_{\alpha}$ restricts to an isomorphism $C_{0,\gamma}\stackrel{\sim}{\longrightarrow}D(\tilde{\m G}_p)\cap C_{0,\alpha}\subset C_{\alpha} $and $s_{\gamma\alpha}( C_\gamma\setminus C_{0,\gamma})\subseteq\{p\}$. 
\proof
Let $\gamma=\beta\circ_{g_p}\alpha$ for some $\beta\in\Z_+^2$. 
The first part of the lemma follows from Remark \ref{rem:sgammaalphadefn} and (\ref{eqn:ImMbetaisomorphism}).
The morphism $s_{\gamma\alpha}$ is induced by the ring homomorphism taking $\tilde Y\mapsto X_{m+1}^{\delta_2}Y^{\beta_2}$, with $M_\beta=\begin{psmallmatrix}\!\!\delta_1\!&\delta_2\!\!\cr \!\!\beta_1\!&\beta_2\!\!\end{psmallmatrix}$. Recall 
\[C_{\gamma}\setminus C_{0,\gamma}=\bar C_\gamma=\{Y=0\}\subset C_\gamma.\] 
Since $\bar{\m G}_p\equiv\m F_{m+1}\mod Y$, the morphism $s_{\gamma\alpha}$ takes $\bar C_{\gamma}$ into the closed subscheme $\{\bar{\m G}_p=0\}$ of $\bar C_\alpha$. This concludes the proof as $\{\bar{\m G}_p=0\}=\{p\}$.
\endproof
\end{lem}

Let $p\in\singpointnt_n$. Considering $p$ as a point of $C_{\alpha}$, denote by $U_{\alpha,p}$ the open subscheme $D(\tilde{\m G}_p)\cap C_{0,\alpha}$ of $C_{\alpha}$. Recall that $C_{0,\alpha}$ is dense in $C_{\alpha}$ by definition. Hence Lemma \ref{lem:D(G)pdense} implies that $U_{\alpha,p}$ is dense. Let $C_{\alpha,p}=U_{\alpha,p}\cup\{p\}$ as subset of $C_{\alpha}$. We want to show that $C_{\alpha,p}$ is dense and open in $C_\alpha$. From the density of $U_{\alpha,p}$ it follows that $C_{\alpha,p}$ is dense and that $V_p:=C_{\alpha}\setminus U_{\alpha,p}$ is a finite set of closed points of $C_{\alpha}$. Thus $C_{\alpha,p}$, complement of $V_p\setminus \{p\}$, is open in $C_\alpha$.

\begin{defn}
For any $p\in\singpointnt_n$ we define $C_{\alpha,p}$ to be the dense open subset $U_{\alpha,p}\cup\{p\}$ of $C_\alpha$, equipped with the canonical structure of open subscheme.
\end{defn}

Let $p\in\singpointnt_n$. By Remark \ref{rem:Toricschemedisjointcharts} and Lemma \ref{lem:spstructure}, the maps $s_{\gamma\alpha}:C_\gamma\rightarrow C_\alpha$, for $\gamma\in\Sigma_p$, glue to a morphism $C_p\rightarrow C_{\alpha,p}$.

\begin{defn}
For any $p\in\singpointnt_n$, define $s_p:C_p\rightarrow C_{\alpha,p}$ as the glueing of the morphisms $s_{\gamma\alpha}:C_\gamma\rightarrow C_\alpha$, for all $\gamma\in\Sigma_p$. 
\end{defn}

\begin{lem}\label{lem:spseparatedness}
The morphism $s_p: C_p\rightarrow C_{\alpha,p}$ is separated.
\proof
Consider the open immersion $\iota_p: C_{\alpha,p}\rightarrow C_\alpha$. 
By \cite[Proposition 3.3.9(e)]{Liu}
it suffices to prove that $\iota_p\circ s_p$ is separated. 
Since $C_\alpha$ is affine, we only have to show that $C_p$ is separated over $\Spec k$ by \cite[Exercise 3.3.2]{Liu}.
Let $\Delta_{\alpha,p}$ be the Newton polygon of $\m F_{\alpha,p}$. 
Recall from \S\ref{subsec:Newtonpolygons} that
\[C_p=\bigcup_{\beta} C_{\beta\circ_{g_p}\alpha}\qquad\text{with}\quad C_{\gamma}=\Spec\frac{A_m[X_{m+1}^{\pm 1},Y]}{(\m{F}_\gamma)+\got a_\gamma}, \quad\gamma=\beta\circ_{g_p}\alpha,\]
where $\beta$ runs through all normal vectors in $\Z_+^2$ of edges of $\Delta_{\alpha,p}$ and the curves $C_{\beta\circ_{g_p}\alpha}$ are glued along their common open $C_{0,\alpha_p}=C_{0,\beta\circ_{g_p}\alpha}$. To avoid confusion, if $\gamma=\beta\circ_{g_p}\alpha$, rename the variables $X_{m+1},Y$ of $\O_{C_\gamma}(C_\gamma)$ to $X_\beta,Y_\beta$. Since closed immersions are separated and separated morphisms are stable under base changes it suffices to prove that the toric variety $\bigcup_\beta\Spec k[X_\beta^{\pm 1},Y_\beta]\subset \T_{\Delta_{\alpha,p}}$ is separated. This follows from the classical theory on toric varieties.
\endproof
\end{lem}

\begin{lem}\label{lem:spOverC0alpha}
The morphism $s_p$ induces an isomorphism $s_p^{-1}(U_{\alpha,p})\rightarrow U_{\alpha,p}$. In particular, $s_p$ is birational.
\proof
This result immediately follows from Lemma \ref{lem:spstructure} as $\Sigma_p\neq\varnothing$.
\endproof
\end{lem}

\begin{lem}\label{lem:spproperness}
The morphism $s_p: C_p\rightarrow C_{\alpha,p}$ is proper.
\proof
By Lemma \ref{lem:spseparatedness}, the morpshism $s_p$ is separated. We will then prove the lemma via the valuative criterion for properness. Let $R$ be a discrete valuation ring with field of fractions $K$. We want to prove that any commutative diagram
\[\begin{tikzcd}
\Spec K \arrow[r, "t_{p}"] \arrow[d, hook] & C_p \arrow[d, "s_p"]\\
\Spec R \arrow[r, "t_{\alpha}"] \arrow[ru, dashed] & C_{\alpha,p}
\end{tikzcd}\]
can be filled in as shown. Let $\pi$ be a uniformiser of $R$ and let $\omega=(\pi)$ be the closed point of $\Spec R$. Since $C_{\alpha,p}=U_{\alpha,p}\cup\{p\}$ and $s_p^{-1}(U_{\alpha,p})\rightarrow U_{\alpha,p}$ is an isomorphism by Lemma \ref{lem:spOverC0alpha}, we can assume $p=t_{\alpha}(\omega)$. Indeed, if not, then $s_p^{-1}$ is defined on the open dense neighbourhood $U_{\alpha,p}$ of $t_\alpha(\omega)$, that therefore contains the image of $t_\alpha$. Moreover, $t_\alpha$ can be supposed not constant, otherwise $\Spec R\rightarrow C_p$ can be defined as the constant morphism of image $t_p((0))$.

Recall that the injective homomorphism $R_\alpha\hookrightarrow k(X_{j_\alpha},\tilde Y)$, given by \ref{construction}\ref{eqn:Bakerprop2}, induces an open immersion $C_\alpha\hookrightarrow \mathrm{Spec}( \nicefrac{k[X_{j_\alpha},\tilde Y]}{(\m F_\alpha)})$. In particular, the local ring $\O_{C_{\alpha},p}$, equal to $\O_{C_{\alpha,p},p}$, is naturally isomorphic to the localisation of $k[X_{j_\alpha},\tilde Y]/(\m F_\alpha)$ at the prime ideal $(\tilde{\m G}_p,\tilde Y)=(\bar{\m G}_p,\tilde Y)$. From the local homomorphism
\[\tau: \tfrac{k[X_{j_\alpha},\tilde Y]_{(\tilde{\m G}_p,\tilde Y)}}{(\m F_\alpha)}\simeq \O_{C_{\alpha,p},p}\xrightarrow{t_{\alpha,\omega}^\#}R\]
we observe that $\ord_\pi(\tilde{\m{G}}_p)>0$, $\ord_\pi(\tilde Y)>0$. We want to show that neither $\tilde Y$ nor $\tilde{\m G}_p$ are taken to $0$ by $\tau$. Note that $\ker(\tau)\subsetneq(\tilde{\m G}_p,\tilde Y)$, since $t_\alpha$ is not constant. 
Hence it suffices to prove that $\tau(\tilde Y)=0$ if and only if $\tau(\tilde{\m G}_p)=0$.

Suppose $\tau(\tilde Y)=0$. Then $\tau(f|_{\alpha})=0$ and $\tau(\tilde{\m G}_p)=\tau(\bar{\m G}_p)$. Recall that $\bar{\m G}_p$ is a factor of $f|_{\alpha}$. Let $h_p\in k[X_{j_\alpha}]$ with $\bar{\m G}_p\nmid h_p$ such that $f|_{\alpha}=h_p\cdot(\bar{\m G}_p)^{m_p}$, for some $m_p\in\Z_+$. Note that $\tau(h_p)$ is invertible as $h_p\notin(\tilde{\m G}_p,\tilde Y)$. Since $\tau(f|_{\alpha})=0$ and $R$ is reduced, it follows that $\tau(\tilde{\m G}_p)=0$. 

Suppose $\tau(\tilde{\m G}_p)=0$. Let $\m H_p\in k[\tilde Y]$ be the normal form of $\m F_\alpha$ by $\tilde{\m G}_p$ with respect to the lexicographic order on $k[X_{j_\alpha},\tilde Y]$ given by $X_{j_\alpha}>\tilde Y$. Note that $\tau(\m H_p)=0$ as $\tau(\m F_\alpha)=0$, but $\m H_p\neq 0$ as $\tilde{\m G}_p\nmid \m F_\alpha$. Recall that $\tilde{\m G}_p-\bar{\m G}_p\in \tilde Yk[\tilde Y]$. 
Since $\bar{\m G}_p$ is a degree $1$ factor of $f|_{\alpha}$ and $\m F_\alpha- f|_{\alpha}\in (\tilde Y)$, one has $\m H_p\in\tilde Yk[\tilde Y]$. Write $\m H_p=\tilde Y^t\cdot \m H$, for $t\in\Z_+$ and $\m H\notin \tilde Y k[\tilde Y]$. Note that $\tau(\m H)$ is invertible as $\m H\notin (\tilde{\m G}_p,\tilde Y)$. It follows that $\tau(\tilde Y)=0$ since $R$ is reduced and $\tau(\m H_p)=0$. 

Hence $\ord_\pi(\tilde{\m{G}}_p),\ord_\pi(\tilde Y)\in\Z_+$ and so the affine function
\[
    \phi:\Z^2 \rightarrow\Z,\quad
    (i,j)\mapsto \ord_\pi\tilde{\m{G}}_p^i \tilde Y^j
\]
is a non-trivial linear map with a rank $1$ kernel spanned by some primitive vector $(\beta_2,-\beta_1)\in \Z_{+}\times\Z_{-}$. 
Set $\beta=(\beta_1,\beta_2)$ and $\gamma=\beta\circ_{g_p}\alpha$. Then
\[C_\gamma=\Spec\frac{A_{m+1}[Y]}{(\m{F}_\gamma)+\got a_\gamma}\]
and $C_\gamma\subset C_p$ from the definition of $C_p$ (also when $\gamma\notin\Sigma_p$). Hence 
\begin{equation}\label{eqn:commdiagramproperness}
\begin{tikzcd}[row sep =3ex]
K &  \tfrac{A_m[X_{m+1}^{\pm 1},Y]}{(\m{F}_\gamma)+\got a_\gamma} \arrow[ld, dashed] \arrow[l] \\
R \arrow[u, hook]  & \tfrac{A_m[\tilde Y]}{(\m{F}_\alpha)+\got a_\alpha}, \arrow[u] \arrow[l]
\end{tikzcd}
\end{equation}
where the ring homomorphism on the right, inducing the map \[s_{\gamma\alpha}:C_\gamma\xrightarrow{s_p}C_{\alpha,p}\hookrightarrow C_\alpha,\] is given by $\tilde Y\mapsto X_{m+1}^{\delta_2} Y^{\beta_2}$ for $M_\beta=\lb\begin{smallmatrix}\!\!\delta_1\!&\delta_2\!\!\cr \!\!\beta_1\!&\beta_2\!\!\end{smallmatrix}\rb\in\SL_2(\Z)$.
To conclude the proof it suffices to show that the commutative diagram (\ref{eqn:commdiagramproperness}) can be filled in as shown.
Recall \vspace{-4pt}\[\m F_{m+1}=X_{m+1}^{\delta_1}Y^{\beta_1}-\m G_p\in \got a _\gamma.\]
\vspace{-4pt} Then \[\ord_\pi(X_{m+1})=\ord_\pi(\tilde{\m{G}}_p^{\beta_2}\tilde Y^{-\beta_1})=\phi((\beta_2,-\beta_1))=0\] 
and so $\ord_\pi(Y)>0$ as $\beta\in\Z_+^2$. Thus (\ref{eqn:commdiagramproperness}) can be filled in as shown.
\endproof
\end{lem}

\begin{lem}\label{lem:spisomorphism}
If $p\in\singpointnt_n$ is a regular point of $C_n$, then $s_p$ is an isomorphism.
\proof
As $p$ is a regular point of codimension $1$, the ring $\O_{C_{\alpha,p},p}$ is normal. Therefore there exists a normal integral open subscheme $U\subseteq C_{\alpha,p}$ containing $p$. Since $s_p$ is proper birational by Lemma \ref{lem:spproperness}, then so is $s_U:s_p^{-1}(U)\rightarrow U$. In particular, $s_p^{-1}(U)$ is integral. It follows from \cite[Corollary 4.4.3]{Liu}) that $s_U$ is an isomorphism. Thus $s_{p}$ is an isomorphism, since $s_p^{-1}(U_{\alpha,p})\rightarrow U_{\alpha,p}$ is an isomorphism by Lemma \ref{lem:spOverC0alpha}. 
\endproof
\end{lem}

\begin{prop}\label{prop:C0gammadenseinCn+1}
For any $\gamma\in\Sigma_{n+1}$, the curve $C_{0,\gamma}$ is dense in $C_{n+1}$.
\proof
For any $\gamma\in\Sigma_{n+1}$ recall that $C_{0,\gamma}$ is dense in its closure $C_\gamma$. Therefore $C_0=\bigcup_{\gamma\in\Sigma_{n+1}}C_{0,\gamma}\cup C_0$ is dense in $C_{n+1}=\bigcup_{\gamma\in\Sigma_{n+1}}C_{\gamma}\cup C_0$. Fix $\gamma\in\Sigma_{n+1}$. It suffices to show that $C_{0,\gamma}$ is dense $C_0$. But this holds as $\gamma\in\Omega$.
\endproof
\end{prop}

\begin{defn}\label{defn:psin}
Define a surjective function $\psi_n:\Sigma_{n+1}\sqcup\{0\}\rightarrow\Sigma_n\sqcup\{0\}$ by
$\psi_n(0)=0,\psi_n|_{\hat\Sigma_n}=id_{\hat\Sigma_n}, \psi_n(\Sigma_{n+1}\setminus\hat\Sigma_n)=\{\alpha\}$.
\end{defn}

Let $\gamma\in\Sigma_{n+1}\sqcup\{0\}$ and denote $\alpha_\gamma=\psi_n(\gamma)$. Then the birational map $s_{\gamma\alpha_\gamma}$ has domain of definition $C_\gamma$. Indeed, it is trivial when $\gamma=0$ or $\gamma\in\hat\Sigma_n$ while it follows from Remark \ref{rem:sgammaalphadefn} if $\gamma\in\Sigma_{\singpointnt_n}$ and from Lemma \ref{lem:Ctildealpha} if $\gamma=\tilde\alpha$.

\begin{thm}\label{thm:sndefn}
There exists a unique morphism $s_n:C_{n+1}\rightarrow C_n$ extending the birational maps $s_{\gamma'\alpha'}: C_{\gamma'}\--> C_{\alpha'}$ for $\gamma'\in\Sigma_{n+1}\sqcup\{0\}$, $\alpha'\in\Sigma_n\sqcup\{0\}$. In particular, 
\[s_n|_{C_0}:C_0\xrightarrow{id} C_0\subseteq C_n,\quad s_n|_{\hat C_n}:\hat C_n\xrightarrow{id}\hat C_n\subseteq C_n,
\] 
and $s_n|_{C_p}:C_p\xrightarrow{s_p} C_{\alpha,p}\subseteq C_n$, for any $p\in\singpointnt_n$.
\proof
For any $\gamma\in\Sigma_{n+1}\sqcup\{0\}$ let $\alpha_\gamma=\psi_n(\gamma)$. We observed that the birational maps $s_{\gamma\alpha_\gamma}$ have domain of definition $C_\gamma$, and so define morphisms \[s_\gamma:C_\gamma\xrightarrow{s_{\gamma\alpha_\gamma}} C_{\alpha_\gamma}\subseteq C_n.\] 
Note that $s_\gamma|_{C_{0,\gamma}}$ is an open immersion. This fact is trivial when $\gamma\in\hat\Sigma_n\sqcup\{0\}$ and follows from Lemmas \ref{lem:Ctildealpha} and \ref{lem:spstructure} otherwise.

Recall the definition of the dense open $U_{\gamma\gamma'}$ of $C_\gamma$ for any $\gamma,\gamma'\in\Omega\sqcup\{0\}$.
We want to show that for any $\gamma,\gamma'\in\Sigma_{n+1}\sqcup\{0\}$ and any $\alpha'\in\Sigma_n\sqcup\{0\}$ the maps $s_{\gamma}$ and $s_{\gamma'\alpha'}:C_{\gamma'}\-->C_{\alpha'}\subseteq C_n$ agree on the intersection of their domains of definition. Let $D$ be the domain of definition of $s_{\gamma'\alpha'}$. Then $D\supseteq U_{\gamma'\alpha'}$. Let $U=D\cap C_\gamma\subseteq C_{n+1}$ and $U_0=C_{0,\gamma}\cap C_{0,\gamma'}\cap C_{0,\alpha'}$. Since $C_{0,\gamma'}\cap C_{0,\alpha'}\subseteq U_{\gamma'\alpha'}$, one has $U_0\subseteq D$. Hence $U_0$ is an open of $U$, dense by Proposition \ref{prop:C0gammadenseinCn+1}. Now, $U$ is reduced, $C_n$ is separated and $s_\gamma|_{U_0}=s_{\gamma'\alpha'}|_{U_0}$ by definition. Therefore \cite[Proposition 3.3.11]{Liu} implies the two maps coincide on $U$, as required.

Thus the morphisms $s_\gamma$ glue to a morphism $s_{n}:C_{n+1}\rightarrow C_{n}$ and $s_n$ extends the birational maps $s_{\gamma'\alpha'}: C_\gamma\--> C_{\alpha'}$ for $\gamma'\in\Sigma_{n+1}\sqcup\{0\}$, $\alpha'\in\Sigma_n\sqcup\{0\}$. Then the uniqueness follows. 
\endproof
\end{thm}

\begin{defn}\label{defn:sn}
Define $s_{n}:C_{n+1}\rightarrow C_{n}$ to be the birational morphism of $k$-schemes of Theorem \ref{thm:sndefn}. We call $s_n$ the morphism \textit{resolving $\singpointnt_n$} (although $s_n^{-1}(\singpointnt_n)$ is not necessarily non-singular).
\end{defn}

\begin{rem}\label{rem:snglueing}
Let $\gamma,\gamma'\in \Sigma_{n+1}\sqcup\{0\}$ and $\alpha'\in\Sigma_n\sqcup\{0\}$. Suppose there exist open subschemes $V_{\alpha'}\subseteq C_{\alpha'}$, $U_\gamma\subseteq C_\gamma$, $U_{\gamma'}\subseteq C_{\gamma'}$ such that $s_n$ restricts to isomorphisms $U_\gamma\rightarrow V_{\alpha'}$, $U_{\gamma'}\rightarrow V_{\alpha'}$. Since $s_n$ extends the rational maps $s_{\gamma\alpha'}$, $s_{\gamma'\alpha'}$, the map $s_{\gamma\gamma'}$ is defined on $U_\gamma$ and induces an isomorphism $U_\gamma\rightarrow U_{\gamma'}$. This implies that the opens $U_\gamma$, $U_{\gamma'}$ are glued, and so are equal in $C_{n+1}$.

It follows that if $U_1,U_2$ are opens of $C_{n+1}$ such that $s_n|_{U_1}$ and $s_n|_{U_2}$ are open immersions, then $s_n|_{U_1\cup U_2}$ is an open immersion.
\end{rem}


\subsection{Geometric properties}\label{subsec:properties}
In this subsection we will show that $\Sigma_{n+1}$ and $C_{n+1}$ satisfy all remaining properties of \ref{construction}, i.e.\ $C_n$ is a projective curve and $C_{\gamma}\cap C_{\gamma'}=C_{0,\gamma}\cap C_{0,\gamma'}$ for any $\gamma,\gamma\in\Sigma_{n+1}\sqcup\{0\}$, $\gamma\neq\gamma'$. Furthermore, we will prove that the morphism $s_n$ defined in Theorem \ref{thm:sndefn}, is a proper birational morphism with exceptional locus $s_n^{-1}(\singpointnt_n\cap\Sing(C_n))$. 


Consider the principal open $D(\tilde{\m G}_{\singpointnt_n})\subset C_{\alpha}$. Note that
\begin{equation}\label{eqn:CnOpenCover}
\m U=\{C_{\alpha'}\mid \alpha'\in\hat\Sigma_n\}\cup \{C_0\}\cup\{C_{\alpha,p}\mid p\in\singpointnt_n\}\cup\{ D(\tilde{\m G}_{\singpointnt_n})\}
\end{equation}
is an open cover of $C_n$.

\begin{lem}\label{lem:snsurjectivity}
The morphism $s_n:C_{n+1}\rightarrow C_n$ is surjective.
\proof
We want to show that every open in the cover (\ref{eqn:CnOpenCover}) is contained in the image of $s_n$. Recall $s_{\tilde\alpha\alpha}(C_{\tilde\alpha})=D(\tilde{\m G}_{\singpointnt_n})$ by Lemma \ref{lem:Ctildealpha}. Moreover, the morphism $s_p:C_p\rightarrow C_{\alpha,p}$ is surjective by Lemma \ref{lem:spstructure} as $\Sigma_p\neq\varnothing$.
Then the lemma follows from Theorem \ref{thm:sndefn}.
\endproof
\end{lem}

\begin{lem}\label{lem:inverseimagesn}
For any $p\in \singpointnt_n$, we have
\begin{equation*}
    s_n^{-1}(p)=C_p\setminus C_{0,\alpha_p},\qquad\text{and}\qquad s_n^{-1}(C_{\alpha,p})=C_p.
\end{equation*}
Furthermore, the morphism $s_n^{-1}(C_n\setminus \singpointnt_n)\rightarrow C_n\setminus\singpointnt_n$ induced by $s_n$ is an isomorphism.
\proof
Let $p\in\singpointnt_n$. Lemma \ref{lem:spstructure} shows that \[s_p^{-1}(p)=\bigcup_{\gamma\in\Sigma_p}(C_\gamma\setminus C_{0,\gamma})=C_p\setminus C_{0,\alpha_p}, \]
where the last equality holds as $C_{0,\gamma}=C_{0,\alpha_p}$ for all $\gamma\in\Sigma_p$. Moreover, $p\notin s_n(C_q)$ for any $q\in \singpointnt_n$, $q\neq p$, and also $p\notin s_n(C_{\tilde\alpha})$ by Lemma \ref{lem:Ctildealpha}.
Recall $p\notin C_0$. In particular, $p\notin C_{\alpha'}$ for any $\alpha'\in\hat\Sigma_n$, since $C_\alpha\cap C_{\alpha'}\subseteq C_0$ (from our assumptions on $C_n$). 
Then $s_n^{-1}(p)=s_p^{-1}(p)$ by Theorem \ref{thm:sndefn}.

Let $U_{\singpointnt_n}=C_n\setminus \singpointnt_n$. We want to show that $s_n^{-1}(U_{\singpointnt_n})\rightarrow U_{\singpointnt_n}$ is an isomorphism. From above
\[s_n^{-1}(U_{\singpointnt_n})=\hat C_{n}\cup C_0\cup C_{\tilde\alpha},\]
as $C_{0,\gamma}\subseteq C_0$ for any $\gamma\in\Sigma_{\singpointnt_n}$. Note that $s_n|_{\hat C_n}$, $s_n|_{C_0}$ and $s_n|_{C_{\tilde\alpha}}$ are open immersions by Theorem \ref{thm:sndefn}. Thus $s_n^{-1}(U_{\singpointnt_n})\rightarrow U_{\singpointnt_n}$ is an isomorphism from Remark \ref{rem:snglueing} and Lemma \ref{lem:snsurjectivity}.

Recall that $C_{\alpha,p}\setminus\{p\}=U_{\alpha,p}\subseteq U_{\singpointnt_n}$ and $s_p^{-1}(U_{\alpha,p})=C_{0,\alpha_p}$ by Lemma \ref{lem:spstructure}. 
Moreover, $s_p$ induces an isomorphism $C_{0,\alpha_p}\rightarrow U_{\alpha,p}$ by Lemma \ref{lem:spOverC0alpha}.
Since $s_n^{-1}(U_{\singpointnt_n})\rightarrow U_{\singpointnt_n}$ is an isomorphism, $s_n^{-1}(C_{\alpha,p})=s_p^{-1}(C_{\alpha,p})=C_p$.
\endproof
\end{lem}

\begin{thm}\label{thm:Bakercompleteness}
The morphism $s_n:C_{n+1}\rightarrow C_{n}$ resolving $\singpointnt_n\subseteq\Sing(\bar C_\alpha)$
is a surjective proper birational morphism with exceptional locus contained in $s_n^{-1}(\singpointnt_n)$. In particular, the curve $C_{n+1}$ is projective.
\proof
First recall $s_n$ is surjective by Lemma \ref{lem:snsurjectivity}.
Consider the open cover $\m U$ of $C_n$ introduced in (\ref{eqn:CnOpenCover}).
As properness is a local property on the codomain, if $s_n^{-1}(U)\rightarrow U$ is proper for any $U\in\m U$, then $s_n$ is proper. Lemma \ref{lem:inverseimagesn} implies that $s_n^{-1}(U)\rightarrow U$ is an isomorphism except when $U=C_{\alpha,p}$ for some $p\in\singpointnt_n$. But $s_n^{-1}(C_{\alpha,p})=C_p$ for any $p\in\singpointnt_n$ again by Lemma \ref{lem:inverseimagesn}. Hence Lemma \ref{lem:spproperness} implies that $s_n$ is proper.
It follows that the curve $C_{n+1}$ is complete, and then projective, since so is $C_n$.

Proposition \ref{prop:C0gammadenseinCn+1} implies that $C_0$ is dense in $C_{n+1}$. Let $U_{\singpointnt_n}=C_n\setminus \singpointnt_n$, dense in $C_n$. Since $C_0\subseteq s_n^{-1}(U_{\singpointnt_n})$, the isomorphism $s_n^{-1}(U_{\singpointnt_n})\rightarrow U_{\singpointnt_n}$ implies that $s_n$ is birational with exceptional locus contained in $s_n^{-1}(\singpointnt_n)$. 
\endproof
\end{thm}

\begin{lem}\label{lem:snexceptionallocus}
Let $s_n: C_{n+1}\rightarrow C_n$ be the morphism resolving $\singpointnt_n\subseteq\Sing(\bar C_\alpha)$. 
Let $p\in\singpointnt_n$. Then $p\in \Reg(C_n)$ if and only if the exceptional locus of $s_n$ is contained in $s_n^{-1}(\singpointnt_n\setminus \{p\})$. In that case, $\bar C_\gamma$ is regular for all $\gamma\in\Sigma_p$.
\proof
If $p\in \Reg(C_n)$ then $s_p:C_p\rightarrow C_{\alpha,p}$ is an isomorphism by Lemma \ref{lem:spisomorphism}.
Then the exceptional locus of $s_n$ is contained in $s_n^{-1}(\singpointnt_n\setminus \{p\})$ by Lemma \ref{lem:inverseimagesn} and Theorem \ref{thm:sndefn}.

Suppose the exceptional locus of $s_n$ is contained in $s_n^{-1}(\singpointnt_n\setminus \{p\})$. In particular, there exists an open neighbourhood $U$ of $p$ such that $s_n^{-1}(U)\rightarrow U$ is an isomorphism. This implies that $s_p: C_p\rightarrow C_{\alpha,p}$ is an isomorphism by Theorem \ref{thm:sndefn} and Lemma \ref{lem:inverseimagesn}. Recall $\Sigma_p\neq\varnothing$. Let $\gamma\in\Sigma_p$ so that $\gamma=\beta\circ_{g_p}\alpha$ with $\beta\in\Z_+^2$. 
As in \S\ref{subsec:BakerInductiveStep}, write
\[C_\alpha=\Spec\frac{A_m[\tilde Y]}{(\m{F}_\alpha)+\got a_\alpha},\qquad C_\gamma=\Spec \frac{A_m[X_{m+1}^{\pm 1},Y]}{(\m{F}_\gamma)+\got a_\gamma}.\]
Consider the morphism $s_{\gamma\alpha}:C_\gamma\rightarrow C_\alpha$ induced by the ring homomorphism taking $\tilde Y\mapsto X_{m+1}^{\delta_2} Y^{\beta_2}$, where $\beta=(\beta_1,\beta_2)$ and $\delta_1,\delta_2\in\Z$ such that $\delta_1\beta_2-\delta_2\beta_1=1$. Recall that $s_{\gamma\alpha}(C_\gamma\setminus C_{0,\gamma})=\{p\}$ by Lemma \ref{lem:spstructure}. 
As $s_p$ is an isomorphism, $s_{\gamma\alpha}$ is an open immersion. 
In particular, 
\begin{equation}\label{eqn:morphismsgammaalpha}
s_{\gamma\alpha}^\#(U_{\alpha\gamma}):\O_{C_\alpha}(U_{\alpha\gamma})\rightarrow \O_{C_\gamma}(C_\gamma)
\end{equation}
is an isomorphism, where $U_{\alpha\gamma}=s_{\gamma\alpha}(C_\gamma)$. 
In fact, $U_{\alpha\gamma}=C_{\alpha,p}$.
Then $p\in U_{\alpha\gamma}$ and $\got m_{p}=(\tilde{\m{G}}_p,\tilde Y)+\got a_\alpha$ is the maximal ideal of $\O_{C_\alpha}(U_{\alpha\gamma})$ corresponding to $p$. Recall $\m{F}_{m+1}=X_{m+1}^{\delta_1} Y^{\beta_1}-\m{G}_p\in\got a_\gamma$. Then $\got m_{p}\O_{C_\gamma}(C_\gamma)\subseteq(\m{F}_\gamma,Y)+\got a_\gamma$, which implies the equality, since $\got m_{p}\O_{C_\gamma}(C_\gamma)$ has to be maximal. It follows that the ring isomorphism (\ref{eqn:morphismsgammaalpha}) induces
\[k\simeq\frac{D_\alpha}{(\bar{\m{G}}_p)}\simeq\frac{A_m[\tilde Y]_{(\tilde{\m{G}}_p,\tilde Y)+\got a_\alpha}}{(\tilde{\m{G}}_p,\tilde Y)+\got a_\alpha}\xrightarrow{\sim}\frac{A_m[X_{m+1}^{\pm 1},Y]_{(\m{F}_\gamma,Y)+\got a_\gamma}}{(\m{F}_\gamma,Y)+\got a_\gamma}\simeq \frac{D_\gamma}{(f|_\gamma)}.\]
Therefore $\bar C_\gamma=\Spec D_\gamma/(f|_{\gamma})\simeq \Spec k$, and so is regular. In particular, the point $w=s_n^{-1}(p)$ is a non-singular point of $C_{n+1}$ (Remark \ref{rem:nonsingularbarjacobian}). Thus $C_n$ is regular at $p$, as $w$ is not in the exceptional locus of $s_n$.
\endproof
\end{lem}

\begin{prop}\label{prop:nonsingularisomorphism}
Let $s_n: C_{n+1}\rightarrow C_n$ be the morphism resolving $\singpointnt_n$.
Then $\singpointnt_n\subset \Reg(C_n)$ if and only if $s_n$ is an isomorphism.
In that case, $\bar C_\gamma$ is regular for all $\gamma\in\Sigma_{\singpointnt_n}$.
\proof
The proposition follows from Lemma \ref{lem:snexceptionallocus}.
\endproof
\end{prop}

Recall from \ref{construction} that $C_{\gamma}\cap C_{\gamma'}=C_{0,\gamma}\cap C_{0,\gamma'}$ for any $\gamma,\gamma'\in\Sigma_n\sqcup\{0\}$, $\gamma\neq\gamma'$. We now want to show this fact is true for $\Sigma_{n+1}$ as well.

\begin{prop}\label{prop:disjointbarcharts}
For any $\gamma,\gamma'\in\Sigma_{n+1}\sqcup\{0\}$, if $\gamma\neq\gamma'$, then \[C_\gamma\cap C_{\gamma'}=C_{0,\gamma}\cap C_{0,\gamma'}.\]
\proof
Let $\gamma,\gamma'\in\Sigma_{n+1}\sqcup\{0\}$, $\gamma\neq\gamma'$. Recall $s_n^{-1}(C_0)=C_0$ and that $s_n$ restricts to the identity $C_0\rightarrow C_0$. Hence it suffices to show that 
\[s_n(C_\gamma)\cap s_n(C_{\gamma'})= s_n(C_{0,\gamma})\cap s_n(C_{0,\gamma'}).\]

Consider the open $D(\tilde{\m G}_{\singpointnt_n})\subseteq C_\alpha$ and let $U_{\alpha,\singpointnt_n}=D(\tilde{\m G}_{\singpointnt_n})\cap C_{0,\alpha}$. If both $\gamma$ and $\gamma'$ belong to $\tilde\Sigma_n\sqcup\{0\}$, we can conclude by the hypothesis on $C_n$ (see \ref{construction}), since $s_{n}(C_{\tilde\alpha})=D(\tilde{\m G}_{\singpointnt_n})$ and $s_n(C_{0,\tilde\alpha})=U_{\alpha,\singpointnt_n}$ by Lemma \ref{lem:Ctildealpha}.
Then assume $\gamma\in\Sigma_p$ for some $p\in \singpointnt_n$. Lemma \ref{lem:spstructure} shows that $s_n(C_\gamma)=C_{\alpha,p}$ and $s_n(C_{0,\gamma})=U_{\alpha,p}$. 
If $\gamma'\in\Sigma_p$ as well, then $C_{0,\gamma}=C_{\gamma}\cap C_{\gamma'}=C_{0,\gamma'}$ from Remark \ref{rem:Toricschemedisjointcharts}. 
If $\gamma'\in\Sigma_{q}$ for some $q\in\singpointnt_n$, $q\neq p$, then \[s_n(C_\gamma)\cap s_n(C_{\gamma'})=C_{\alpha,p}\cap C_{\alpha,q}=U_{\alpha,p}\cap U_{\alpha,q}=s_n(C_{0,\gamma})\cap s_n(C_{0,\gamma'}).\] 
If $\gamma'=\tilde\alpha$, then \[s_n(C_\gamma)\cap s_n(C_{\tilde\alpha})=C_{\alpha,p}\cap D(\tilde{\m G}_{\singpointnt_n})=U_{\alpha,p}\cap U_{\alpha,\singpointnt_n}=s_n(C_{0,\gamma})\cap s_n(C_{0,\tilde\alpha}).\]
Finally, suppose $\gamma'\in\hat\Sigma_n\sqcup\{0\}$. Note that $U_{\alpha,p}=C_{\alpha,p}\cap C_{0,\alpha}$. Then 
\[s_n(C_\gamma)\cap s_n(C_{\gamma'})=C_{\alpha,p}\cap C_{\gamma'}=U_{\alpha,p}\cap C_{0,\gamma'}=s_n(C_{0,\gamma})\cap s_n(C_{0,\gamma'}),\] 
as $C_\alpha\cap C_{\gamma'}=C_{0,\alpha}\cap C_{0,\gamma'}$ from the inductive hypothesis on $C_n$.
\endproof
\end{prop}


\section{A generalised Baker's model}\label{sec:GenBakersthm}

Let $k$ be an algebraically closed field. Let $f\in k[x_1^{\pm 1},y^{\pm 1}]$ be a Laurent polynomial defining a smooth curve $C_0:f=0$ over $\G_m^2$. Let $\Delta$ be the Newton polygon of $f$ and let $C_1$ be the completion of $C_0$ with respect to $\Delta$.

\begin{defn}\label{defn:simpleBR}
Let $C_0$ and $C_1$ as above. A \textit{simple Baker's resolution of $C_0$} is a sequence of proper birational morphisms of $k$-schemes
\begin{equation}\label{eqn:simpleBR}
    \dots\stackrel{s_{n+1}}{\twoheadrightarrow}C_{n+1}\stackrel{s_n}{\twoheadrightarrow}C_{n}\stackrel{s_{n-1}}{\twoheadrightarrow}\dots\stackrel{s_1}{\twoheadrightarrow}C_1
\end{equation}
where the curves $C_n/k$ are constructed from subsets $\Sigma_n\subset\Omega$ as described in \ref{construction} and the maps $s_n$ are the morphisms resolving sets $\singpointnt_n\subseteq\Sing(\bar C_\alpha)$, for some $\alpha\in\Sigma_n$. 
\end{defn}

We have showed how to construct simple Baker's resolutions of $C_0$ recursively for any choice of sets $\singpointnt_n\subseteq\Sing(\bar C_\alpha)$, $\alpha\in\Sigma_n$. 
We want
to prove that for any simple Baker's resolution of $C_0$, the sets $\singpointnt_n$ are eventually empty. Thus simple Baker's resolutions can be used to desingularise $C_1$.

\begin{defn}\label{defn:GeneralisedBakerModelAlgebraicallyClosed}
Recall $k$ is supposed algebraically closed. Let $C/k$ be a smooth projective curve.
A smooth curve $\tilde C/k$ is a \textit{generalised Baker's model} of $C$ if there exist a smooth curve $\tilde C_0\subset\G_m^2$, birational to $C$, and a simple Baker's resolution 
\[\dots\stackrel{s_{n+1}}{\twoheadrightarrow}C_{n+1}\stackrel{s_n}{\twoheadrightarrow}C_{n}\stackrel{s_{n-1}}{\twoheadrightarrow}\dots\stackrel{s_1}{\twoheadrightarrow}C_1\]
of $\tilde C_0$ so that $\tilde C=C_n$ for some $n\in\Z_+$. In this case we say that $\tilde C$ is a generalised Baker's model of $C$ \textit{with respect to $\tilde C_0$}. Note that $\tilde C$ is a \textit{model of $C$ over $k$}, i.e.\ $\tilde C\simeq C$, by Lemma \ref{lem:birationalequalinclusion}.
\end{defn}

For the remainder of the section we fix a simple Baker's resolution of $C_0$
\[\dots\stackrel{s_{n+1}}{\twoheadrightarrow}C_{n+1}\stackrel{s_n}{\twoheadrightarrow}C_{n}\stackrel{s_{n-1}}{\twoheadrightarrow}\dots\stackrel{s_1}{\twoheadrightarrow}C_1\]
where the maps $s_n$ are the morphisms resolving $\singpointnt_n\subseteq\Sing(\bar C_\alpha)$, $\alpha\in\Sigma_n$.

\begin{thm}\label{thm:eventuallyregular}
There exists $h\in\Z_+$ such that $\singpointnt_n\subset\Reg(C_n)$ for all $n\geq h$.
\proof
Let $n\in\Z_+$ and consider $s_{n}:C_{n+1}\twoheadrightarrow C_{n}$ resolving $\singpointnt_n$. As birational morphism between projective curves, $s_n$ is finite (\cite[Lemma 7.3.10]{Liu}).
By Theorem \ref{thm:Bakercompleteness} we have an exact sequence of sheaves
\[0\rightarrow \O_{C_{n}}\rightarrow s_{n}^\ast \O_{C_{n+1}} \rightarrow \m S_n\rightarrow 0,\]
where $\m S_n$ is a skyscraper sheaf with support contained in $\singpointnt_n$. Denote the arithmetic genus of a curve $X/k$ by $p_a(X)$. Then we get
\begin{equation}\label{eqn:arithgenus}
p_a(C_{n+1})=p_a(C_{n})-\dim_k H^0(C_{n},\m S_n).
\end{equation}
Let $r$ be the number of irreducible components of $C_0$. 
For any $n$, recall there is a natural open immersion $C_0\hookrightarrow C_n$ with dense image. Therefore the curve $C_n$ is reduced and has $r$ irreducible components $X_1,\dots,X_r$. Let $i=1,\dots,r$ and let $X_i'$ be the normalisation of $X_i$. Then $H^0(X_i', \O_{X_i'})=k$ as $k$ is algebraically closed (\cite[Corollary 3.3.21]{Liu}). Hence $p_a(X_i')\geq 0$. Therefore $p_a(C_n)\geq 1-r$ by \cite[Proposition 7.5.4]{Liu}. It follows from (\ref{eqn:arithgenus}) that $(p_a(C_{n}))_{n\in\Z_+}$ is a decreasing sequence in $\Z$ bounded below by $1-r$. Hence it is eventually constant, i.e.\ there exists $h\in\Z_+$ such that $p_a(C_{n+1})=p_a(C_n)$ for all $n\geq h$. From (\ref{eqn:arithgenus}), we have $H^0(C_{n},\m S_n)=0$, that implies $\m S_n=0$, as it is a skyscraper sheaf. Hence $\O_{C_n}\simeq s_{n}^\ast \O_{C_{n+1}}$. It follows that $s_n$ is an isomorphism since it is affine. Thus Lemma \ref{prop:nonsingularisomorphism} shows that $\singpointnt_n\subset\Reg(C_n)$ for any $n\geq h$.
\endproof
\end{thm}

\begin{rem}\label{rem:singsetforregularity}
Let $n\in\Z_+$. In Remark \ref{rem:nonsingularbarjacobian} we noticed that any singular point of $C_n$ is the image of a point in $\Sing(\bar C_\alpha)$ under the immersion $\bar C_\alpha\hookrightarrow C_n$, for some $\alpha\in\Sigma_n$. Therefore if $C_n$ is singular, we can choose $\singpointnt_n\subseteq \Sing(\bar C_\alpha)$, $\alpha\in\Sigma_n$, such that $\singpointnt_n\cap\Sing(C_n)\neq\varnothing$.
\end{rem}

\begin{thm}\label{thm:Sequenceregularity}
Let $N=\{n\in\Z_+\mid C_n \text{ is singular}\}$. Suppose $\singpointnt_n\cap\Sing(C_n)\neq\varnothing$  for all $n\in N$. Then $N$ is finite. In other words, there exists $h\in\Z_+$ so that $C_n$ is regular for all $n\geq h$. In particular, for any $n\geq h$, the curve $C_n$ is a generalised Baker's model of the smooth completion of $C_0$.
\proof
The result follows from Theorem \ref{thm:eventuallyregular}. 
\endproof
\end{thm}

\begin{rem}
The arithmetic genus of the curve $C_1$ is $p_a(C_1)=|\Delta(\Z)|$, where $|\Delta(\Z)|$ is the number of internal integer points of the Newton polygon of $C_0$ (\cite[Remark 2.6(d)]{Dok}). Therefore it can be explicitly computed. Equation (\ref{eqn:arithgenus}) gives a recursive way to calculate the arithmetic genus of the following curves $C_n$. 

By choosing the sets $\singpointnt_n$ as in Theorem \ref{thm:Sequenceregularity}, we would eventually compute the genus $g$ of the smooth completion of $C_0$.
Furthermore, if $h\in\Z_+$ is as in Theorem \ref{thm:Sequenceregularity} then $g\leq |\Delta(\Z)|-h$. Hence the number of steps needed to desingularise $C_1$ via a simple Baker's resolution is $\leq|\Delta(\Z)|$.
\end{rem}

\begin{lem}\label{lem:star}
For any $n\in\Z_+$, 
\[C_n\setminus C_0=\bigsqcup_{\gamma\in\Sigma_n} C_\gamma\setminus C_{0,\gamma}=\bigsqcup_{\gamma\in\Sigma_n}\bar C_\gamma.\]
\proof
From \ref{construction}, for any $\gamma,\gamma'\in\Sigma_n\sqcup\{0\}$, one has $C_{\gamma}\cap C_{\gamma'}=C_{0,\gamma}\cap C_{0,\gamma'}$. This implies that if $\gamma\in\Sigma_n$ then $C_{\gamma}\cap C_0=C_{0,\gamma}$ and $C_\gamma\cap C_{\gamma'}\subseteq C_0$ for every $\gamma'\in\Sigma_n$, $\gamma'\neq\gamma$. The lemma follows.
\endproof
\end{lem}

\begin{thm}\label{thm:eventually_pn_in_C0}
There exists $h\in\Z_+$ such that $\singpointnt_n=\varnothing$ for all $n\geq h$.
\proof
By Theorem \ref{thm:eventuallyregular} there exists $h'\in\Z_+$ such that $\singpointnt_n\subset\Reg(C_n)$ for all $n\geq h'$. Let $n\geq h'$. For any $\Gamma\subseteq\Sigma_n$ let $N(\Gamma)$ be the number of points of $C_n\setminus C_0$ which are singular on $\bar C_\gamma$ for some $\gamma\in\Gamma$. Note that by Lemma \ref{lem:star}, one has $N(\Gamma)=\sum_{\gamma\in\Gamma}N(\{\gamma\})$.

Let $\alpha\in\Sigma_n$ such that $\singpointnt_n\subseteq\Sing(\bar C_\alpha)$. Since $C_{\tilde\alpha}$ embeds in $C_\alpha$ via $s_{n}$ and $\singpointnt_n=\bar C_{\alpha}\setminus s_{n}(\bar C_{\tilde\alpha})$ by Lemma \ref{lem:Ctildealpha}, we have $N(\tilde\Sigma_n)=N(\Sigma_n)-|\singpointnt_n|$. 
On the other hand, $N(\Sigma_{\singpointnt_n})=0$ by Proposition \ref{prop:nonsingularisomorphism}, as $\singpointnt_n\subset\Reg(C_n)$. Hence
\[N(\Sigma_{n+1})=N(\Sigma_{\singpointnt_n})+N(\tilde\Sigma_n)=N(\Sigma_n)-|\singpointnt_n|\leq N(\Sigma_n).\]

Then $N(\Sigma_n)_{n\geq h'}$ forms a decreasing sequence bounded below by $0$. Thus it is eventually constant, i.e.\ there exists $h\in\Z_+$ such that $N(\Sigma_{n+1})=N(\Sigma_{n})$ for all $n\geq h$.
But we saw above that this happens only if $\singpointnt_n=\varnothing$.
\endproof
\end{thm}

\begin{defn}\label{defn:outerregular}
For any $n\in\Z_+$, the curve $C_n$ is said \textit{outer regular} if $\bar C_\gamma$ is regular for any $\gamma\in\Sigma_n$. In other words, $C_n$ is outer regular if the closed subset $C_n\setminus C_0$ of $C_n$, equipped with the structure of closed subscheme coming from Lemma \ref{lem:star}, is regular.
\end{defn}

Note that from Remark \ref{rem:nonsingularbarjacobian}, if $C_n$ is outer regular, then it is regular.

\begin{thm}\label{thm:secondmodel}
Suppose $\singpointnt_n\neq\varnothing$ for all $n\in\Z_+$ such that $C_n$ is not outer regular. Then there exists $h\in\Z_+$ so that for all $n\geq h$ the closed subschemes $\bar C_\gamma$ are regular for all $\gamma\in\Sigma_n$. 
In particular, the curve $C_h$ is an outer regular generalised Baker's model of the smooth completion of $C_0$.
\proof
The result follows from Theorem \ref{thm:eventually_pn_in_C0}.
\endproof
\end{thm}

\begin{cor}
Every smooth projective curve defined over an algebraically closed field $k$ admits an outer regular generalised Baker's model.
\proof
By Corollary \ref{cor:smoothopenmodel}, for any smooth projective curve $C$ there exists a smooth curve $C_0\subset \G_m^2$ birational to $C$. Construct a simple Baker's resolution (\ref{eqn:simpleBR}) of $C_0$ recursively by choosing $\singpointnt_n\neq\varnothing$ whenever $C_n$ is not outer regular. Theorem \ref{thm:secondmodel} concludes the proof.
\endproof
\end{cor}

\begin{lem}\label{lem:regularpointsofbarCgamma}
Let $n\in\Z_+$. For any $\gamma\in\Sigma_n$ we have a natural bijection
\[\Reg(\bar C_\gamma)\xleftrightarrow{1:1}\{\text{simple roots of $f|_\gamma$ in $k^\times$}\}.\]
\proof
For any $\gamma\in\Sigma_n$, we have
\[\Reg\Big(\Spec\sfrac{k[X_{j_\gamma}^{\pm 1}]}{(f|_\gamma)}\Big)\xleftrightarrow{1:1}\{\text{simple roots of $f|_\gamma$ in $k^\times$}\}.\]
We will prove by induction on $n$ that $\Reg(\bar C_\gamma)=\Reg\big(\Spec k[X_{j_\gamma}^{\pm 1}]/(f|_\gamma)\big)$. If $n=1$, the statement follows since $D_\gamma=k[X_{j_\gamma}^{\pm 1}]$ for all $\gamma\in\Sigma_1$. Suppose $n>0$ and $\gamma\in\Sigma_{n+1}$. 
Let $s_n:C_{n+1}\rightarrow C_n$ be the morphism resolving $\singpointnt_n\subseteq\Sing(\bar C_\alpha)$, for $\alpha\in\Sigma_n$.
By Definition \ref{defn:Sigman} the result follows from the inductive hypothesis except when either $\gamma=\tilde\alpha$ or $\gamma\in\Sigma_{\singpointnt_n}$. If $\gamma\in\Sigma_{\singpointnt_n}$, then $D_\gamma=k[X_{j_\gamma}^{\pm 1}]$ by Lemma \ref{lem:Dgamma}, so $\bar C_\gamma=\Spec k[X_{j_\gamma}^{\pm 1}]/(f|_\gamma)$. If $\gamma=\tilde\alpha$, then $\bar C_{\gamma}=\bar C_{\alpha}\setminus \singpointnt_n$ by Lemma \ref{lem:Ctildealpha}. Then $\Reg(\bar C_{\gamma})=\Reg(\bar C_\alpha)$. 
Thus
$\Reg(\bar C_{\gamma})=\Reg\big(\Spec k[X_{j_{\gamma}}^{\pm 1}]/(f|_{\gamma})\big)$ since $j_{\tilde\alpha}=j_\alpha$ and $f|_{\tilde\alpha}=f|_\alpha$.
\endproof
\end{lem}

\begin{thm}\label{thm:pointsbijection}
Let $n\in\Z_+$. Suppose $C_n$ is outer regular.
Then we have a natural bijection
\[C_n(k)\setminus C_0(k)\xleftrightarrow{1:1}\bigsqcup_{\gamma\in\Sigma_n}\{\text{simple roots of $f|_\gamma$ in $k^\times$}\}.\]
\proof
Lemma \ref{lem:star} shows that $C_n\setminus C_0=\bigsqcup_{\gamma\in\Sigma_n}\bar C_\gamma$. 
Thus Lemma \ref{lem:regularpointsofbarCgamma} concludes the proof.
\endproof
\end{thm}

We conclude the section with the following two lemmas, proving that for any $n\in\Z_+$ the unions in Definition \ref{defn:Sigman} are all disjoint. This fact is particularly useful in applications: together with Proposition \ref{prop:disjointbarcharts} it implies the points in $C_{\gamma}\setminus C_{0,\gamma}$ for $\gamma\in\Sigma_{\singpointnt_n}\cup\{\tilde\alpha\}$ are not visible on $\hat C_n$.

Recall the partial order $<$ on $\Omega$ given in Definition \ref{defn:Omegaposet}.

\begin{lem}\label{lem:noinequalityinSigman}
Let $n\in\Z_+$. For any $\gamma,\gamma'\in\Sigma_n$, neither $\gamma<\gamma'$ nor $\gamma'>\gamma$.
\proof
We are going to prove the lemma by induction on $n$. If $n=1$ the result is trivial. Suppose $n>0$ and let $\alpha\in\Sigma_n$ such that $\singpointnt_n\subseteq\Sing(\bar C_\alpha)$. Suppose by contradiction there exist $\gamma,\gamma'\in\Sigma_{n+1}$ such that $\gamma<\gamma'$. By definition $\Sigma_{n+1}=\tilde\Sigma_{\singpointnt_n}\cup\hat\Sigma_n$,
where $\tilde\Sigma_{\singpointnt_n}=\Sigma_{\singpointnt_n}\cup\{\tilde\alpha\}$. Let $m\in\Z_+$ such that $\alpha\in\Omega_m$. Then $\alpha'\in\Omega_{m+1}$ for any $\alpha'\in\tilde\Sigma_{\singpointnt_n}$. In particular, $\gamma$ and $\gamma'$ cannot be both in $\tilde\Sigma_{\singpointnt_n}$. In fact, by inductive hypothesis, either $\gamma\in \tilde\Sigma_{\singpointnt_n}$ and $\gamma'\in \hat\Sigma_{n}$ or viceversa. Suppose $\gamma\in\tilde\Sigma_{\singpointnt_n}$. Then $\alpha<\gamma<\gamma'$. But this gives a contradiction since $\alpha,\gamma'\in\Sigma_n$. Suppose $\gamma'\in\tilde\Sigma_{\singpointnt_n}$. Then $\alpha$ is the unique element of $\Omega_m$ such that $\alpha<\gamma'$. In particular, $\gamma\leq \alpha$. But $\gamma\neq \alpha$ since $\gamma\in\hat\Sigma_n$. Thus $\gamma<\alpha$, contradicting the inductive hypothesis on $\Sigma_n$.
\endproof
\end{lem}

\begin{lem}\label{lem:SigmaDisjointUnion}
Let $n\in\Z_+$ and let $\alpha\in\Sigma_n$ such that $\singpointnt_n\subseteq\Sing(\bar C_{\alpha})$.
Then the sets $\hat\Sigma_n$, $\{\tilde\alpha\}$, and $\Sigma_p$, for $p\in \singpointnt_n$, are pairwise disjoint.
\proof
Let $p\in\singpointnt_n$. First note that for every $\gamma\in\Sigma_p$ and $\gamma'\in\bigcup_{q\in\singpointnt_n\setminus\{p\}}\Sigma_q\cup\{\tilde\alpha\}$, the images of $C_\gamma$ and $C_{\gamma'}$ under $s_n$ are different. Then $\gamma\neq\gamma'$. It remains to show that if $\gamma\in\Sigma_p\cup\{\tilde\alpha\}$ and $\alpha'\in\hat\Sigma_n$, then $\gamma\neq\alpha'$. Note that $\gamma>\alpha$. Therefore if $\gamma=\alpha'$ then $\alpha'>\alpha$, where both $\alpha$ and $\alpha'$ are elements of $\Sigma_n$. But this is not possible by Lemma \ref{lem:noinequalityinSigman}.
\endproof
\end{lem}

\section{Simultaneous resolution of different charts}\label{sec:multipleSing}
Let $k$ be an algebraically closed field and let $f\in k[x_1^{\pm 1},y^{\pm 1}]$ be a Laurent polynomial defining a smooth curve $C_0:f=0$ over $\G_m^2$. Let $C_1$ be the completion of $C_0$ with respect to its Newton polygon.
In the previous sections we showed that we can construct a sequence of proper birational morphisms
\[\dots\stackrel{s_{n+1}}{\twoheadrightarrow}C_{n+1}\stackrel{s_n}{\twoheadrightarrow}C_{n}\stackrel{s_{n-1}}{\twoheadrightarrow}\dots\stackrel{s_1}{\twoheadrightarrow}C_1,\]
where the curves $C_n/k$ are constructed from sets $\Sigma_n\subseteq\Omega$ as described in \ref{construction} and the maps $s_n$ are the morphisms resolving $\singpointnt_n\subset\Sing(\bar C_{\alpha_n})$ for $\alpha_n\in\Sigma_n$. 

Let $n\in\Z_+$. 
Note that, once we have chosen the polynomials $\tilde {\m G}_p$ for any $p\in\singpointnt_n$, the construction of $\Sigma_{n+1}\setminus\hat\Sigma_n$ only depends on $\alpha_n$ and $\singpointnt_n$ by Lemma \ref{lem:SigmaDisjointUnion}. 
Suppose $\alpha_{n+1}\in\hat\Sigma_n$. 
Then
\[\Sigma_{n+2}=\Sigma_{\singpointnt_n}\cup \Sigma_{\singpointnt_{n+1}}\cup\{\tilde\alpha_{n},\tilde\alpha_{n+1}\}\cup(\Sigma_n\setminus\{\alpha_n,\alpha_{n+1}\}).\] 
Thus $\Sigma_{n+2}$ would have been defined in the same way if, instead of resolving $\singpointnt_n$ first and then $\singpointnt_{n+1}$, we had resolved $\singpointnt_{n+1}$ first and then $\singpointnt_{n}$. In other words, the construction of $\Sigma_{n+2}$, and so of $C_{n+2}$, from $\Sigma_n$ does not depend on the order of resolution of $\singpointnt_n$ and $\singpointnt_{n+1}$.

In this section we will show that from our construction we can \textit{resolve} points coming from different charts simultaneously. More precisely, we will explain how to construct a sequence as in \S\ref{sec:Bakerdescription} where the morphisms $s_n$ resolve finite sets of points $\singpointnt_n\subseteq\bigsqcup_{\alpha\in\Sigma_n}\Sing(\bar C_\alpha)$. Note that by Lemma \ref{lem:star} we can identify the points in $\singpointnt_n$ with points of $C_n$ via the immersions $\bar C_\alpha\hookrightarrow C_n$.

Suppose that, for some $n\in\Z_+$, we have constructed $\Sigma_n\subset\Omega$ and $C_n$ as in \ref{construction}. Let $\singpointnt_n\subseteq\bigsqcup_{\alpha\in\Sigma_n}\Sing(\bar C_\alpha)$. Denote $\singpointnt_{n,\alpha}=\singpointnt_n\cap\bar C_{\alpha}$ for any $\alpha\in\Sigma_n$. 
Consider the subset $\Sigma_{n,\singpointnt_n}:=\{\alpha\in\Sigma_n\mid \singpointnt_{n,\alpha}\neq\varnothing\}$ of $\Sigma_n$ and order its elements $\alpha_{0},\alpha_{1},\dots,\alpha_{h}$. 
For each $i=0,\dots,h$ we can recursively construct the morphism $s_{n+i}: C_{n+i+1}\rightarrow C_{n+i}$ resolving $\singpointnt_{n,\alpha_{i}}\subseteq\Sing(\bar C_{\alpha_{i}})$ as described in \S\ref{sec:Bakerconstruction}. 
Indeed $\alpha_{0}\in\Sigma_n$ and $\alpha_{i}\in\Sigma_{n+i}$ since \[\alpha_{i}\in\Sigma_n\setminus\{\alpha_{0},\dots,\alpha_{i-1}\}\subseteq\hat\Sigma_{n+i-1}\qquad \text{for any $i\geq 1$}.\]
Therefore from the observation made at the beginning of the section
\begin{align*}
\Sigma_{n+h+1}&=\bigcup_{i=0}^h\Sigma_{\singpointnt_{n,\alpha_{i}}}\cup\{\tilde\alpha_{0},\dots,\tilde\alpha_{h}\}\cup(\Sigma_n\setminus\{\alpha_{0},\dots,\alpha_{h}\})\\
&=\bigcup_{\alpha\in\Sigma_{n,\singpointnt_n}}(\Sigma_{\singpointnt_{n,\alpha}}\cup\{\tilde\alpha\})\cup(\Sigma_n\setminus\Sigma_{n,\singpointnt_n}).
\end{align*}
In particular, $\Sigma_{n+h+1}$ is independent of the order chosen for the elements in $\Sigma_{n,\singpointnt_n}$.
This approach eventually constructs a complete curve $C_{n+h+1}$ and a surjective birational morphism
\[C_{n+h+1}\xrightarrow{s_{n+h}\circ s_{n+h-1}\circ\dots\circ s_{n}}C_n,\]
with exceptional locus equal to the inverse image of $\singpointnt_n\cap\Sing(C_n)$. This morphism does not depend on the order chosen for the elements $\alpha_{i}$ of $\Sigma_{n,\singpointnt_n}$. Indeed by Theorem \ref{thm:sndefn} it is the unique morphism extending the birational maps $s_{\gamma\alpha}:C_\gamma\--> C_\alpha$ for $\gamma\in\Sigma_{n+h+1}\sqcup\{0\}$ and $\alpha\in\Sigma_n\sqcup\{0\}$.

\begin{defn}
We will say that $s_{n+h}\circ\dots\circ s_{n}$ is the morphism \textit{resolving} the finite set $\singpointnt_n\subseteq\bigsqcup_{\alpha\in\Sigma_n}\Sing(\bar C_\alpha)$.
\end{defn}

We can then redefine $\Sigma_{n+1}:=\Sigma_{n+h+1}$ and $C_{n+1}:=C_{n+h+1}$ to see that we can construct finite subsets $\Sigma_n\subset\Omega$ and projective curves $C_n/k$ as described in \ref{construction} and a sequence of proper birational morphisms
\[\dots\stackrel{s_{n+1}}{\twoheadrightarrow}C_{n+1}\stackrel{s_n}{\twoheadrightarrow}C_{n}\stackrel{s_{n-1}}{\twoheadrightarrow}\dots\stackrel{s_1}{\twoheadrightarrow}C_1,\]
where the maps $s_n:C_{n+1}\rightarrow C_n$ are the morphisms resolving freely chosen sets $\singpointnt_n\subseteq\bigsqcup_{\alpha\in\Sigma_n}\Sing(\bar C_\alpha)$.

\begin{defn}\label{defn:BakerResolution}
Let $C_0$ and $C_1$ as above. A \textit{Baker's resolution of $C_0$} is
a sequence of proper birational morphisms of $k$-schemes 
\[
    \dots\stackrel{s_{n+1}}{\twoheadrightarrow}C_{n+1}\stackrel{s_n}{\twoheadrightarrow}C_{n}\stackrel{s_{n-1}}{\twoheadrightarrow}\dots\stackrel{s_1}{\twoheadrightarrow}C_1
\]
where the curves $C_n/k$ are constructed from subsets $\Sigma_n\subset\Omega$ as indicated in \ref{construction} and the maps $s_n$ are the morphisms resolving sets $\singpointnt_n\subseteq\bigsqcup_{\alpha\in\Sigma_n}\Sing(\bar C_\alpha)$. 
\end{defn}

Simple Baker's resolutions are Baker's resolution. In fact, from what discussed in this section, Baker's resolutions of $C_0$ are just contraptions of simple Baker's resolutions.
Hence the results in \S\ref{sec:GenBakersthm} extends to Baker's resolutions. Let us explicitly restate Theorem \ref{thm:eventually_pn_in_C0} in light of the terminology introduced in the current section as an example.

\begin{thm}\label{thm:multiplesingmodel}
For any Baker's resolution of $C_0$ given as in Definition \ref{defn:BakerResolution}, 
there exists $h\in\Z_+$ such that $\singpointnt_n=\varnothing$ for any $n\geq h$.
\end{thm}

Baker's resolutions are not really a new concept, but rather a more general point of view which will be useful in the next section, where we tackle the case of a non-algebraically closed base field.

\section{The case of non-algebraically closed base field}\label{sec:GaloisBaker}
In this section let $k$ be a perfect field with algebraic closure $\bar k$. Denote by $G_k$ the absolute Galois group $\Gal(\bar k/k)$. Let $f\in k[x_1^{\pm 1},y^{\pm 1}]$ such that $C_{0,k}: f=0$ is a smooth curve defined over $\G_{m,k}^2$. Set $C_0=C_{0,k}\times_k \bar k$. In the previous section we showed how to construct a sequence of proper birational morphisms of $\bar k$-schemes
\[\dots\xrightarrow{s_{n+1}}C_{n+1}\xrightarrow{s_n}C_{n}\xrightarrow{s_{n-1}}\dots\xrightarrow{s_1}C_1,\]
called Baker's resolution of $C_0$, where the curves $C_n/\bar k$ are equipped with canonical open immersions $\iota_n:C_0\hookrightarrow C_n$ such that $s_{n}\circ\iota_{n+1}=\iota_n$. Suppose that for any $n\in\Z_+$ one has $G_k\subseteq \mathrm{Aut}(C_n)$ and $s_n\circ\sigma=\sigma\circ s_n$ for all $\sigma\in G_k$. Then, from the universal property of quotient schemes, one has an induced sequence of proper birational morphisms of $k$-schemes
\[\dots\xrightarrow{s_{n+1,k}} C_{n+1,k}\xrightarrow{s_{n,k}} C_{n,k} \xrightarrow{s_{n-1,k}}\dots\xrightarrow{s_{1,k}} C_{1,k}.\]
where the curves $C_{n,k}:=C_n/G_k$ are defined over $k$. Furthermore, the morphisms $\iota_n$ induce open immersions $\iota_{n,k}:C_{0,k}\hookrightarrow C_{n,k}$ such that $s_{n,k}\circ\iota_{n+1,k}=\iota_{n,k}$.
In fact, $C_n\simeq C_{n,k}\times_k \bar k$ and the quotient morphism $C_n\rightarrow C_{n,k}$ is the canonical projection. Then $C_n$ is smooth if and only if so is $C_{n,k}$. 

The argument above motivates the subject of this section, which is constructing a Baker's resolution of $C_0$ such that $G_k\subseteq \mathrm{Aut}(C_n)$ and $s_n$ is Galois-invariant for any $n\in\Z_+$. 
The following definition extends Definitions \ref{defn:GeneralisedBakerModelAlgebraicallyClosed}, \ref{defn:outerregular} to the case of general perfect fields.

\begin{defn}\label{defn:GeneralisedBakerModel}
Let $C/k$ be a smooth projective curve. A curve $\tilde C/k$ is a \textit{generalised Baker's model} of $C$ if $\tilde C\simeq C$ and there exists a smooth curve $\tilde C_{0,k}/k$ such that the base extended curve $\tilde C\times_k\bar k$ is a generalised Baker's model of $C\times_k\bar k$ with respect to $\tilde C_{0,k}\times_k\bar k$.
Furthermore, a generalised Baker's model $\tilde C$ of $C$ is \textit{outer regular} if $\tilde C\times_k\bar k$ is outer regular.
\end{defn}

Let us first describe a group action of $G_k$ on $\Omega$. Let $\alpha=(v,\tuplent)\in\Omega$ and $\sigma\in G_k$. Let $m\in\Z_+$ such that $\alpha\in\Omega_m$ and write $\tuplent=(g_2,\dots,g_m)$ for Laurent polynomials $g_i\in\bar k[x_1^{\pm 1},y^{\pm 1}]$. Set $\tuplent^\sigma=(g_2^\sigma,\dots,g_m^\sigma)$ and define $(v,\tuplent)^\sigma=(v,\tuplent^\sigma)$. Recall
\[C_{0,\alpha}=\Spec\frac{\bar k[x_1^{\pm 1},\dots, x_m^{\pm 1}, y^{\pm 1}]}{(f_1,f_2,\dots,f_m)}\]
with $f_1=f\in k[x_1^{\pm 1},y^{\pm 1}]$ and $f_i=x_i-g_i\in \bar k[x_1^{\pm 1},y^{\pm 1}]$ for $i\geq 2$. Hence $C_{0,\alpha^\sigma}=C_{0,\alpha}^\sigma$. Then $C_{0,\alpha^\sigma}$ is dense in $C_0$ and so $\alpha^\sigma\in\Omega$. Thus the element $\alpha^\sigma$ is set as the image of $\alpha$ under the action of $\sigma$.
The next lemma follows.

\begin{lem}\label{lem:Galoisconjgamma}
Let $\sigma\in G_k$ and $\alpha\in\Omega$. Let $m\in\Z_+$ such that $\alpha\in\Omega_m$. If $\gamma=\beta\circ_{g}\alpha$, for some primitive vector $\beta\in\N\times\Z_+$ and $g\in \bar k[x_1,\dots,x_m,y]$ then $\gamma^\sigma=\beta\circ_{g^\sigma}\alpha^\sigma$.
\end{lem}

We will show that if the morphisms $s_n$ resolve Galois-invariant sets of points for any $n\in\Z_+$, the curves $C_n$ can be constructed from subsets $\Sigma_n\subset\Omega$ with the properties of \ref{construction} and the following additional one:
\begin{enumerate}[label=(d)]
    \item 
    The action of $G_k$ on $\Omega$ restricts to $\Sigma_n$.
    Furthermore, for any $\sigma\in G_k$ and any $\alpha\in\Sigma_n$, we have $M_{\alpha^\sigma}=M_\alpha$, $j_{\alpha^\sigma}=j_\alpha$, $\m{F}_{\alpha^\sigma}=\m{F}_\alpha^\sigma$. \label{eqn:Bakeradditionalprop}
\end{enumerate}
In particular, if \ref{eqn:Bakeradditionalprop} holds for $n\in\Z_+$, then $G_k\subseteq\mathrm{Aut}(C_n)$. 

Suppose the set $\Sigma_n$ defining $C_n$ satisfies the additional property \ref{eqn:Bakeradditionalprop}.
Let $\alpha\in \Sigma_n$ and let $m\in\Z_+$ such that $\alpha\in\Omega_m$. 
Let $\sigma\in G_k$. 
From \ref{eqn:Bakeradditionalprop} it follows that $\alpha^\sigma\in\Sigma_n$ and $\got a_\alpha^\sigma=\got a_{\alpha^\sigma}$, $R_{\alpha^\sigma}=R_\alpha^\sigma$, $C_{\alpha^\sigma}=C_\alpha^\sigma$, $f|_{\alpha^\sigma}=f|_\alpha^\sigma$. Hence $D_{\alpha^\sigma}=D_\alpha^\sigma$ and so $\bar C_{\alpha^\sigma}=\bar C_{\alpha}^\sigma$.

Let $p\in\Sing(\bar C_\alpha)$. Recall $\bar{\m{G}}_p\in \bar k[X_{j_\alpha}]$ is monic of degree $1$ generating the maximal ideal of $\O_{\bar C_\alpha,p}$. Since 
$\bar C_\alpha^\sigma=\bar C_{\alpha^\sigma}$, the ideal $(\bar{\m{G}}_p^\sigma)$ is the maximal ideal of $\O_{\bar C_\alpha^\sigma,p^\sigma}$. Therefore $\bar{\m G}_{p^\sigma}=\bar{\m{G}}_p^\sigma$ as $\bar{\m{G}}_p^\sigma\in \bar k[X_{j_{\alpha^\sigma}}]$ is linear and monic. Finally, the equality $\m F_{\alpha^\sigma}=\m F_{\alpha}^\sigma$ implies that we can choose $\tilde{\m G}_{p^\sigma}=\tilde{\m G}_p^\sigma$. 
Let $g_p\in \bar k[x_1,\dots,x_m,y]$ related to $\tilde{\m{G}}_p$ by $M_{\alpha}$. If $\tilde{\m G}_{p^\sigma}=\tilde{\m G}_p^\sigma$, then $g_p^\sigma$ is the polynomial related to $\tilde{\m G}_{p^\sigma}$ by $M_{\alpha^\sigma}=M_{\alpha}$; hence $g_{p^\sigma}=g_p^\sigma$.

Now let $\singpointnt_n\subseteq\bigsqcup_{\alpha\in\Sigma_n}\Sing(\bar C_\alpha)$ be a $G_k$-invariant set. Consider the morphism $s_n:C_{n+1}\rightarrow C_n$ resolving $\singpointnt_n$.
We want to show that we can construct the collection $\Sigma_{n+1}$ defining $C_{n+1}$ in such a way that it satisfies \ref{eqn:Bakeradditionalprop}. 
Define $\singpointnt_{n,\alpha}=\singpointnt_n\cap \bar C_\alpha$ for any $\alpha\in\Sigma_n$ and $\Sigma_{n,\singpointnt_n}=\{\alpha\in\Sigma_n\mid \singpointnt_{n,\alpha}\neq\varnothing\}$. Note that since $\singpointnt_n$ is $G_k$-invariant, so is $\Sigma_{n,\singpointnt_n}$. Moreover, \[\singpointnt_{n,\alpha}^\sigma=\{p^\sigma\mid p\in\singpointnt_{n,\alpha}\}=\singpointnt_{n,\alpha^\sigma}\] for any $\alpha\in\Sigma_n$ and $\sigma\in G_k$.

Let $\gamma\in\Sigma_{n+1}$ and $\sigma\in G_k$. Assume $\tilde{\m G}_{p^\sigma}=\tilde{\m G}_p^\sigma$ for any $p\in\singpointnt_n$. If $\gamma\notin\Sigma_n$, then for some $\alpha\in\Sigma_{n,\singpointnt_n}$ either $\gamma=\tilde\alpha$ or $\gamma=\beta\circ_{g_p}\alpha$, for some $p\in\singpointnt_{n,\alpha}$, and $\beta\in\Z_+^2$ primitive. It follows from Lemma \ref{lem:Galoisconjgamma} that $\gamma^\sigma$ equals either $\widetilde{\alpha^\sigma}$ or $\beta\circ_{g_{p^\sigma}}\alpha^\sigma$. 
In particular, the matrix $M_{\gamma^\sigma}$, the positive integer $j_{\gamma^\sigma}$ and the polynomial $\m F_{\gamma^\sigma}$ have been defined in \S\ref{subsec:BakerInductiveStep} even when $\gamma^\sigma\notin\Sigma_{n+1}$ (see Notation \ref{nt:Notationforgamma}, \ref{nt:Notationfortildealpha}). This allows us to state the following result.

\begin{thm}\label{thm:Galoisequivariantpolynomials}
Let $s_n:C_{n+1}\rightarrow C_n$ be the morphism resolving the $G_k$-invariant set $\singpointnt_n \subseteq\bigsqcup_{\alpha\in\Sigma_n}\Sing(\bar C_\alpha)$. Suppose $\Sigma_n$ satisfies the additional property \ref{eqn:Bakeradditionalprop}. Then $\Sigma_{n+1}$ satisfies \ref{eqn:Bakeradditionalprop} if for all $\sigma\in G_k$, one has
\begin{enumerate}[label=(\arabic*)]
    \item $\tilde{\m G}_{p^\sigma}=\tilde{\m G}_{p}^\sigma$ for all $p\in \singpointnt_n$; \label{item:Galoistheorem(i)}
    \item $M_\gamma=M_{\gamma^\sigma}$ for any $\gamma\in\Sigma_{n+1}$; \label{item:Galoistheorem(ii)}
    \item $\ord_{X_{j_\gamma}}(\m{F}_\gamma)=\ord_{X_{j_{\gamma^\sigma}}}(\m{F}_{\gamma^\sigma})$ for any $\gamma\in\Sigma_{n+1}$. \label{item:Galoistheorem(iii)}
\end{enumerate}
Furthermore, if $\alpha_1,\dots,\alpha_h\in\Sigma_{n, \singpointnt_n}$ so that $\Sigma_{n, \singpointnt_n}=\bigsqcup_{i=1}^h G_k\,\alpha_i$, then
\[\Sigma_{n+1}=G_k\cdot\bigcup_{i=1}^h(\Sigma_{\singpointnt_{n,\alpha_i}}\cup\{\tilde\alpha_i\})\cup(\Sigma_n\setminus \Sigma_{n,\singpointnt_n}).\]
\proof
Assume \ref{item:Galoistheorem(i)}, \ref{item:Galoistheorem(ii)} and \ref{item:Galoistheorem(iii)} and let $\sigma\in G_k$. Let $\gamma\in\Sigma_{n+1}\setminus\Sigma_n$. Then there exists $\alpha\in\Sigma_{n,\singpointnt_n}$ such that $\gamma=\tilde\alpha$ or $\gamma\in\Sigma_{\singpointnt_{n,\alpha}}$. 

Suppose $\gamma=\tilde\alpha$ for some $\alpha\in\Sigma_{n,\singpointnt_n}$. Then $\gamma^\sigma=\widetilde{\alpha^\sigma}$ by Lemma \ref{lem:Galoisconjgamma} and so $\gamma^\sigma\in\Sigma_{n+1}$. Note that $j_{\gamma}=j_{\gamma^\sigma}$ and $\m F_{\gamma}=\m F_{\gamma^\sigma}$. Indeed, $j_{\tilde\alpha}=j_\alpha$, $\m F_{\tilde\alpha}=\m F_{\alpha}$ by construction, and $j_{\gamma^\sigma}=j_{\alpha^\sigma}=j_\alpha$ and $\m F_{\gamma^\sigma}=\m F_{\alpha^\sigma}=\m F_\alpha^\sigma$, where the last equalities follow from the fact that $\Sigma_n$ satisfies \ref{eqn:Bakeradditionalprop}.

Suppose now that $\gamma\in\Sigma_{\singpointnt_{n,\alpha}}$. Then $\gamma=\beta\circ_{g_p}\alpha$ for some $p\in\singpointnt_{n,\alpha}$ and some primitive vector $\beta\in\Z_+^2$. Lemma \ref{lem:Galoisconjgamma} implies that $\gamma^\sigma=\beta\circ_{g_{p^\sigma}}\alpha^\sigma$. Let $m\in\Z_+$ such that $\alpha\in\Omega_m$.
Note that $j_\gamma=j_{\gamma^\sigma}$. Indeed $j_\gamma=m+1$ by construction, and similarly $j_{\gamma^\sigma}=m+1$ since $\alpha^\sigma\in\Omega_m$.

Now we want to show that $\m{F}_\gamma^\sigma=\m F_{\gamma^\sigma}$.  
Let $\alpha_p=(0,1)\circ_{g_p}\alpha$, as in \S\ref{subsec:BakerInductiveStep}. Then $(\alpha_p)^\sigma=(0,1)\circ_{g_{p^\sigma}}\alpha^\sigma=(\alpha^\sigma)_{p^\sigma}$ by Lemma \ref{lem:Galoisconjgamma}. Since $g_p^\sigma=g_{p^\sigma}$ and $M_\alpha=M_{\alpha^\sigma}$, Remark \ref{rem:Matrixalphap} shows that $M_{\alpha_p}=M_{(\alpha^\sigma)_{p^\sigma}}$.
Recall that the matrix $M_\gamma$ is obtained as the product $(I_m\oplus M_\beta)\cdot M_{\alpha_p}$, for some matrix $M_\beta$ attached to $\beta$. Similarly, $M_{\gamma^\sigma}=(I_m\oplus M_\beta')\cdot M_{(\alpha^\sigma)_{p^\sigma}}$ for some matrix $M_\beta'$ attached to $\beta$. It follows that $M_\beta=M_\beta'$ as we are assuming $M_\gamma=M_{\gamma^\sigma}$.

We recall $\m{F}_\gamma$ and $\m F_{\gamma^\sigma}$ are constructed from $\m{F}_{\alpha,p}$ and $\m F_{\alpha^\sigma,p^\sigma}$ respectively, via the change of variables given by $M_\beta$. Explicitly,
\[\m{F}_{\alpha,p}\stackrel{M_\beta}{=}X_{m+1}^{n_1}Y^{n_2}\cdot\m{F}_\gamma,\qquad\m F_{\alpha^\sigma,p^\sigma}\stackrel{M_\beta}{=}X_{m+1}^{n_3}Y^{n_4}\cdot\m F_{\gamma^\sigma},\]
for some $n_1,n_2,n_3,n_4\in\Z$. Note that $\m F_{\alpha,p}^\sigma=\m F_{\alpha^\sigma, p^\sigma}$ since $\tilde{\m G}_p^\sigma=\tilde{\m G}_{p^\sigma}$ and $\m F_\alpha^\sigma=\m F_{\alpha^\sigma}$. Therefore $\m F_\gamma^\sigma=\m F_{\gamma^\sigma}$ as $\ord_{X_{m+1}}(\m{F}_\gamma)=\ord_{X_{m+1}}(\m{F}_{\gamma^\sigma})$ by assumption and $\ord_{Y}(\m{F}_\gamma)=\ord_{Y}(\m{F}_{\gamma^\sigma})=0$ by construction.

To conclude the proof it only remains to show that $\bar C_{\gamma^\sigma}\neq\varnothing$ since this would imply $\gamma^\sigma\in\Sigma_{\singpointnt_{n,\alpha^\sigma}}$. We showed $j_{\gamma^\sigma}=j_\gamma$, $M_{\gamma^\sigma}=M_\gamma$ and $\m F_{\gamma^\sigma}=\m F_\gamma^\sigma$, and so $\bar C_{\gamma^\sigma}=\bar C_\gamma^\sigma$. But $\bar C_\gamma\neq \varnothing$ since $\gamma\in\Sigma_{\singpointnt_{n,\alpha}}$. Thus $\bar C_{\gamma^\sigma}\neq\varnothing$.
\endproof
\end{thm}

\begin{rem}\label{rem:Galoisconditionscanbeachieved}
Suppose $\Sigma_n$ satisfies \ref{eqn:Bakeradditionalprop}. In this remark we show that conditions \ref{item:Galoistheorem(i)},\ref{item:Galoistheorem(ii)},\ref{item:Galoistheorem(iii)} of Theorem \ref{thm:Galoisequivariantpolynomials} can always be obtained. 

Let $\sigma\in G_k$. We have already observed that we can choose the polynomials $\tilde{\m G}_p$, for $p\in\singpointnt_n$, satisfying \ref{item:Galoistheorem(i)}. Let $\gamma\in\Sigma_{n+1}$. If $\gamma\in\Sigma_n$, then the equalities $M_\gamma=M_{\gamma^\sigma}$ and $\ord_{X_{j_\gamma}}(\m{F}_\gamma)=\ord_{X_{j_{\gamma^\sigma}}}(\m{F}_{\gamma^\sigma})$ follow from the fact that $\Sigma_n$ satisfies \ref{eqn:Bakeradditionalprop}. Suppose $\gamma=\tilde\alpha$ for some $\alpha\in\Sigma_{n,\singpointnt_n}$. Assuming \ref{item:Galoistheorem(i)}, the equality $M_\gamma=M_{\gamma^\sigma}$ follows from Lemma \ref{lem:Galoisconjgamma} and Remark \ref{rem:Matrixtildealpha}. Furthermore, $j_{\gamma}=j_{\gamma^{\sigma}}$ and $\m{F}_\gamma^{\sigma}=\m{F}_{\gamma^\sigma}$ as $j_{\alpha}=j_{\gamma^{\alpha}}$ and $\m F_{\alpha}^\sigma=\m F_{\alpha^\sigma}$.
Suppose $\gamma\in\Sigma_{\singpointnt_{n,\alpha}}$ for some $\alpha\in\Sigma_{n,\singpointnt_n}$. Then $\gamma=\beta\circ_{g_p}\alpha$ for some primitive $\beta\in\Z_+^2$ and some $p\in\singpointnt_{n,\alpha}$. Let $m\in\Z_+$ such that $\alpha\in\Omega_m$. In the proof of Theorem \ref{thm:Galoisequivariantpolynomials} we showed that $M_\gamma=(I_m\oplus M_\beta)\cdot M_{\alpha_p}$ and $M_{\gamma^\sigma}=(I_m\oplus M_\beta')\cdot M_{\alpha_p}$, for some matrices $M_\beta,M_\beta'$ attached to $\beta$ that can be freely chosen. Therefore it suffices to choose $M_\beta=M_\beta'$ to have $M_\gamma=M_{\gamma^\sigma}$. Finally, the polynomial $\m F_\gamma$ is fixed up to a power of $X_{j_\gamma}$, so we can easily require $\ord_{X_{j_\gamma}}(\m{F}_\gamma)=\ord_{X_{j_{\gamma^\sigma}}}(\m{F}_{\gamma^\sigma})$. 
\end{rem}

\begin{rem}\label{rem:strictGaloisconditions}
The conditions of Theorem \ref{thm:Galoisequivariantpolynomials} are satisfied if
\begin{enumerate}[label=(\arabic*)]
    \item $\tilde{\m G}_p=\bar{\m G}_p$ for any $p\in \singpointnt_n$;
    \item \label{Equivariant(ii)} for any primitive $\beta\in\Z_+^2$, a fixed matrix $M_\beta\in\SL_2(\Z)$ attached to $\beta$ is chosen whenever choosing a matrix attached to $\beta$ is required; 
    \item \label{Equivariant(i)} there exists $a\in\N$ such that $\ord_{X_{j_\gamma}}(\m{F}_\gamma)=a$ for any $\gamma\in\Sigma_{n+1}$. 
\end{enumerate}
Note that point \ref{Equivariant(ii)} implies that if $\gamma=\beta\circ_{g_p}\alpha$, for some $\alpha\in\Sigma_n$, $p\in\singpointnt_{n,\alpha}$, and some primitive vector $\beta\in\Z_+^2$, then we use the fixed matrix $M_\beta$ to construct $M_\gamma=(I_m\oplus M_\beta)\cdot M_{\alpha_p}$.
\end{rem}

Let $C_1$ be the completion of $C_0$ with respect to its Newton polygon. From \S\ref{subsec:BakerFirstStep} we easily see that $\gamma^\sigma=\gamma$ and $\m F_\gamma^\sigma=\m F_\gamma$ for any $\gamma\in\Sigma_1$ and any $\sigma\in G_k$. Hence the set $\Sigma_1\subset\Omega$, defining $C_1$, satisfies \ref{eqn:Bakeradditionalprop}. Theorem \ref{thm:Galoisequivariantpolynomials} and Remark \ref{rem:Galoisconditionscanbeachieved} show that 
we can construct Baker's resolutions of $C_0$
\[\dots\xrightarrow{s_{n+1}}C_{n+1}\xrightarrow{s_n}C_{n}\xrightarrow{s_{n-1}}\dots\xrightarrow{s_1}C_1,\]
such that for all $n\in\Z_+$ the sets $\Sigma_n$ satisfy the additional property \ref{eqn:Bakeradditionalprop}. In particular, $G_k\subseteq\mathrm{Aut}(C_n)$ and $\Sigma_n$ is $G_k$-invariant. The Galois-invariance of $\Sigma_n$ makes the action on $C_n$ easy to describe. Fix such a Baker's resolution.

\begin{lem}\label{lem:snGalois}
Let $n\in\Z_+$. Then $\sigma\circ s_n=s_n\circ\sigma$, for any $\sigma\in G_k$.
\proof
Recall that $s_n$ restricts to the identity on $C_0$. Then the two morphisms of $k$-schemes $\sigma\circ s_n$ and $s_n\circ\sigma$ agree on $C_0$. But $C_0$ is a dense open of $C_{n+1}$, thus $\sigma\circ s_n=s_n\circ\sigma$ by \cite[Proposition 3.3.11]{Liu}.
\endproof
\end{lem}

Let $n\in\Z_+$. Recall that for any $\sigma\in G_k$ and $\gamma\in\Sigma_n$, we have $f|_{\gamma^\sigma}=f|_\gamma^\sigma$, as $\Sigma_n$ satisfies \ref{eqn:Bakeradditionalprop}.
Therefore there is a natural action of $G_k$ on the set 
\[\textstyle\bigsqcup_{\gamma\in\Sigma_n}\{\text{simple roots of $f|_\gamma$ in $\bar k^\times$}\},\] 
where the simple root $r\in\bar k^\times$ of $f|_\gamma$ is taken to the simple root $\sigma(r)$ of $f|_{\gamma^\sigma}$.

\begin{thm}\label{thm:Galoismodel}
Let $f\in k[x_1^{\pm 1},y^{\pm 1}]$ be a Laurent polynomial defining a smooth curve $C_{0,k}:f=0$ over $\G_{m,k}^2$. Denote $C_0=C_{0,k}\times_k \bar k$.
We can recursively construct a Baker's resolution of $C_0$
\[\dots\stackrel{s_{n+1}}{\twoheadrightarrow}C_{n+1}\stackrel{s_n}{\twoheadrightarrow}C_{n}\stackrel{s_{n-1}}{\twoheadrightarrow}\dots\stackrel{s_1}{\twoheadrightarrow}C_1\]
where the maps $s_n$ are the birational morphisms resolving Galois-invariant sets $\singpointnt_n\subseteq\bigsqcup_{\alpha\in\Sigma_n}\Sing(\bar C_\alpha)$ (chosen arbitrarily) and the sets $\Sigma_n$, defining the curves $C_n/\bar k$, satisfy the additional property \ref{eqn:Bakeradditionalprop}. For any such sequence:
\begin{enumerate}[label=(\roman*)]
    \item There exists $h\in\Z_+$ such that $\singpointnt_n=\varnothing$ for any $n\geq h$.
    \label{item:Galoismodelitem(i)}
    \item If $\Sing(\bar C_\alpha)\subseteq\Reg(C_n)$ for all $\alpha\in\Sigma_n$, then the scheme-theoretical quotient $C_n/G_k$ is a generalised Baker's model of the smooth completion $C$ of $C_{0,k}$.\label{item:Galoismodelitem(ii)}
    \item If $C_n$ is outer regular, then there is a natural bijection
\[C(\bar k)\setminus C_{0,k}(\bar k)\xleftrightarrow{1:1}\bigsqcup_{\gamma\in\Sigma_n}\{\text{simple roots of $f|_\gamma$ in $\bar k^\times$}\},\]
preserving the action of the Galois group $G_k$.\label{item:Galoismodelitem(iii)}
\end{enumerate}
\proof
Theorem \ref{thm:Galoisequivariantpolynomials} and Remark \ref{rem:Galoisconditionscanbeachieved} show that the sequence can be constructed recursively, for any choice of Galois-invariant $\singpointnt_n\subseteq\bigsqcup_{\alpha\in\Sigma_n}\Sing(\bar C_\alpha)$. Part \ref{item:Galoismodelitem(i)} follows from Theorem \ref{thm:multiplesingmodel}. Part \ref{item:Galoismodelitem(ii)} is implied by Remark \ref{rem:nonsingularbarjacobian}, Lemma \ref{lem:snGalois} and the argument presented at the beginning of the current section.
Now assume $C_n$ is outer regular, i.e.\ $\bar C_\alpha$ is regular for all $\alpha\in\Sigma_n$. Therefore Lemma \ref{lem:regularpointsofbarCgamma} shows that, for every $\gamma\in\Sigma_n$, from the definition $\bar C_\gamma=\Spec\frac{D_\gamma}{(f|_\gamma)}$ we obtain a natural bijective map
\[\bar C_\gamma\stackrel{1:1}{\longleftrightarrow}\{\text{simple roots of $f|_\gamma$ in $\bar k^\times$}\}.\]
By part \ref{item:Galoismodelitem(ii)}, the smooth completion $C$ of $C_{0,k}$ is isomorphic to the quotient $C_n/G_k$. Therefore $C\times_k \bar k\simeq C_n$ and so $C(\bar k)\simeq C_n(\bar k)$. Since $C_{0,k}(\bar k)\simeq C_0(\bar k)$ by definition, Lemma \ref{lem:star} implies part \ref{item:Galoismodelitem(iii)}.
\endproof
\end{thm}

\begin{cor}\label{cor:OuterRegBakerModExistenceGeneral}
Any smooth projective curve $C$ defined over a perfect field $k$ has an outer regular generalised Baker's model.
\proof
By Corollary \ref{cor:smoothopenmodel}, for any projective smooth curve $C/k$ there exists a curve $C_{0,k}/k$ as in Theorem \ref{thm:Galoismodel}, birational to $C$. By Theorem \ref{thm:Galoismodel} we can construct a Baker's resolution of $C_{0,k}\times_k \bar k$ 
\[\dots\stackrel{s_{n+1}}{\twoheadrightarrow}C_{n+1}\stackrel{s_n}{\twoheadrightarrow}C_{n}\stackrel{s_{n-1}}{\twoheadrightarrow}\dots\stackrel{s_1}{\twoheadrightarrow}C_1\]
where $s_n$ are the birational morphisms resolving the Galois-invariant sets $\singpointnt_n=\bigsqcup_{\alpha\in\Sigma_n}\Sing(\bar C_\alpha)$ and the sets $\Sigma_n$ satisfy the additional property \ref{eqn:Bakeradditionalprop}. Furthermore, by Theorem \ref{thm:Galoismodel}\ref{item:Galoismodelitem(i)} there exists $n\in\Z_+$ such that $\singpointnt_n=\varnothing$. It follows that $\bar C_\gamma$ is regular for all $\gamma\in\Sigma_n$, i.e.\ $C_n$ is outer regular. Let $\tilde C=C_n/G_k$. Thus $\tilde C$ is an outer regular generalised Baker's model of $C$.
\endproof
\end{cor}

In the next proof we will show how Algorithm \ref{alg:TheAlgorithm} and Theorem \ref{thm:introduction} follow from previous results.

\begin{proof}[Proof of Theorem \ref{thm:introduction}]
Suppose $C_{0,k}$ is geometrically connected. We recursively construct a Baker's resolution of $C_0=C_{0,k}\times_k\bar k$
\[\dots\stackrel{s_{n+1}}{\twoheadrightarrow}C_{n+1}\stackrel{s_n}{\twoheadrightarrow}C_{n}\stackrel{s_{n-1}}{\twoheadrightarrow}\dots\stackrel{s_1}{\twoheadrightarrow}C_1\]
where the morphisms $s_n$ resolve the sets $\singpointnt_n=\bigsqcup_{\alpha\in\Sigma_n}\Sing(\bar C_\alpha)$. In the construction, for any $n\in\Z_+$, we make the following choices:
\begin{enumerate}[label=(\arabic*)]
    \item For any point $p\in\singpointnt_n$ choose $\tilde{\m G}_p=\bar{\m G}_p$. This is always possible, since $C_0$ is connected (see Remark \ref{rem:tildeGpchoice}).
    \item Every time we need to choose a matrix $M_\beta\in\SL_2(\Z)$ attached to some primitive vector $\beta=(\beta_1,\beta_2)\in\Z_+^2$,
    choose $M_\beta=\big(\begin{smallmatrix}\delta_1\!\!&\delta_2\!\cr\beta_1\!\!&\beta_2\!\end{smallmatrix}\big)$, where $(\delta_1,\delta_2)=\delta_\beta$ (Notation \ref{nt:deltabeta}).
    \item For any $\gamma\in\Sigma_{n+1}\setminus\Sigma_n$, choose $\m F_\gamma$ with $\ord_{X_{j_{\gamma}}}(\m F_\gamma)=0$. 
\end{enumerate}
With the choices above, by Theorem \ref{thm:Galoisequivariantpolynomials} and Remark \ref{rem:strictGaloisconditions}, the sets $\Sigma_n$ satisfy the additional property \ref{eqn:Bakeradditionalprop} and the sets $\singpointnt_n$ are Galois-invariant.
Theorem \ref{thm:Galoismodel}\ref{item:Galoismodelitem(i)} implies that there exists $n\in\Z_+$ such that $\bar C_\alpha$ is regular for all $\alpha\in\Sigma_n$. In other words, $C_n$ is outer regular.
Let $n$ be as small as possible, i.e.\ such that $C_h$ is not outer regular for every $h<n$.
By Theorem \ref{thm:Galoismodel}\ref{item:Galoismodelitem(iii)} there is a natural bijection preserving the action of the Galois group $G_k$,
\[C(\bar k)\setminus C_{0,k}(\bar k)\xleftrightarrow{1:1}\bigsqcup_{\gamma\in\Sigma_n}\{\text{simple roots of $f|_\gamma$ in $\bar k^\times$}\},\]
where $C$ is the smooth completion of $C_{0,k}$.

For any $h<n$ recall $\singpointnt_{h,\alpha}=\singpointnt_h\cap\bar C_\alpha$ for any $\alpha\in\Sigma_h$, and note that
\[\Sigma_{h,\singpointnt_h}=\{\alpha\in\Sigma_h\mid\singpointnt_{h,\alpha}\neq\varnothing\}=\{\alpha\in\Sigma_h\mid \bar C_\alpha\text{ is singular}\},\]
as $\singpointnt_{h,\alpha}=\Sing(\bar C_{\alpha})$. Define 
\[\tilde\Sigma_h=\{\tilde\alpha\mid\alpha\in\Sigma_{h,\singpointnt_h}\}\cup(\Sigma_h\setminus\Sigma_{h,\singpointnt_h}),\quad\text{and}\quad \Sigma_{h+1}^+=\textstyle{\bigcup_{\alpha\in\Sigma_{h,\singpointnt_h}}}\Sigma_{\singpointnt_{h,\alpha}},\]
so that $\Sigma_{h+1}=\Sigma_{h+1}^+\cup\tilde\Sigma_h$. We are going to show that $\bar C_\gamma$ is regular for any $\gamma\in\tilde\Sigma_h$. 
From the choice of $\singpointnt_h$, we have $\bar C_\gamma$ regular for any $\gamma\in\Sigma_h\setminus\Sigma_{h,\singpointnt_h}$. 
Now let $\alpha\in\Sigma_{h,\singpointnt_h}$. Lemma \ref{lem:Ctildealpha} shows that $\bar C_{\tilde\alpha}$ is isomorphic to $\bar C_\alpha\setminus\singpointnt_{h,\alpha}$. 
This is a regular scheme since $\singpointnt_{h,\alpha}=\Sing(\bar C_\alpha)$. 

Now we want to describe the set $\singpointnt_h$ for any $h<n$. Define $\Sigma_1^+=\Sigma_1$. If $h>1$, then $\bar C_\gamma$ is regular when $\gamma\in\tilde\Sigma_{h-1}$. Therefore $\Sigma_{h,\singpointnt_h}\subseteq\Sigma_{h}^+$ for all $h<n$. 
Now $D_\gamma=k[X_{j_\gamma}^{\pm 1}]$ for all $\gamma\in\Sigma_h^+$ from \S\ref{subsec:BakerFirstStep} (case $h=1$) and Lemma \ref{lem:Dgamma} (case $h>1$). 
Hence the points in $\singpointnt_h$ bijectively corresponds to non-zero multiple roots of $f|_\gamma$, $\gamma\in\Sigma_{h}^+$. 
Furthermore, given $\gamma\in\Sigma_h^+$, for any multiple root $r\in\bar k^\times$ of $f|_\gamma$ we have $\bar{\m G}_{p_r}=X_{j_{\gamma}}-r$, where $p_r$ is the point of $\singpointnt_h$ corresponding to $r$.

Let $P_h$, $P$ be the indexed sets of polynomials in $\bar k[X,Y]$ constructed in \S\ref{subsec:IntroductionTheorem} via Algorithm \ref{alg:TheAlgorithm}. We are going to prove the following facts:
\begin{enumerate}[label=(\roman*)]
    \item $P_h=\bigsqcup_{\gamma\in\Sigma_h^+}\{\m F_\gamma(X,Y)\}$ for any $h\leq n$;\label{item:Introductionthm(i)}
    \item $\bigsqcup_{i=1}^h P_i=\bigsqcup_{\gamma\in\Sigma_h}\{\m F_\gamma(X,Y)\}$ for any $h\leq n$; \label{item:Introductionthm(ii)}
    \item $P=\bigsqcup_{\gamma\in\Sigma_n}\{\m F_\gamma(X,Y)\}$; \label{item:Introductionthm(iii)}
\end{enumerate}
where $\m F_\gamma(X,Y)$ denotes the image of $\m F_\gamma$ under the isomorphism 
\[\bar k[X_{j_\gamma},Y]\rightarrow \bar k[X,Y],\quad X_{j_\gamma}\mapsto X,\,\, Y\mapsto Y.\] 
Note that \ref{item:Introductionthm(iii)} concludes the proof of Theorem \ref{thm:introduction}. 

We prove \ref{item:Introductionthm(i)} by induction on $h$. If $h=1$, then $\Sigma_1^+=\Sigma_1$, and so the equality follows from \S\ref{subsec:BakerFirstStep}. Suppose $h\geq 1$. We want to show that 
\[P_{h+1}=\bigsqcup_{\gamma\in\Sigma_{h+1}^+}\{\m F_\gamma(X,Y)\}.\] 
Let us recall the steps that have to be done to construct the polynomials $\m F_\gamma$, for $\gamma\in\Sigma_{h+1}^+$. We observed that the points in $\singpointnt_h$ correspond to non-zero multiple roots of $f|_\alpha$ for $\alpha\in\Sigma_h^+$. For any $\alpha\in\Sigma_h^+$ and any multiple root $a\in \bar k^\times$ of $f|_\alpha$ do:
\begin{enumerate}[label=(\arabic*)]
    \item Replace $Y$ with $\tilde Y$ in $\m F_\alpha$ so that $\m F_\alpha\in\bar k[X_{j_{\alpha}},\tilde Y]$.
    \item Denote by $p_a$ the point of $\singpointnt_h$ corresponding to $a$. We noted that $\bar{\m G}_{p_a}=X_{j_\alpha}-a$. Since we chose $\tilde{\m G}_{p_a}=\bar{\m G}_{p_a}$, the normal form $\m F_{\alpha,p_a}$ of $\m F_\alpha$ by $\tilde X_{m+1}-\tilde{\m G}_{p_a}$ with respect to the lexicographic order given by $X_{j_\alpha}>\tilde X_{m+1}>\tilde Y$ is 
    \[\m F_{\alpha,p_a}(\tilde X_{m+1},\tilde Y)=\m F_\alpha(\tilde X_{m+1}+a,\tilde Y)\]
    (here $m\in\Z_+$ such that $\alpha\in\Omega_m$).
    \item Draw the Newton polygon $\Delta_{\alpha,p_a}$ of $\m F_{\alpha,p_a}$.
    \item Let $\gamma=\beta\circ_{g_{p_a}}\alpha$ for the normal vector $\beta\in\Z_+^2$ of some edge of $\Delta_{\alpha,p_a}$. From \S\ref{subsec:Newtonpolygons}, we have $\gamma\in\Sigma_{p_a}$.
    \item The fixed matrix $M_\beta=\left(\begin{smallmatrix}\!\delta_1\!\!&\delta_2\!\!\cr\!\beta_1\!\!&\beta_2\!\!\end{smallmatrix}\right)$ gives the change of variables 
    \[(\tilde X_{m+1},\tilde Y)=(X_{m+1},Y)\bullet M_\beta=(X_{m+1}^{\delta_1}Y^{\beta_1}, X_{m+1}^{\delta_2}Y^{\beta_2}).\]
    Via this transformation we define $\m F_\gamma$ to be the unique polynomial in $\bar k[X_{m+1},Y]$ such that $\ord_{X_{m+1}}\m F_\gamma=\ord_Y\m F_\gamma=0$, satisfying
    \[\m F_{\alpha,p_a}(\tilde X_{m+1},\tilde Y)\stackrel{M_\beta}{=}X_{m+1}^{n_X}Y^{n_Y}\cdot \m F_\gamma(X_{m+1},Y),\]
    for some $n_X,n_Y\in\Z$.
    \item In fact, all elements $\gamma\in\Sigma_{p_a}$ equals $\beta\circ_{g_{p_a}}\alpha$ with $\beta\in\Z_+^2$ normal vector of some edge of $\Delta_{\alpha,p_a}$. 
\end{enumerate}
The procedure presented here describes how to construct the polynomials $\m F_\gamma$ for all $\gamma\in\Sigma_{h+1}^+$ knowing the polynomials $\m F_\alpha$, for all $\alpha\in\Sigma_{h, \singpointnt_n}\subseteq\Sigma_h^+$. Comparing it with Algorithm \ref{alg:TheAlgorithm} we see that $P_{h+1}=\bigsqcup_{\gamma\in\Sigma_{h+1}^+}\{\m F_\gamma(X,Y)\}$ since $\bigsqcup_{\alpha\in\Sigma_{h}^+}\{\m F_\alpha(X,Y)\}=P_{h}$ by inductive hypothesis.

We now prove \ref{item:Introductionthm(ii)} by induction on $h$. If $h=1$, then \ref{item:Introductionthm(ii)} follows from \ref{item:Introductionthm(i)} as $\Sigma_1=\Sigma_1^+$ by definition. Suppose then $h\geq 1$. We want to show that $\bigsqcup_{i=1}^{h+1} P_i=\bigsqcup_{\gamma\in\Sigma_{h+1}}\{\m F_\gamma(X,Y)\}$. 
But $\Sigma_{h+1}=\Sigma_{h+1}^+\sqcup\tilde\Sigma_{h}$, so, by \ref{item:Introductionthm(i)} and inductive hypothesis, it suffices to show that
\[\bigsqcup_{\gamma\in\tilde\Sigma_{h}}\{\m F_\gamma(X,Y)\}=\bigsqcup_{\gamma\in\Sigma_{h}}\{\m F_\gamma(X,Y)\}.\]
But this easily follows from the definition of $\tilde\Sigma_{h}$ since $\m F_{\tilde\alpha}=\m F_\alpha$ for any $\alpha\in\Sigma_{\singpointnt_h,h}$ (Notation \ref{nt:Notationfortildealpha}).

To prove \ref{item:Introductionthm(iii)}, first note that from \ref{item:Introductionthm(i)}, for any $h\leq n$ the indexed set $P_h$ is non-empty. Then \ref{item:Introductionthm(iii)} is implied by \ref{item:Introductionthm(ii)} if for any $f_\ell\in P_n$, the polynomial $f|_\ell(X)=f_\ell(X,0)$ has no non-zero multiple roots. But this follows from \ref{item:Introductionthm(i)} since $\bar C_\alpha$ is regular for any $\alpha\in\Sigma_n$, and so $f|_\gamma$ has no multiple roots in $\bar k^\times$ for any $\gamma\in\Sigma_h^+$ as $D_\gamma=k[X_{j_\gamma}^{\pm 1}]$ in this case (Lemma \ref{lem:Dgamma}). As already observed, this concludes the proof of Theorem \ref{thm:introduction}.
\end{proof}

\section{Superelliptic equations}\label{sec:Superelliptic}
Let $k$ be a perfect field and let $\bar k$ be an algebraic closure of $k$. Denote by $G_k$ the absolute Galois group of $k$. As application of the construction presented in the previous sections, we consider a curve $C_{0,k}$ in $\G_{m,k}^2$ defined by an equation
\[y^s=h(x),\]
for some polynomial $h\in k[x]$ and some $s\in\Z_+$ not divisible by $\mathrm{char}(k)$. By convention the polynomial $f(x,y)$ defining $C_{0,k}$ will be $y^s-h(x)$. 
Denote by $C_0$ the curve $C_{0,k}\times_{k}\bar k$. 
Note that $C_0$ is smooth, but may be not connected, e.g.\ when $h(x)$ is an $s$-th power. 
Expand
\[h(x)=\sum_{i=m_0}^d c_i x^i,\qquad c_i\in k,\]
where $c_{m_0}$ and $c_d$ are non-zero.
We want to study a Baker's resolution of $C_0$
\[\dots\stackrel{s_{n+1}}{\twoheadrightarrow}C_{n+1}\stackrel{s_n}{\twoheadrightarrow}C_{n}\stackrel{s_{n-1}}{\twoheadrightarrow}\dots\stackrel{s_1}{\twoheadrightarrow}C_1\]
as in Theorem \ref{thm:Galoismodel}, where the Galois-invariant sets $\singpointnt_n$ which the birational morphisms $s_n$ resolve are as large as possible, i.e.\ $\singpointnt_n=\bigsqcup_{\alpha\in\Sigma_n}\Sing(\bar C_\alpha)$. For the purpose of the construction of the Baker's resolution $x_1=x$.

The Newton polygon $\Delta$ of $f$ always has at least two edges: $\ell_1$ with endpoints $(m_0,0)$, $(0,s)$ and normal vector $\gcd(m_0,s)^{-1}(s,m_0)$, and $\ell_2$ with endpoints $(d,0)$, $(0,s)$ and normal vector $\gcd(d,s)^{-1}(-s,-d)$. If $h$ is a monomial then $\Delta$ is a segment, otherwise $\Delta$ is a triangle. In the latter case, the third edge $\ell$ has endpoints $(m_0,0), (d,0)$ and normal vector $(0,1)$. Construct the completion $C_1$ of $C_0$ with respect to $\Delta$, as described in \S\ref{subsec:BakerFirstStep}. For any $i=1,2$ let $v_i\neq(0,1)$ be the normal vector of $\ell_i$ and set $\alpha_i=(v_i,())\in\Sigma_1$. From Proposition \ref{prop:reductionfalphabasecase} it follows that
\[f|_{\alpha_i}= X_1^\ast\cdot (a_l X_1^l+a_0),\qquad l\in\Z_+,\,\,\,a_0,a_l\in k^\times,\]
where $\mathrm{char}(k)\nmid l$. In fact, if $i=1$ then $l=\gcd(m_0,s)$, $a_l=-c_{m_0}$, $a_0=1$, while if $i=2$ then $l=\gcd(d,s)$, $a_l=1$, $a_0=-c_d$. In particular, $f|_{\alpha_i}$ has no multiple roots in $\bar k^\times$. 


Suppose now that $h$ is not a monomial. Let $v=(0,1)$ be the normal vector of $\ell$ and let $\alpha=(v,())$ be the corresponding element of $\Sigma_1$. Consider $\m F_\alpha\in\bar k[X_1,Y]$. Note that since $v=(0,1)$, we can choose $M_\alpha=\lb\begin{smallmatrix}\!1\!&0\!\\\!0\!&1\!\end{smallmatrix}\rb$ and so $\m F_\alpha=f(X_1,Y)$. In particular, $f|_\alpha=f(X_1,0)=-h(X_1)$. Since $D_\alpha=\bar k[X_1^{\pm 1}]$, the singular points of $\bar C_\alpha$ correspond to the non-zero multiple roots of $f|_\alpha$, or, equivalently, to the non-zero multiple roots of $h$. Hence $\singpointnt_1$ is the set of those points. If $\singpointnt_1=\varnothing$, then $C_1$ is (outer) regular. We deduce the following lemma.

\begin{lem}
If $h$ has no multiple root in $\bar k^\times$, then $C_1$ is an outer regular (generalised) Baker's model of the smooth completion of $C_0$.
\end{lem}

Suppose $\singpointnt_1\neq\varnothing$. Construct the morphism $s_1:C_2\rightarrow C_1$ resolving $\singpointnt_1$. 
Let $v$ and $\alpha$ as above. Rename the variable $Y$ of $\m F_\alpha$ to $\tilde Y$, so that $\m F_\alpha\in \bar k[X_1,\tilde Y]$. Let $p\in\singpointnt_1$ and let $r\in\bar k^\times$ be the multiple root of $h$ corresponding to $p$. 
One has $\bar{\m G}_p=X_1-r$. Note that $\bar{\m G}_p$ does not divide $\m F_\alpha$, so choose $\tilde{\m G}_p=\bar{\m G}_p$. 
Then \[\m F_{\alpha,p}(\tilde X_2,\tilde Y)=\m F_\alpha(\tilde X_2+r,\tilde Y)=f(\tilde X_2+r,\tilde Y)=\tilde Y^s-h(\tilde X_2+r).\]
It follows that the Newton polygon $\Delta_{\alpha,p}$ of $\m F_{\alpha,p}$ has a unique edge $\ell_r$ with normal vector in $\Z_+^2$. Denoting by $m_r$ the multiplicity of the root $r$ of $h$, the endpoints of $\ell_r$ are $(m_r,0)$, $(0,s)$ and $\beta_r=\gcd(m_r,s)^{-1}(s,m_r)$ is its normal vector. 
Let $\gamma_r=\beta_r\circ_{g_p}\alpha$, where $g_p$ is the polynomial related to $\tilde{\m G}_p$ by $M_\alpha$. Define $h_r(x)=h(x)/(x-r)^{m_r}\in\bar k[x]$. Then Proposition \ref{prop:reductionfgamma} implies 
\[f|_{\gamma_r}(X_2)=X_2^\ast\cdot(-a_rX_2^{\gcd(m_r,s)}+1),\]
where $a_r=h_r(r)$. In particular, since $\mathrm{char}(k)\nmid s$, the polynomial $f|_{\gamma_r}$ has no multiple root in $\bar k^\times$. Therefore $\bar C_{\gamma_r}$ is regular for any non-zero multiple root $r$ of $h$. Moreover, $\bar C_{\tilde\alpha}$ is also regular as $\bar C_{\tilde\alpha}\simeq \bar C_\alpha\setminus \singpointnt_1$. Recall the notation $\tilde\Sigma_1=\hat\Sigma_1\cup\{\tilde\alpha\}$, where $\hat\Sigma_1=\Sigma_1\setminus\{\alpha\}$. Since
\[\Sigma_2=\{\gamma_r\mid r\,\,\text{multiple root of }h\}\cup\tilde\Sigma_1\] 
the schemes $\bar C_\gamma$ are regular for all $\gamma\in\Sigma_2$. We obtain the following result.

\begin{lem}
If $h$ has multiple roots in $\bar k^\times$, then $C_1$ is singular, but $C_2$ is an outer regular generalised Baker's model of the smooth completion of $C_0$.
\end{lem}

\begin{rem}
Note that $C_2=\bigcup_{\gamma\in\Sigma_2} C_\gamma$ since $C_0\subseteq C_\gamma$ for any $\gamma\in\Sigma_2$.
\end{rem}


We want to give an explicit description of the curve $C_{2,k}=C_2/G_k$, when $h$ has multiple roots in $\bar k^\times$. First note that for any $\gamma\in\tilde\Sigma_1$ the polynomials defining the curves $C_\gamma$ have coefficients in $k$. Therefore $G_k\subseteq\mathrm{Aut}(C_\gamma)$ for all $\gamma\in\tilde\Sigma_1$ and the charts $C_\gamma/G_k$ of $C_{2,k}$ easily follows. It remains to describe the curve $\big(\bigcup_{\sigma\in G_k}C_{\gamma_{\sigma(r)}}\big)/G_k$ for any non-zero multiple root $r$ of $h$.  

Let $g\in k[x]$ be the minimal polynomial of a multiple root $r\in\bar k^\times$ of $h$. Let $m_r$, $h_r$, $\beta_r$, $\gamma_r$ as above. Set $s_r=\gcd(m_r,s)$. Note that $\ord_g(h)=m_r$. If $\lb\begin{smallmatrix}\!\delta_1\!&\delta_2\!\\\!\beta_1\!&\beta_2\!\end{smallmatrix}\rb$ is the matrix attached to $\beta_r$ used for the construction of $C_{\gamma_r}$ then
\[\O_{C_{\gamma_r}}(C_{\gamma_r})=\sfrac{\bar k[X_1^{\pm 1},X_2^{\pm 1},Y]}{(1-X_2^{s_r}\cdot h_r(X_1), X_2^{\delta_1}Y^{\beta_1}-X_1+r)}.\]
Define $g_r, h_g\in \bar k[x]$ by $g_r(x)=g(x)/(x-r)$, $h_g(x)=h(x)/g(x)^{m_r}$. 
Note that $g_r(X_1)$ is invertible in $\O_{C_{\gamma_r}}(C_{\gamma_r})$.
Consider the homomorphism
\[\phi_r:\sfrac{\bar k[X_1^{\pm 1},X_2^{\pm 1},Y]}{(1-X_2^{s_r}\cdot h_g(X_1), X_2^{\delta_1}Y^{\beta_1}-g(X_1))}\longrightarrow
\sfrac{\bar k[X_1^{\pm 1},X_2^{\pm 1},Y]}{(1-X_2^{s_r}\cdot h_r(X_1), X_2^{\delta_1}Y^{\beta_1}-X_1+r)},\]
taking $X_1\mapsto X_1$, $X_2\mapsto X_2\cdot g_r(X_1)^{\beta_2}$, $Y\mapsto Y\cdot g_r(X_1)^{-\delta_2}$. Let $A_g:=\mathrm{Dom}(\phi_r)$. Note that $\Spec A_g=C_{\gamma_g}$, where $\gamma_g=\beta_r\circ_{g}\alpha\in\Omega$. Then $\phi_r$ induces an open immersion $\iota_r:C_{\gamma_r}\hookrightarrow C_{\gamma_g}$.
The glueing of the open immersions $\iota_{\sigma(r)}$, for $\sigma\in G_k$, gives an isomorphism 
\[\big(\textstyle\bigcup_{\sigma\in G_k}C_{\gamma_{\sigma(r)}}\big)\simeq C_{\gamma_g},\] 
commuting with the Galois action. Since $C_{\gamma_g}$ is defined by polynomials with coefficients in $k$, the quotient $C_{\gamma_g}/G_k$ is easy to describe, as required.


\section{Example}\label{sec:Example}
Let $C_{0.\F_2}:f=0\subset \G_{m,\F_2}^2$ with $f=x_1^4+1+y^2+y^3$. Note that $C_{0,\F_2}$ is smooth. Write $C_0=C_{0,\F_2}\times_{\F_2}\bar \F_2$, where $\bar \F_2$ is an algebraic closure of $\F_2$.

\subsection{Construction of $C_1$}

The Newton polygon $\Delta$ of $f$ is
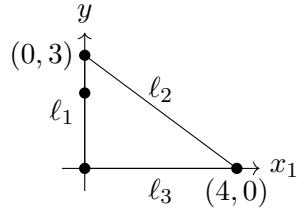
\[\begin{tikzpicture}[scale=0.5]
    \draw[->] (-0.6,0) -- (4.6,0) node[right] {$x_1$};
    \draw[->] (0,-0.6) -- (0,3.6) node[above] {$y$};
    \tkzDefPoint(4,0){A}
    \tkzDefPoint(0,2){B}
    \tkzDefPoint(0,3){C}
    \tkzDefPoint(0,0){O}
    \tkzLabelPoint[below](A){$(4,0)$}
    \tkzLabelPoint[left](C){$(0,3)$}
    \foreach \n in {O,A,B,C}
    \node at (\n)[circle,fill,inner sep=1.5pt]{};
    \draw (A) -- (O) node [midway,below, fill=none] {$\ell_3$};     
    \draw (O) -- (C) node [midway,left, fill=none] {$\ell_1$};
    \draw (C) -- (A) node [midway,right, above, fill=none] {$\ell_2$};
\end{tikzpicture}\]
We want to construct the completion $C_1$ of $C_0$ with respect to $\Delta$ as explained in \S\ref{subsec:BakerFirstStep}.
For any edge $\ell_i$ of $\Delta$ let $\beta_i$ be the normal vector of $\ell_i$. Then $\beta_1=(1,0)$, $\beta_2=(-3,-4)$, $\beta_3=(0,1)$. Let $\alpha_i=(\beta_i,())\in\Sigma_1$ for $i=1,2,3$. Then $\Sigma_1=\{\alpha_1,\alpha_2,\alpha_3\}$ and
\[C_1=C_{\alpha_1}\cup C_{\alpha_2}\cup C_{\alpha_3},\]
where we omitted $C_0$ as $C_0\subset C_\alpha$ for every $\alpha\in\Sigma_1$.
From Proposition \ref{prop:reductionfalphabasecase} 
the polynomials
$f|_{\alpha_1}$ and $f|_{\alpha_2}$ are separable (up to a power of $X_1$)
and so
the corresponding curves $C_{\alpha_1}$ and $C_{\alpha_2}$ are regular. 
On the other hand, $1\in \F_2$ is a non-zero multiple root of $f|_{\alpha_3}$, so $C_{\alpha_3}$ may be singular.
Let us compute the defining polynomial $\m{F}_{\alpha_3}$. 
The identity matrix $I\in\SL_2(\Z)$ is attached to $\beta_3$, so we fix $M_{\alpha_3}=I$. Via $I$ we get
\[\m{F}_{\alpha_3}=X_1^4+1+Y^2+Y^3.\]
Then $C_{\alpha_3}=\Spec \bar\F_2[X_1^{\pm1},Y]/(\m{F}_{\alpha_3})$ is singular. 
Thus $C_1$ is not smooth, having $1$ singular point, visible on $C_{\alpha_3}$. 

\subsection{Construction of $C_2$}

Rename the variable $Y$ of $C_{\alpha_3}$ to $\tilde Y$. Let $p$ be the singular point of $C_{\alpha_3}$. Then $\bar{\m G}_{p}=X_1+1$. Choose $\tilde{\m G}_{p}=\bar{\m G}_{p}$. 
We will construct the morphism $s_1:C_2\rightarrow C_1$ resolving the set $\singpointnt_1=\{p\}$. Note that $\singpointnt_1=\bigsqcup_{\alpha\in\Sigma_1}\Sing(\bar C_\alpha)$.
Let $\alpha=\alpha_3$ and $\beta=\beta_3$.
Then
\[\tilde{\m{G}}_{p}\lb(x_1,y)\bullet M_{\alpha}^{-1}\rb=x_1+1,\] 
so $g_p=x_1+1\in \F_2[x_1,y]$ is the polynomial related to $\tilde{\m G}_{p}$ by $M_\alpha$. Define $g_2=g_p$ and $f_2=x_2-g_2$. Note that since $\singpointnt_1$ consists of a single point, we have $\tilde{\m G}_{\singpointnt_1}=\tilde{\m G}_p$ and $g_{\singpointnt_1}=g_p$. Then $\alpha_p=\tilde\alpha$. Compute $\ord_{\beta}(g_p)=0$ and $\tilde\alpha=\alpha_p=(0,1)\circ_{g_{\singpointnt_1}}\alpha=((0,0,1),(g_2))$. Then
\[C_{\tilde\alpha}=C_{\alpha_p}=\Spec\frac{\bar \F_2[X_1^{\pm 1},\tilde X_2^{\pm 1},\tilde Y]}{(\m{F}_{\alpha_3}, \tilde X_2+X_1+1)}.\]
The normal form of $\m{F}_{\alpha_3}$ by $\tilde X_2-\m{G}$ with respect to the lexicographic order given by $X_1>\tilde X_2>\tilde Y$ is
\[\m{F}_{\alpha,p}=\m F_\alpha\big(\tilde X_2+1,\tilde Y\big)=\tilde X_2^4+\tilde Y^2+\tilde Y^3.\]
The Newton polygon of $\m{F}_{\alpha,p}$ is
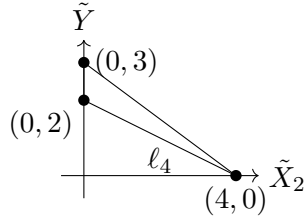
\[\begin{tikzpicture}[scale=0.5]
    \draw[->] (-0.6,0) -- (4.6,0) node[right] {$\tilde X_2$};
    \draw[->] (0,-0.6) -- (0,3.6) node[above] {$\tilde Y$};
    \tkzDefPoint(4,0){A}
    \tkzDefPoint(0,2){B}
    \tkzDefPoint(0,3){C}
    \tkzLabelPoint[below](A){$(4,0)$}
    \tkzLabelPoint[below left](B){$(0,2)$}
    \tkzLabelPoint[right](C){$(0,3)$}
    \foreach \n in {A,B,C}
    \node at (\n)[circle,fill,inner sep=1.5pt]{};
    \draw (A) -- (B) node [midway,below, fill=none] {$\ell_4$};     
    \draw (B) -- (C);
    \draw (C) -- (A);
\end{tikzpicture}\]
There is only $1$ edge, denoted $\ell_4$, with normal vector in $\Z_+^2$. The normal vector of $\ell_4$ is $\beta_4=(1,2)$. It follows that $v_4=\beta_4\circ_{g_p}\beta=(0,1,2)$. Hence $\gamma_4=\beta_4\circ_{g_p}\alpha=(v_4,(g_2))$ is the corresponding element of $\Sigma_p$.
Then $\Sigma_2=\{\alpha_1,\alpha_2,\tilde\alpha_3,\gamma_4\}$.

 To check whether $C_{\gamma_4}$ is regular, compute $\m{F}_{\gamma_4}$. The matrix $M_{\beta_4}=\left(\begin{smallmatrix}\!1\!&1\!\\\!1\!&2\!\end{smallmatrix}\right)$, attached to $\beta_4$, defines the change of variables $\tilde X_2=X_2Y$, $\tilde Y=X_2Y^2$, from which we get
\begin{align*}
&\m{F}_{\alpha,p}=X_2^2Y^4\m{F}_{\gamma_4},&&\m{F}_{\gamma_4}=X_2^2+1+X_2Y^2,&\\
&\tilde X_2-\tilde{\m{G}}_p=\m{F}_2,&&\m{F}_2=X_2Y+X_1+1,&
\end{align*}
where $\m F_2$ is the generator of the ideal $\got a_{\gamma_4}$. Therefore the curve
\[C_{\gamma_4}=\Spec \frac{\bar \F_2[X_1^{\pm 1},X_2^{\pm 1},Y]}{(\m{F}_{\gamma_4})+\got a_{\gamma_4}}\]
is singular, and so is the projective curve $C_2=C_{\alpha_1}\cup C_{\alpha_2}\cup C_{\tilde\alpha_3} \cup C_{\gamma_4}$. In the union we omitted $C_0$, as $C_0\subset C_{\alpha_1}$.

\subsection{Construction of $C_3$}
Let $q$ be the singular point of $C_{\gamma_4}$.
We now construct the morphism $s_2:C_3\rightarrow C_2$ resolving $\singpointnt_2=\{q\}$. Let $\gamma=\gamma_4$. Rename the variable $Y$ of $C_{\gamma}$ to $\tilde Y$.
Choose $\tilde{\m{G}}_q=\bar{\m{G}}_q=X_2+1$. By definition
\[M_{\gamma}=((1)\oplus M_{\beta_4})\cdot M_{\alpha_p}=\left(\begin{smallmatrix}1&0&0\\0&1&1\\0&1&2\end{smallmatrix}\right)\cdot \left(\begin{smallmatrix}1&0&0\\0&1&0\\0&0&1\end{smallmatrix}\right)=\left(\begin{smallmatrix}1&0&0\\0&1&1\\0&1&2\end{smallmatrix}\right),\quad M_\gamma^{-1}=\left(\begin{smallmatrix}1&0&0\\0&2&-1\\0&-1&1\end{smallmatrix}\right).\]
Then $g_q=x_2^2+y\in \F_2[x_1,x_2,y]$ is the polynomial related to $\tilde{\m G}_q$ by $M_\gamma$, as 
\[\tilde{\m{G}}_q\lb(x_1,x_2,y)\bullet M_{\gamma}^{-1}\rb=x_2^2y^{-1}+1.\] 
Let $g_3=(x_1+1)^2+y$ be the Laurent polynomial in $k[x_1^{\pm 1}, y^{\pm 1}]$ congruent to $g_q$ modulo $f_2$. Compute $\ord_{v_4}(g_q)=2$. Then \[\tilde\gamma=\gamma_q=(0,1)\circ_{g_q}\gamma=((0,1,2,2),(g_2,g_3)).\]
The normal form of $\m{F}_{\gamma}$ by $\tilde X_3-\tilde{\m{G}}_q$ with respect to the lexicographic order given by $X_2>\tilde X_3>\tilde Y$ is \[\m{F}_{\gamma,q}=\tilde X_3^2+(\tilde X_3+1)\tilde Y^2.\] The Newton polygon of $\m{F}_{\gamma,q}$ is
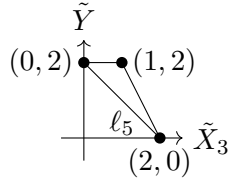
\[\begin{tikzpicture}[scale=0.5]
    \draw[->] (-0.6,0) -- (2.6,0) node[right] {$\tilde X_3$};
    \draw[->] (0,-0.6) -- (0,2.6) node[above] {$\tilde Y$};
    \tkzDefPoint(2,0){A}
    \tkzDefPoint(0,2){B}
    \tkzDefPoint(1,2){C}
    \tkzLabelPoint[below](A){$(2,0)$}
    \tkzLabelPoint[left](B){$(0,2)$}
    \tkzLabelPoint[right](C){$(1,2)$}
    \foreach \n in {A,B,C}
    \node at (\n)[circle,fill,inner sep=1.5pt]{};
    \draw (A) -- (B) node [midway,below, fill=none] {$\ell_5$};     
    \draw (B) -- (C);
    \draw (C) -- (A);
\end{tikzpicture}\]
There is only $1$ edge, denoted $\ell_5$, with normal vector in $\Z_+^2$. The normal vector of $\ell_5$ is $\beta_5=(1,1)$ and so the corresponding element of $\Sigma_q$ is \[\gamma_5=\beta_5\circ_{g_q}\gamma=((0,1,3,2),(g_2,g_3)).\] 
Hence $\Sigma_3=\{\alpha_1,\alpha_2,\tilde\alpha_3,\tilde\gamma_4,\gamma_5\}$. 

The matrix $M_{\beta_5}=\left(\begin{smallmatrix}\!1\!&0\!\\\!1\!&1\!\end{smallmatrix}\right)$, attached to $\beta_5$, defines the change of variables $\tilde X_3=X_3Y$, $\tilde Y=Y$ from which we get
\begin{align*}
&\m{F}_{\gamma,q}=Y^2\m{F}_{\gamma_5}&&\m{F}_{\gamma_5}=X_3^2+X_3Y+1,&\\
&\tilde X_3-\tilde{\m{G}}_q=\m{F}_3&&\m{F}_3=X_3Y+X_2+1,&
\end{align*}
and $\m F_2=X_2Y+X_1+1$ is the image of the generator of $\got a_\gamma$ under $M_{\beta_5}$.
Then $\got a_{\gamma_5}=(\m F_2,\m F_3)$ and
\[C_{\gamma_5}=\Spec \frac{\bar \F_2[X_1^{\pm 1},X_2^{\pm 1},X_3^{\pm 1},Y]}{(\m{F}_{\gamma_5})+\got a_{\gamma_5}}\]
is regular (even if $f|_{\gamma_5}$ is not separable).
Therefore the curve
\[C_3=C_{\alpha_2}\cup C_{\alpha_3}\cup C_{\tilde\alpha_3} \cup C_{\tilde\gamma_4} \cup C_{\gamma_5}\]
is regular as well, and is a generalised Baker's model of the smooth completion of $C_0$. It is not outer regular, since $\bar C_{\gamma_5}$ has a singular point. One more step is therefore necessary (and sufficient by Proposition \ref{prop:nonsingularisomorphism}) to construct an outer regular generalised Baker's model. Note that in the description of $C_3$ we omitted $C_0$, as $C_0\subset C_{\alpha_1}$. Finally, the polynomials defining the charts $C_\gamma$, $\gamma\in\Sigma_3$ have coefficients in $\F_2$, so the construction of the generalised Baker's model $C_3/G_{\F_2}$ of the smooth completion of $C_{0,\F_2}$ easily follows.

\appendix
\input{AppendixA}
\input{AppendixB}

\end{document}

%% file: AppendixA.tex
\section{Birational smooth hypersurface of a variety}\label{appendix:smoothopen}
Let $k$ be a perfect field. Recall that an algebraic variety $Z$ over $k$, denoted $Z/k$, is a scheme $Z\rightarrow \Spec k$ of finite type.

\begin{lem}\label{lem:GeneralCurvesFunctionRing}
Let $Z/k$ be a geometrically reduced algebraic variety, pure of dimension $n$. Suppose either $n>0$ or $k$ infinite. Then there exists a separable polynomial $f\in k(x_1,\dots,x_n)[y]$, such that $k(Z)=k(x_1,\dots,x_n)[y]/(f)$.
\proof
Let $Z_1,\dots, Z_m$ be the irreducible components of $Z$. From \cite[Proposition 7.1.15]{Liu}, \cite[Lemma 7.5.2(a)]{Liu} it follows that $k(Z)\simeq \bigoplus_{i=1}^mk(Z_i)$. Let $i=1,\dots,m$. As $Z$ is pure, $\dim Z_i=\dim Z=n$. Since $Z_i$ is geometrically reduced and integral, it follows from \cite[Proposition 3.2.15]{Liu} that the field of functions $k(Z_i)$ is a finite separable extension of a purely trascendental extension $k(x_1,\dots,x_n)$. 
Hence there exists a monic irreducible separable polynomial $f_i\in k(x_1,\dots,x_n)[y]$ such that \[k(Z_i)\simeq k(x_1,\dots,x_n)[y]/(f_i).\] 

We want to show that we can inductively choose the polynomials $f_i$ above such that $\gcd(f_i,f_j)=1$ for all $j<i$. Suppose we have fixed $f_1,\dots,f_{i-1}$ for some $i\geq 1$, and let $g_i\in k(x_1,\dots,x_n)[y]$ be any monic irreducible polynomial such that $k(Z_i)\simeq k(x_1,\dots,x_n)[y]/(g_i)$. Since $k(x_1,\dots,x_n)$ is infinite, there exists $c\in k(x_1,\dots,x_n)$ such that $\tau_c g_i\neq f_j$ for any $j<i$, where $\tau_c g_i$ is the polynomial defined by $\tau_c g_i(y)=g_i(y-c)$. But $\tau_c g_i$ and $f_j$ are irreducible monic polynomials, so $\gcd(\tau_c g_i,f_j)=1$. Moreover, $\tau_c g_i$ is separable and
\[k(x_1,\dots,x_n)[y]/(g_i)\simeq k(x_1,\dots,x_n)[y]/(\tau_c g_i)\]
via the map taking $y\mapsto y-c$. Then choose $f_i=\tau_cg_i$.

Thus assume $\gcd(f_i, f_j)=1$ for any $i,j=1,\dots,m$. From the Chinese Remainder Theorem it follows that
\[k(Z)\simeq \bigoplus_{i=1}^mk(Z_i)\simeq \bigoplus_{i=1}^m\frac{k(x_1,\dots,x_n)[y]}{(f_i)}\simeq \frac{k(x_1,\dots,x_n)[y]}{(f)},\]
where $f=\prod_{i=1}^mf_i$.
\endproof
\end{lem}

The following result is a variant of \cite[Theorem 5.7]{BMS}.

\begin{thm}\label{thm:Openofvariety}
Let $Z/k$ be a geometrically reduced, separated algebraic variety, pure of dimension $n$. Suppose either $n>0$ or $k$ infinite. Then there exists a smooth affine hypersurface $V$ in $\A_k^{n+1}$ birational to $Z$.
\proof
Lemma \ref{lem:GeneralCurvesFunctionRing} shows that there exists a separable polynomial $f\in k(x_1,\dots,x_n)[y]$ such that $k(Z)\simeq k(x_1,\dots,x_n)[y]/(f)$. Rescaling $f$ by an element of $k(x_1,\dots,x_n)$ if necessary, we can assume that $f$ is a polynomial in $k[x_1,\dots,x_n,y]$ with no irreducible factors in $k[x_1,\dots,x_n]$. Hence the total quotient ring of $k[x_1,\dots,x_n,y]/(f)$ is $k(x_1,\dots,x_n)[y]/(f)$.
It follows that there exists a birational map $Z\-->Z_0$, where $Z_0$ is the affine hypersurface defined by $f(x_1,\dots,x_n,y)=0$. 
Let $A=k[x_1,\dots,x_n,y]/(f)$ be the coordinate ring of $Z_0$.
If $Z_0$ is smooth then we are done. Suppose $Z_0$ is not smooth. Then there exists $h\in J\cap k[x_1,\dots,x_n]$, where $J\subset k[x_1,\dots,x_n,y]$ is the ideal defining the singular locus of $Z_0$.

The rest of the proof follows the spirit of \cite[Theorem 5.7]{BMS}. Expand $f=\sum_{i=0}^dc_iy^i$, where $c_i\in k[x_1,\dots,x_n]$, and $c_0\neq 0$. Via the change of variable $(hc_0^2)y'=y$ we get $f=\sum_{i=0}^d c_i(hc_0^2)^i(y')^i$. Dividing by $c_0$, we define $f'=1+\sum_{i=1}^d c_ic_0^{i-1}(hc_0y')^i$ and $Z_0'=\Spec k[x_1,\dots,x_n,y']/(f')$. Then via the homomorphism $y\mapsto (hc_0^2)y'$ we see that $Z_0'$ is isomorphic to the smooth dense open subvariety $D(hc_0)$ of $Z_0$. Thus $Z_0'$ is a smooth affine hypersurface in $\A_k^{n+1}$ birational to $Z$.
\endproof
\end{thm}

\begin{lem}\label{lem:birationalequalinclusion}
If a smooth affine curve $C_0/k$ is birational to a smooth projective curve $C/k$, then $C$ is isomorphic to the smooth completion of $C_0$. Equivalently, there exists an open immersion with dense image $C_0\hookrightarrow C$.
\proof
Since $C$ is complete and $C_0$ is smooth and separated (as affine), the birational map $C_0\--> C$ uniquely extends to a separated birational morphism $\iota: C_0\rightarrow C$. Denoting by $\tilde C$ the smooth completion of $C_0$ note that $\iota$ decomposes into the canonical open immersion $C_0\hookrightarrow \tilde C$ and the morphism $\tilde\iota:\tilde C\rightarrow C$ extending the rational map given by $\iota$. Therefore it suffices to prove that $\tilde\iota$ is an isomorphism.

First note that $\tilde \iota$ is proper by \cite[Proposition 3.3.16(e)]{Liu} since $\tilde C$ and $C$ are complete. Furthermore, both $\tilde C$ and $C$ are smooth, so they are geometrically reduced and have irreducible connected components. For any connected component $\tilde U$ of $\tilde C$ there is a connected component $U$ of $C$ such that $\tilde \iota$ restricts to a morphism $\iota_U:\tilde U\rightarrow U$. Note that $\iota_U$ is a proper birational morphism, as $\tilde U$ is a closed subscheme of $\tilde C$ and $\tilde\iota$ is proper birational. Since both $\tilde U$ and $U$ are integral and smooth of dimension $1$, and so normal, \cite[Corollary 4.4.3(b)]{Liu} implies that $\iota_{U}:\tilde U\rightarrow U$ is an isomorphism. It follows that $\tilde \iota: \tilde C\rightarrow C$ is an isomorphism.
\endproof
\end{lem}

\begin{cor}\label{cor:smoothopenmodel}
Every smooth projective curve $C/k$ has a dense affine open which is isomorphic to a smooth plane curve.
\proof
From Theorem \ref{thm:Openofvariety} there exists a smooth affine plane curve $C_0$ birational to $C$. Then Lemma \ref{lem:birationalequalinclusion} concludes the proof.
\endproof
\end{cor}

%% file: AppendixB.tex
\section{Existence of a Baker's model}\label{appendix:ExistenceBakerModel}
Let $k$ be a perfect field. We say that a curve $C/k$ is \textit{nice} if it is geometrically connected, smooth and projective over $k$.
In this appendix we slightly extend some results in \cite{CV1, CV2} for studying the existence of a Baker's model of a nice curve.
Define the \textit{index} of a nice curve $C/k$ to be the smallest extension degree of a field $K/k$ such that $C(K)\neq \varnothing$.

\begin{lem}
Let $C$ be a nice curve of genus $1$. Then $C$ admits a Baker's model if and only if $C$ has index at most $3$.
\proof
Suppose $C$ has index at most $3$. Then by \cite[Lemma 4.1]{CV1} the curve $C$ is nondegenerate. Hence $C$ has an outer regular Baker's model.

Suppose now that $C$ admits a Baker's model. Then there exists a smooth curve $C_0\hookrightarrow C$ defined in $\G_{m,k}^2$ by $f\in k[x^{\pm 1},y^{\pm 1}]$ such that the completion $C_1$ of $C_0$ with respect to the Newton polygon $\Delta$ of $f$ is regular. We follow the spirit of the proof of \cite[Lemma 4.1]{CV1}. Since the arithmetic genus of $C$ is $1$ there is exactly $1$ interior integer point of $\Delta$. There are $16$ equivalence classes of integral polytopes with this condition (see \cite[Appendix]{CV1}). Then without loss of generality we can assume $\Delta$ is in this list. Note that there is an edge $\ell\subseteq\partial\Delta$ such that $\#(\ell\cap\Z^2)\leq 4$. Let $v$ be the normal vector of $\ell$ and $\alpha=(v,())\in\Sigma_1$. Then $f|_{\alpha}$ has at most $3$ roots in $\bar k^\times$ by Proposition \ref{prop:reductionfalphabasecase}. Therefore the splitting field $K$ of $f|_{\alpha}$ has degree $\leq 3$ over $k$. Furthermore, by definition $C_1$ has at least one point defined over $K$ visible on $C_\alpha$. Thus $C_1$, and so $C$, has index at most $3$.
\endproof
\end{lem}

\begin{rem}
The lemma above implies that there are nice curves which does not have a Baker's model. Indeed, if $k$ is a number field, \cite{Cla} proves there exist nice curves of genus $1$ of any index.
\end{rem}

\begin{thm}
Let $C$ be a nice curve of genus $g\leq 3$. If $k$ is finite or $C(k)\neq\varnothing$ then $C$ admits a Baker's model.
\proof
The first theorem in \cite{CV1} and \cite[Proposition 3.2]{CV2} show $C$ is nondegenerate except when $C$ is birational to a curve $C_0$ given in $\G_{m,k}^2$ by 
\begin{align*}
    f^{(2)}&=(x+y)^4+(xy)^2+xy(x+y+1)+(x+y+1)^2,& &\text{with }k=\F_2,\text{ or}\\
    f^{(3)}&=(x^2+1)^2+y-y^3, & &\text{with }k=\F_3.
\end{align*}
Recall that if $C$ is nondegenerate then it has an outer regular Baker's model. Therefore it suffices to show that in the two exceptional cases above the completion $C_1$ of the curve $C_0$ with respect to its Newton polygon is smooth.
We use the notation of \S\ref{subsec:IntroductionTheorem}.

Suppose $k=\F_2$ and $C_0: f^{(2)}=0$ over $\G_{m,\F_2}^2$. Note that $C_0$ is smooth. Denote $f=f^{(2)}$. The Newton polygon $\Delta$ of $f$ is
\[\begin{tikzpicture}[scale=0.4]
    \draw[->] (-0.6,0) -- (4.6,0) node[right] {$x$};
    \draw[->] (0,-0.6) -- (0,4.6) node[above] {$y$};
    \tkzDefPoint(4,0){A}
    \tkzDefPoint(0,4){B}
    \tkzDefPoint(2,2){C}
    \tkzDefPoint(2,1){D}
    \tkzDefPoint(1,2){E}
    \tkzDefPoint(1,1){F}
    \tkzDefPoint(2,0){G}
    \tkzDefPoint(0,2){H}
    \tkzDefPoint(0,0){O}
    \tkzLabelPoint[below](A){$(4,0)$}
    \tkzLabelPoint[left](B){$(0,4)$}
    \foreach \n in {O,A,B,C,D,E,F,G,H}
    \node at (\n)[circle,fill,inner sep=1.5pt]{};
    \draw (A) -- (O) node [midway,below, fill=none] {$\ell_1$};     
    \draw (O) -- (B) node [midway,left, fill=none] {$\ell_2$};
    \draw (B) -- (A) node [midway,right, above, fill=none] {$\ell_3$};
\end{tikzpicture}\]
where the normal vectors of the edges $\ell_1$, $\ell_2$, $\ell_3$ of $\Delta$ are respectively $\beta_1=(0,1)$, $\beta_2=(1,0)$, $\beta_3=(-1,-1)$. Then by fixing $\delta_{\beta_1}=(1,0)$, $\delta_{\beta_2}=(-1,-1)$, $\delta_{\beta_3}=(0,1)$ we have
\[
    f_{\ell_i}(X,Y)=(X^2+X+1)^2+X(X+1)Y+(X^2+X+1)Y^2+Y^4,
\]
for every $i=1,2,3$. Note that the points on $Y=0$ are regular points of $C_{\ell_i}$. Thus $C_{\ell}$ is smooth for any edge $\ell$ of $\Delta$ and so $C_1$ is smooth.

Suppose $k=\F_3$ and $C_0: f^{(3)}=0$ over $\G_{m,\F_3}^2$. Note that $C_0$ is smooth. Denote $f=f^{(3)}$. The Newton polygon $\Delta$ of $f$ is
\[\begin{tikzpicture}[scale=0.4]
    \draw[->] (-0.6,0) -- (4.6,0) node[right] {$x$};
    \draw[->] (0,-0.6) -- (0,3.6) node[above] {$y$};
    \tkzDefPoint(4,0){A}
    \tkzDefPoint(0,3){B}
    \tkzDefPoint(2,0){C}
    \tkzDefPoint(0,1){D}
    \tkzDefPoint(0,0){O}
    \tkzLabelPoint[below](A){$(4,0)$}
    \tkzLabelPoint[left](B){$(0,4)$}
    \foreach \n in {O,A,B,C,D}
    \node at (\n)[circle,fill,inner sep=1.5pt]{};
    \draw (A) -- (O) node [midway,below, fill=none] {$\ell_1$};     
    \draw (O) -- (B) node [midway,left, fill=none] {$\ell_2$};
    \draw (B) -- (A) node [midway,right, above, fill=none] {$\ell_3$};
\end{tikzpicture}\]
where the normal vectors of the edges $\ell_1$, $\ell_2$, $\ell_3$ of $\Delta$ are respectively $\beta_1=(0,1)$, $\beta_2=(1,0)$, $\beta_3=(-3,-4)$. We can choose $\delta_{\beta_1}=(1,0)$ so that
\[
    f_{\ell_1}(X,Y)=(X^2+1)^2+Y-Y^3.
\]
The points on $Y=0$ are regular points of $C_{\ell_1}$ and so $C_{\ell_1}$ is smooth. Furthermore, up to a power of $X$ the polynomials $f|_{\ell_2}$ and $f|_{\ell_3}$ equal $X^3+X^2-1$ and $-X+1$ respectively. It follows that the charts $C_{\ell_2}$ and $C_{\ell_3}$ of $C_1$ are regular. Thus $C_1$ is smooth.
\endproof
\end{thm}